\documentclass[10pt]{amsart}

\usepackage{amsmath, amstext, amsthm, amssymb, amsfonts, latexsym}
\usepackage{xcolor}
\usepackage[colorlinks=true]{hyperref}
\usepackage{tikz}
\usepackage{graphicx}
\usetikzlibrary{decorations.pathreplacing, patterns}
\usepackage{enumerate}
\usepackage{url}

\definecolor{wildstrawberry}{rgb}{0.5, 0.3, 1}

\numberwithin{equation}{section}

\newcommand{\T}{\mathbb{T}}
\newcommand{\R}{\mathbb{R}}
\newcommand{\Z}{\mathbb{Z}}
\newcommand{\Q}{\mathbb{Q}}
\newcommand{\N}{\mathbb{N}}

\newcommand{\TT}{\mathbb{T}^2}
\newcommand{\RR}{\mathbb{R}^2}
\newcommand{\ZZ}{\mathbb{Z}^2}
\newcommand{\cO}{\mathcal{O}}

\newcommand{\M}{\mathcal{M}}

\newcommand{\cR}{\mathcal{R}}
\newcommand{\hx}{\hat{x}}

\newcommand{\cC}{\mathcal{C}}

\newcommand{\cL}{\mathcal{L}}
\newcommand{\cA}{\mathcal{A}}
\newcommand{\cB}{\mathcal{B}}
\newcommand{\cK}{\mathcal{K}}
\newcommand{\cD}{\mathcal{D}}
\newcommand{\cF}{\mathcal{F}}
\newcommand{\cW}{\mathcal{W}}
\newcommand{\half}{\frac{1}{2}}
\newcommand{\ab}{\frac{a}{b}}

\newcommand{\B}{\operatorname{Box}}

\newcommand{\p}{\mathbf{p}}
\newcommand{\q}{\mathbf{q}}
\newcommand{\x}{\mathbf{x}}

\newcommand{\zero}{\mathbf{0}}
\newcommand{\la}{\boldsymbol{\lambda}}
\newcommand{\lala}{\la^2}
\newcommand{\deltap}{\delta_{\mathbf{p}}}
\newcommand{\deltaq}{\delta_{\mathbf{q}}}

\newcommand{\card}{\operatorname{card}}

\newtheorem{theo}{Theorem}

\newtheorem{theorem}{Theorem}[section]

\newtheorem{question}[theorem]{Question}
\newtheorem{proposition}[theorem]{Proposition}
\newtheorem{lemma}[theorem]{Lemma}
\newtheorem{corollary}[theorem]{Corollary}

\theoremstyle{definition}
\newtheorem{definition}[theorem]{Definition}
\newtheorem{remark}[theorem]{Remark}
\newtheorem*{hypotheses*}{Standing hypotheses (SH)}

\title[Irrational flows with multiple stopping points]{Historic behaviour vs. physical measures for irrational flows with multiple stopping points}


\author{Martin Andersson, Pierre-Antoine Guih{\'e}neuf}

\address{Martin Andersson: Universidade Federal Fluminense, Departamento de Matem\`atica Aplicada, Rua Professor Marcos Waldemar de Freitas Reis, s/n, 24210-201, Niter\'oi, Brazil}
\email{nilsmartin@id.uff.br}

\address{Pierre-Antoine Guih{\'e}neuf: Sorbonne Universit{\'e} and Universit{\'e} de Paris, CNRS, IMJ-PRG, F-75006 Paris, France}
\email{pierre-antoine.guiheneuf@imj-prg.fr}

\thanks{We warmly thank Corinna Ulcigrai and Damien Thomine for their kind remarks about the first version of this paper.}

\date{\today}

\begin{document}

\begin{abstract}
We study Birkhoff averages along trajectories of smooth reparameterizations of irrational linear flows of the two torus with two stopping points, say $\p$ and $\q$, of quadratic order. The limiting behaviour of such averages is independent of the starting point in a set of full Haar-Lebesgue measure and depends in an intricate way on the Diophantine properties of both the slope $\alpha$ of the linear flow as well as the relative position of $\p$ and $\q$. In particular, if $\alpha$ is Diophantine, then Birkhoff limits diverge almost everywhere (historic behaviour) and if $\alpha$ is sufficiently Liouville, then there exists some $\p$ and $\q$ such that the Birkhoff averages converge almost everywhere (unique physical measure).
\end{abstract}

\maketitle


\setcounter{tocdepth}{1}
\tableofcontents

\section{Introduction}

One of Rufus Bowen's many contributions to smooth ergodic theory is a construction often referred to as Bowen's eye. It is an example of a flow $\phi^t$ on $\mathbb{R}^2$ with two hyperbolic fixed points $\p$ and $\q$ such that one branch of the stable manifold of $\p$ coincides with one branch of the unstable manifold of $\q$ and vice versa, forming an eye-shaped region between two separatrices. In the interior of this region, there is a repelling fixed point with  complex eigenvalues, inducing a spiralling behaviour towards the boundary of the eye. One can show that, if the eigenvalues of $D\phi^1(\p)$ and $D\phi^1(\q)$ are chosen appropriately, then the Birkhoff averages 
\[\frac{1}{t} \int_0^t f(\phi^s(\x)) \ ds \]
diverge as $t \to \infty$ for every $\x$ inside the eye except for the fixed point, whenever $f: \mathbb{R}^2 \to \mathbb{R}$ is a continuous function taking different values at $\p$ and $\q$ (see e.g. Takens \cite{MR1274765}).

Bowen's eye is the best known example of what is now called \emph{historic behaviour}: the existence of a positive Lebesgue measure set of initial points for which Birkhoff averages are divergent. Historic behaviour has been found in many different contexts. Herman gave a very simple example of a system exhibiting such a behaviour \cite{herman}.
Hofbauer and Keller \cite{hofbauer1990} proved that there are uncountably many quadratic maps with almost everywhere divergent Birkhoff averages.  
Kiriki and Soma \cite{KIRIKI2017524} have showed that for $r \geq 2$, there is a $C^r$ open set $\mathcal{N}$ of surface diffeomorphisms, and a dense subset $D \subset \mathcal{N}$, such that every $f \in D$ has a wandering domain, giving rise to historic behaviour\footnote{Very recently, Berger and Biebler have anounced in \cite{berger2020emergence} a similar result for the classes $C^\infty$ and $C^\omega$.}. The set $\mathcal{N}$ is obtained by Newhouse's construction of persistent tangencies and contains the H{\'e}non family. Their results have been generalized to higher dimension by Barrientos \cite{2103.11964} and were adapted to flows by Labouriau and Rodrigues \cite{Labouriau_2017}.  Saburov \cite{MR4185284} recently found that historic behaviour is abundant among predator-pray dynamics. A transitive partially hyperbolic diffeomorphism on $\mathbb{T}^3$ displaying historic behaviour is given in a recent work of Crovisier, Yang, and Zhang \cite{MR4082180}. 
In a recent work \cite{2003.02185}, Talebi studies historic behaviour in the setting of rational maps of the Riemann sphere. He has announced that the set of maps with historic behaviour contains a dense $G_\delta$ set in the closure of  strictly post-critically finite maps, i.e. maps for which all critical points lie in the pre-orbit of a periodic repeller. In the continuous setting, Abdenur and the first author \cite{Abdenur2013} showed that in the $C^0$ conjugacy class of expanding circle maps, there is a dense $G_\delta$ set of maps with historic behaviour.

Although historic behaviour is abundant in some special families of dynamical systems, it is generally believed that it cannot be persistent among smooth maps or flows without any extra structure. However, Ruelle has expressed some hope that such examples may exist \cite{MR1858471} and Takens  emphasized it as an important problem in \cite{MR2396607}, whence it has become known as \emph{Taken's last problem.}

The present work deals with Taken's last problem for the special family of repara\-metrized linear flows on the torus with two stopping points. In this setting, one must either have historic behaviour or a unique physical measure whose basin has full measure (a dichotomy which does not hold for other dynamical systems). We show that both possibilities occur, depending on the angle of the flow and the relative position of the stopping points, but historic behaviour is the more abundant phenomenon, both from a topological and a measure theoretic point of view. 

As far as we know, irrational flows with two stopping points have not been studied before. The case of flows with one stopping point, however, has been extensively studied, from the grounding work of Ko\v{c}ergin \cite{MR0516507} to finer results about the mixing rate (e.g. the recent result of  Fayad, Forni, and Kanigowski \cite{fayad2019lebesgue}). See the survey of Dolgopyat and Fayad \cite{MR3309100}; in particular the results of the present article  are based on studies of Birkhoff sums that give partial answers to Question 41 (this question was already tackled by Sina\u{\i} and Ulcigrai in \cite{MR2478478}).

For its part, the study of physical measures for flows on surfaces is a bit more developped, with among others Katok example (e.g. Kwapisz \cite{MR2351022}), the examples of Saghin, Sun and Vargas \cite{MR2670926} and the special attention paid to Cherry flows by Palmisano \cite{palmisano2014physical}, Saghin and Vargas \cite{Saghin_2012} and Yang \cite{YANG_2016}.

\begin{remark}
The definition of historic behaviour varies in the litterature. Sometimes it is often defined pointwise, so that a point is said to have historic behaviour whenever Birkhoff averages fail to converge for some continuous observable. There has been a recent surge in research about systems for which such historic behaviour occurs on a residual (dense $G_\delta$) set of points \cite{MR4212116, MR4055947, 2107.01200, MR3567830, 2107.12498}.
\end{remark}

\subsection*{Formulation of the problem and  summary of results}

Consider a constant vector field $X_0 = (1,\alpha)$ on $\TT= \RR / \ZZ$. We shall always assume that $\alpha$ is irrational (otherwise any reparametrization of the flow is periodic, thus has extremely simple  ergodic behaviour). Let $\varphi: \TT \to \R$ be a non-negative smooth function that vanishes at exactly two points, $\p, \q$ say. We assume that $\varphi$ is of quadratic order at these points -- by that we mean that its derivative $D\varphi$ vanishes and that the Hessian $D^2 \varphi$ is positive definite at both $\p$ and $\q$. This assumption is quite natural, being the case of lowest codimension\footnote{If the map $\varphi$ is not smooth at $(0,0)$, then the behaviour of the reparametrized flow can be quite different, see e.g. Kwapisz and Mathinson \cite{MR2947933}.}. Let $X = \varphi X_0$ and let $\phi^t$ be the corresponding flow on $\TT$, which will be referred to as the \emph{reparametrized linear flow}. Such flows are topologically mixing but have zero entropy (see the introduction of Kaginowski \cite{MR3819702} for more informations).

Since we assume $\alpha$ to be irrational, the stable sets of $\p$ and $\q$ are densely immersed semi-lines, consisting of those initial points $\x$ for which $\phi^t(\x)$ approaches $\p$ or $\q$ as $t$ tends to infinity. All points on none of these stable sets (in particular, points on a set of full Haar measure on $\TT$) have a dense future orbit under $\phi^t$. The question arises as to what can be said about the time averages of such points. As we shall see in Proposition~\ref{invariant probs better}, the flow $\phi^t$ has no invariant probabilities other than the point masses at $\p$ and $\q$ as well as their convex combinations. 

Let $\M$ be the set of Borel probabilities on $\TT$; endowed with the weak-* topology this set becomes compact. 
The flow $\phi^t$ induces at every point $\x \in \TT$ a family $\{\mu_\x^t \}_{t>0}$ of what we may call empirical measures,  given by 
\begin{equation}\label{EqDefPhi}
\int_{\T^2} f \ d\mu_\x^t = \frac{1}{t} \int_0^t  f(\phi^s(\x)) \  ds
\end{equation}
for every continuous $f: \T^2 \to \R$.	

Let $p\omega(\x)$ be the compact subset of $\M$ defined by
\[p\omega(\x)  = \bigcap_{T > 0} \overline{ \{\mu_\x^t: t \geq T \} }. \]

We denote by $\la$ the Haar measure (also referred to as Lebesgue measure) on $\T$ and by $\lala = \la \times \la$ the Haar measure (also referred to as Lebesgue measure) on $\TT$. 
Although $\lala$ is not invariant under the flow $\phi^t$, it is still ergodic in the sense that if $A$ is a Borel measurable set such that $\phi^t(A) = A$ for every $t \in \mathbb{R}$, then $\lala(A)$ is either $0$ or $1$ (because $\phi^t$ is almost everywhere orbit equivalent to the linear flow $\phi_0^t$ associated to the vector field $X_0$). As a consequence, $p\omega(\x)$ is $\lala$-almost everywhere constant (see Proposition~\ref{invariant probs better}). More precisely, let 
\begin{equation}\label{EqFormPhys}
\mu_\infty := \frac{\sqrt{d_\q}}{\sqrt{d_\p} + \sqrt{d_\p}} \delta_\p\, + \, \frac{\sqrt{d_\p}}{\sqrt{d_\p} + \sqrt{d_\q}} \delta_\q,
\end{equation}
where $d_\p$ and $d_\q$ are the determinants of the Hessian at $\p$ and $\q$ and $\delta_\p$, $\delta_\q$ the point mass measures at these points. We have the following dichotomy: Given a triple $\p, \q, \alpha$ and a reparametrerized linear flow $\phi^t$ with stopping points at $\p$ and $\q$, then (see Proposition~\ref{PropPossibOmega})
\begin{enumerate}[(i)]
\item either Birkhoff averages $\mu_\x^t$ are $\lala$-almost everywhere divergent (i.e. $\operatorname{card} p\omega(\x)\ge 2$ a.e.);
\item or $\mu_\x^t$ converges to $\mu_\infty$ for $\lala$-almost every $\x \in \TT$.
\end{enumerate}
In case (i) we say that $\phi^t$ has \emph{historic behaviour}, and in case (ii) we say that $\phi^t$ has a \emph{physical measure}. If, in case (i), it so happens that 
\[p\omega(\x) = \{ \alpha \delta_\p + (1-\alpha) \delta_\q: \ 0 \leq \alpha \leq 1 \} \] 
for $\lala$-almost every $\x \in \TT$, then we say that $\phi^t$ has a \emph{extreme historic behaviour}.

\begin{remark}
More generally,  a measure $\mu$ is called \emph{physical} if its basin $B(\mu) = \{\x: p\omega(\x) = \{\mu \} \}$ has positive $\lala$-measure.  Some dynamical systems have several physical measures. In our setting, whenever we have a physical measure, it is unique, and its basin has full $\lala$-measure.

A curious feature of irrational flows of category (ii) is that they give rise to physical measures that are not ergodic. This is a rare phenomenon for transitive systems which, to our knowledge, has only been found once before, in Saghin, Sun and Vargas \cite{MR2670926} (see Mu\~{n}oz, Navas, Pujals, and V\'{a}squez \cite{MR2373211} for a non transitive interesting example).
\end{remark}

The question arises as to what choices of $\alpha$ and stopping points $\p$, $\q$ give rise to historic behaviour, and what choices result in a physical measure.

\begin{hypotheses*} \label{SH}
Throughout this work, a reparameterized linear flow with angle $\alpha$ and  stopping points at $\p$ and $\q$ always refers to a flow $\phi^t$  generated by a vector field $X=\varphi X_0$ with the following properties:
\begin{enumerate}
\item $X_0 = (1,\alpha)$ and $\alpha$ is irrational;
\item $\varphi: \TT \to \R_+$ is of class  $C^3$;
\item $\varphi$ vanishes at two distinct points $\p$ and $\q$, and is  positive elsewhere (hence $D \varphi (\p) = D \varphi(\q) = 0$);
\item the Hessians $D^2 \varphi (\p)$ and $D^2 \varphi (\q)$ are positive definite matrices. 
\end{enumerate}
\end{hypotheses*}
\medskip

In this paper we get a quite complete set of criteria, expressed in terms of Diophantine approximation properties, under which a reparametrized flow satisfying (SH) has historic behaviour or a unique physical measure.

Our first result says that in most cases (from both a measure-theoretic and topological viewpoint), the reparametrized flow satisfying (SH) has historic behaviour. Still, there are nontrivial cases where there is a unique physical measure. The following statement is a combination of Theorems \ref{generic distinct orbits}, \ref{PropDivSum} and \ref{physmeas different orbits}.

\begin{theo}\label{TheoIntro1}
There are subsets  $\mathcal{F}, \cR, \cD \subset \R \times \TT \times \TT$ with the following characteristics: $\mathcal{F}$ is of full Lebesgue measure, $\cR$ is a dense $G_\delta$ set, and  $\cD$ is dense (but not $G_\delta$), such that for any $(\alpha, \p, \q) \in \T \times\T^2\times\T^2$ and any $\phi^t$ be a reparameterized linear flow satisfying (SH) with angle $\alpha$ and stopping points at $\p$ and $\q$. Then $\phi^t$ has
\begin{itemize}
\item a unique physical measure if $(\alpha, \p, \q) \in \cD$,
\item historic behaviour if $(\alpha, \p, \q) \in \mathcal{F}$, and
\item an extreme historic behaviour if $(\alpha, \p, \q) \in \cR$.
\end{itemize}
In the first case, the physical measure is given by (\ref{EqFormPhys}).
\end{theo}

All the sets of angles (i.e. the projections of $\cF$, $\cR$ and $\cD$ to $\R$) in Theorem~\ref{TheoIntro1} are explicit in the sense that they arise from conditions involving quantities that are determined by the expression of the angle as a continued fraction.

Let us be a little more precise: The set $\mathcal{F}$ can be written as  $A\times \T^2\times \T^2$, with $A\subset\T$ of full measure. We also prove that that the first case (i.e. $(\alpha, \p, \q) \in \cD$) occurs in two ways, namely with $\p$ and $\q$ in the same orbit (of the non-reparameterized flow) as well as $\p$ and $\q$ in different orbits. Both situations occur densely in the phase space $\R \times \TT \times \TT$, and there exists a full Hausdorff dimensional set $B\subset \R$ such that for any $\alpha\in B$, there exists $\p,\q\in\T^2$ such that $(\alpha,\p,\q)\in\cD$.

We give special attention to extreme historic behaviour in the case where $\p=(0,0)$ and $\q=(0,\beta)$ for some rational $\beta$ (see Theorem~\ref{refined rational distinct orbits} for a refined version). 

\begin{theo}\label{rational distinct orbits}
There exists a dense $G_\delta$ set $A \subset \R$ such that if $\alpha \in A$, $\beta \in \Q \setminus \{0\}$, and if $\phi^t$ is a reparametrized linear flow satisfying (SH) with angle $\alpha$ with stopping points at $(0,0)$ and $(0, \beta)$, then $\phi^t$ has an extreme historic behaviour.
\end{theo}

Finally, we describe more specifically what happens in the case where the singularities $\p$ and $\q$ lie on the same orbit. The results are described by this combination of Theorems \ref{PropConv} and \ref{PropDivSum}.

\begin{theo}\label{TheoIntro2}
There exists a full measure set $\cA \subset \R$, and a full Hausdorff dimensional dense set $\cB \subset \T$ such that if $\alpha\in \R$ and $\q = \p + r(1, \alpha) \mod \ZZ$ for some $r>0$, then for any reparameterized linear flow $\phi^t$  satisfying (SH),
\begin{itemize}
\item $p\omega(\x) = [\mu_\infty,  \delta_\p]$ $\x$-a.e. if $\alpha\in \cA$ (historic behaviour);
\item $p\omega(\x) = \{\mu_\infty\}$ $\x$-a.e. if $\alpha\in \cB$, $\x$-a.e. (physical measure).
\end{itemize}
\end{theo}

(Here $[\mu_\infty, \delta_\p]$ denotes the set $\{\alpha \mu_\infty + (1-\alpha) \delta_\p : \ 0 \leq \alpha \leq 1 \}$ of convex combinations of $\mu_\infty$ and $\delta_\p$.) As before, the sets $\cA$ and $\cB$ are given by explicit Diophantine conditions. Some simulations relative to this theorem can be found in Figure~\ref{simul2}.

\begin{figure}\label{simul2}
\noindent
\includegraphics[width=.45\linewidth]{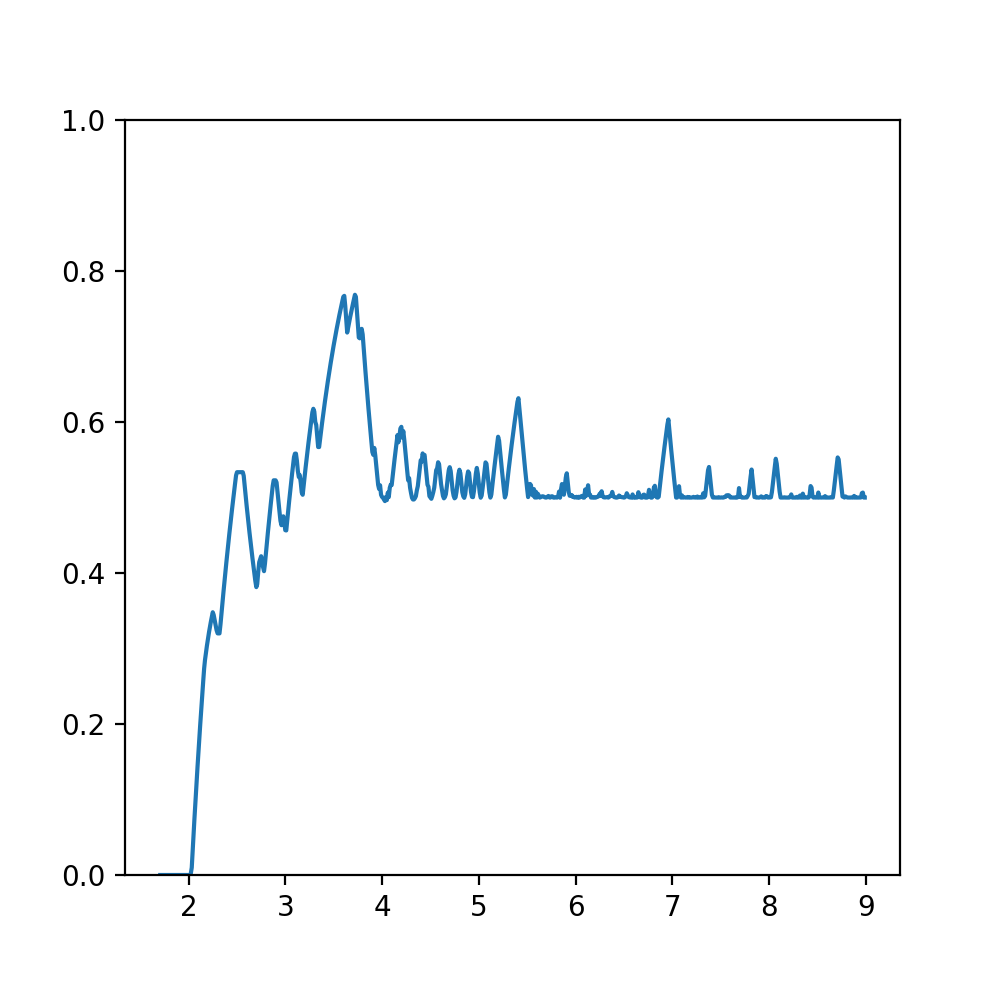}\hfill 
\includegraphics[width=.45\linewidth]{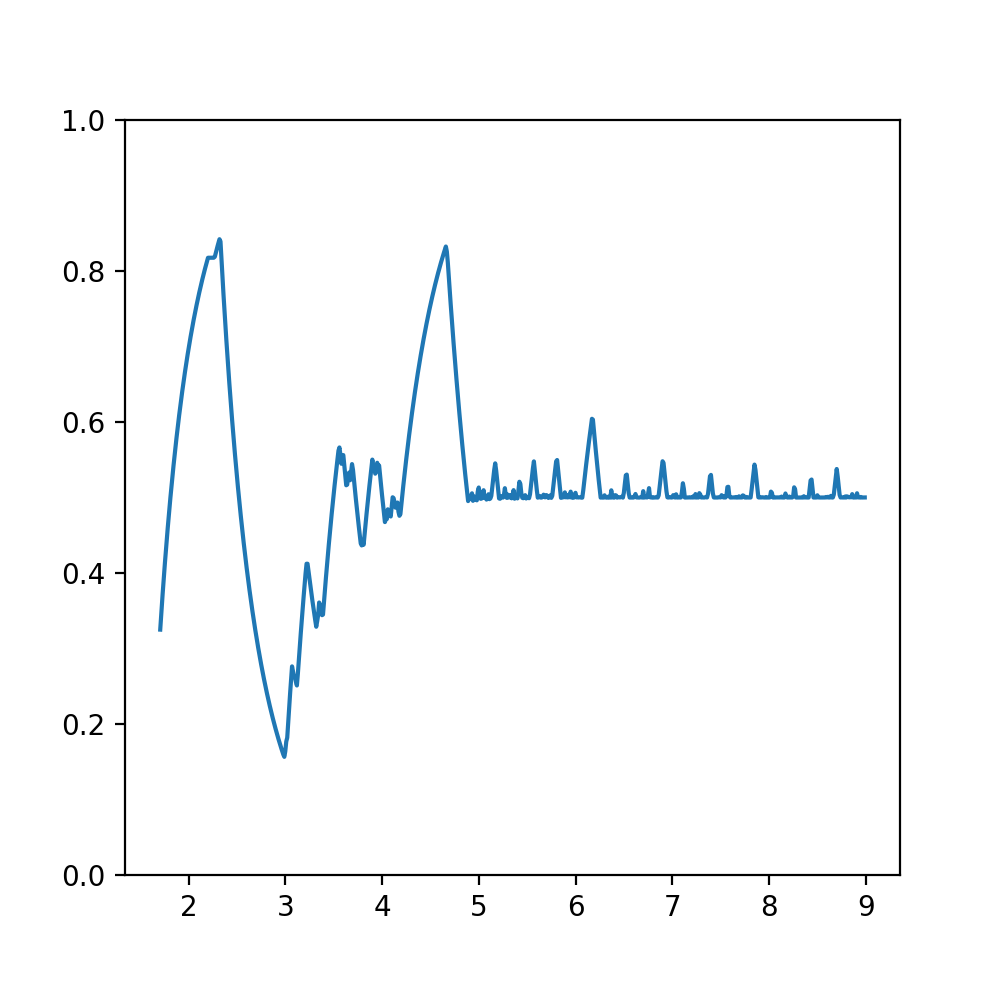}
\caption{Simulations of the proportion of time spent by the flow $\phi^t$ in some fixed small neighbourhood of $\p$ depending on $\log_{10}$ of the time (i.e. the right of the graphics with abscissa 9 corresponds to time $10^9$). More precisely, these are simulations of a single orbit starting at point $M=(0.6319874,\,0.3684641)$ of the map $M\mapsto M+\delta\varphi(M)(1,\alpha)$, with $\delta=0.1572348$ and $\alpha=4/13+2/135+1/26\,714+2/166\,267\,121$ (left) resp. $\alpha=\sqrt{2}-1$ (right), and $\varphi(M) = \min(\|M-\p\|_2, \|M-\q\|_2)$ for $\p=(0.25,0.75)$ and $\q=\phi^{8.357}(\p)$. It is not clear if our theorems' predictions can be observed here or not: as the $\alpha$ for the left graphic is ``Liouville-like'' (at least for the times considered in the simulations) and the right one is Diophantine (it is of bounded type), from Theorems \ref{PropConv} and \ref{PropDivSum}, the left graphic should eventually oscillate between two different values and the right one should converge to $1/2$.}
\end{figure}

Let us say a few words about the global strategy for proving these theorems. By considering a Poincar{\'e} section, we reduce the study to the one of Birkhoff sums $S_n(x)$ of points of $\T$ under the rotation $R_\alpha$ for the observable $\|x\|^{-1} = d(x,\Z)^{-1}$ (Proposition~\ref{criterium1}). More precisely, 
\begin{itemize}
\item If for almost any $x$, one has $\|R_\alpha^n(x)\|^{-1} = o(S_n(x))$, then the system has a unique physical measure;
\item If for almost any $x$, one has $\|R_\alpha^n(x)\|^{-1} \neq o( S_n(x))$, then the system has an historic behaviour.
\end{itemize}
Roughly speaking, one wants to decide whether the orbit of most of points $x$ eventually come very close to 0 or not (close enough to kill all the previous contributions made by $\|\cdot\|^{-1}$ to the Birkhoff sums).

\subsection*{Irrational flows with more than two stopping points}

Let us say a few words about the case of more than 2 stopping points by pointing out some direct consequences of our theorems in the case of three stopping points $\p,\q$ and $\mathbf{r}$. 

By a trivial generalization of Proposition~\ref{criterium1} to the case of more than two stopping points, if a flow with angle $\alpha$ with stopping points $\p$ and $\q$ has a physical measure, and a flow with angle $\alpha$ with stopping points $\q$ and $\mathbf{r}$ has a physical measure, then a flow with angle $\alpha$ with stopping points $\p,\q$ and $\mathbf{r}$ also has a physical measure. This gives generalizations of Theorems~\ref{simple physical measure}, \ref{PropConv} and \ref{physmeas different orbits}, in particular the set of parameters of flows with $N$ stopping points contains a dense subset made of those with a unique physical measure.

Similarly, if a flow with angle $\alpha$ with stopping points $\p$ and $\q$ has a historic behaviour, and a flow with angle $\alpha$ with stopping points $\q$ and $\mathbf{r}$ has a historic behaviour, then a flow with angle $\alpha$ with stopping points $\p,\q$ and $\mathbf{r}$ also has a historic behaviour. This allows to generalize Theorems~\ref{refined rational distinct orbits}, \ref{generic distinct orbits} and \ref{PropDivSum} to get flows with multiple stopping points and historic behaviour.

However, the generalization of the notion of extreme historic behaviour is unclear. In the case of three stopping points, the set of invariant measures is the simplex spanned by $\delta_\p$, $\delta_\q$ and $\delta_{\mathbf r}$. An in-depth look at the proofs of Theorems~\ref{refined rational distinct orbits} and \ref{generic distinct orbits} would probably lead to the fact that on a full measure set of initial conditions $\x$, the segments $[\delta_\p,\delta_\q]$ and $[\delta_\q,\delta_{\mathbf r}]$ and $[\delta_{\mathbf r},\delta_\p]$ are included in $p\omega(\x)$, but this is only the boundary of the simplex. this leads to the following question.

\begin{question}
Consider an irrational flow with three stopping points with parameters $\alpha,\p,\q,\mathbf r$.
Is $p\omega(\x)$ equal to the whole simplex spanned by $\delta_\p$, $\delta_\q$ and $\delta_{\mathbf r}$ for a.e. $\x$ and a full measure set of parameters $(\alpha,\p,\q,\mathbf r)$? For a generic set of parameters $(\alpha,\p,\q,\mathbf r)$? If not, what is the dimension of $p\omega(\x)$?
\end{question}

Such a result would need much deeper techniques as the ones developed in the present paper, as we would probably have to determine the whole set of accumulation points of $\Theta_k^\beta(x)$ (see \eqref{DefTheta}) instead of just proving that it is large enough for a big set of points $\x$, moreover taking into account the interplay between the points $\p,\q$ and $\mathbf r$.

\begin{remark}
By taking a product of the time $t$ map of the flow flow $\phi^t$ for a small but non-zero $t$ with the Arnold cat map
\[A: (x,y) \mapsto (2x+y, x+y) \mod \mathbb{Z}^2\quad \text{ for } (x,y) \in \TT,\]
we obtain a partially hyperbolic diffeomorphism on $f = \phi^t \times A:\mathbb{T}^4 \to \mathbb{T}^4$. It is straightforward to see that the time $t$ map of an irrational flow with stopping points is topologically mixing, and that the product of two topologically mixing maps is itself topologically mixing. Thus any $f$ obtained in this way is topologically mixing. It turns out that maps of this form have rather unusual ergodic properties that are worthwhile pointing out.
 
Suppose, first, that $\phi^t$ has a unique physical measure. Then $f$ has a unique, non-ergodic physical measure. Moreover, the center Lyapunov exponents are zero for \emph{every} $x \in \mathbb{T}^4$ and $f$ is mixing. We do not know of any such example in the literature.

Next suppose that $\phi^t$ has historic behaviour. Then so has $f = \phi^t \times A$. To our knowledge, it is the first topologically mixing example of this kind. A transitive example on $\mathbb{T}^3$ was given by Crovisier, Yang, and Zhang in  \cite{MR4082180}.
\end{remark}

\subsection*{Outline of the paper}

In Section \ref{special flows}, we relate the ergodic behaviour of most of orbits of the reparametrized linear flow satisfying (SH) with some related special flow obtained as a suspension flow over a rotation. This allows us to get asymptotics of return times to a Poincar{\'e} section, in terms of the Hessian determinants at the stopping points.

This interpretation in terms of suspension flow allows us, in Section~\ref{SecInv}, to get estimates on return times in terms of Birkhoff sums for the non-integrable observable\footnote{Similar Birkhoff sums are studied by Sina\u{\i} and Ulcigrai \cite{MR2478478}, but with $x^{-1}$ instead of $\|x\|^{-1}$, which allows the authors to use cancellations between the positive and negative parts of the observable.} $\|\cdot\|^{-1}$. Using these estimates, we get an exact formula relying the set $p\omega(\x)$ of limit measures of $\x\in\T^2$ with the asymptotic behaviour of some quantity defined from Birkhoff sums (Proposition~\ref{criterium1}). Using some symmetry properties of this quantity, we then get criteria for historic behaviour/physical measure (Subsection~\ref{SubsecSym}).

Section~\ref{PartRotations} is devoted to some reminders about properties of circle rotations, their renormalizations, linked with the continued fraction of the angle.

Section~\ref{SecTech} is quite technical: we get some crucial bounds (from above and below) for Birkhoff sums for the observable $\|\cdot\|^{-1}$, using in particular comparison of the orbits with a rational rotation.

This section is used in the four last sections of the paper, each one of which being aimed to prove a part of Theorems~\ref{TheoIntro1}, \ref{rational distinct orbits} and \ref{TheoIntro2}. Note that the last section uses the proof strategy of the previous one (Section~\ref{SecPhysSame}).




\section{Special flows} \label{special flows}

\subsection{Definition and notations}\label{SecDefFlow}

Fix $\varphi$ and $X$ as described in the introduction, i.e. $X = \varphi X_0$ where $X_0 = (1,\alpha)$ is a constant vector field on $\TT$ and $\varphi: \TT \to \R$ is a non-negative smooth function that vanishes at exactly two points $\p$ and $\q$. Let $\phi_0^t$ and $\phi^t$ be the flows of $X_0$ and $X$ respectively. Fix some $x_0 \in \T$ such that $\p, \q \notin \Sigma \stackrel{\text{def.}}{=} \{x_0\} \times \T$ (with $\T = \R/\Z$). Then $\Sigma$ is a transverse section of the flow $\phi^t$. Let $(x_0,p_0)$ be the unique point on $\Sigma$ such that $\phi_0^t(x_0, p_0) = \p$ for some $t \in (0, 1)$; define $q_0$ analogously. In other words, $(x_0,p_0)$ is the only point of $\Sigma$ satisfying $\phi^t(x_0,p_0) \notin \Sigma$ for all $t>0$ and $\lim_{t \to \infty} \phi^t(x_0,p_0) = \p$. Likewise for $(x_0,q_0)$.

We say that $\p$ and $\q$ \emph{lie on the same orbit} if they belong to the same orbit of the flow $\phi_0^t$. This is equivalent to say that there is some point $\x$ such that $\lim_{t \to - \infty} \phi^t(\x) = \p$ and $\lim_{t \to \infty} \phi^t(\x) = \q$ or vice versa. Note that $\p$ and $\q$ lie on the same orbit if and only if $p_0$ and $q_0$ lie on the same orbit under the rotation
\begin{align*}
R_\alpha: \T & \to \T \\
            y & \mapsto y +\alpha \mod 1.
\end{align*} 

Note that we can (and do) always choose $x_0$ so that $p_0 \neq q_0$. Indeed if $\p$ and $\q$ are not on the same orbit, this is always the case. If $\p$ and $\q$ are on the same orbit, it suffices to choose $x_0$ so that $\Sigma$ intersects the orbit that joins $\p$ and $\q$ (indeed, $p_0=q_0$ implies that $\Sigma$ does not intersect the orbit between $\p$ and $\q$).

Let $Y = \T \setminus \{p_0, q_0 \}$.
We define a return time map $T: Y \to \R$ (also called \emph{roof function} in the sequel) by 
\begin{equation}\label{EqDefTau}
T(y) = \min\big\{t>0: \phi^t(x_0,y) \in \Sigma \big\}.
\end{equation}
Let 
\[D_0 = \big\{(u,t) \in Y\times\R: 0 \leq t \leq T(u)\big\} \]
and
\[D = D_0/\sim,\qquad \text{where}\ (u,T(u))\sim(R_\alpha(u),0).\]
This allows to define a map $\Xi:D \to \TT$ by 
\[\Xi(u,t) = \phi^t(x_0, u).\]

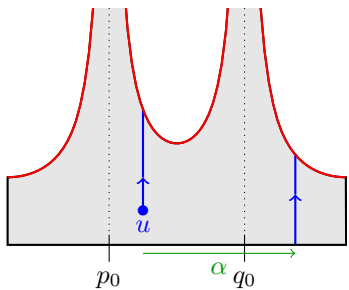
\begin{figure}
\begin{tikzpicture}[scale=4.5]
\clip(-.2,-.2) rectangle (1,.7);
\fill[color=black, opacity=.1] (0,0) -- plot[domain=0:1, scale=1, samples=80] (\x,{0.03*max(1/abs(\x-.3),1/abs(\x-1.3)) + 0.03*max(1/abs(\x-.7),1/abs(\x+.3))}) -- (1,0) -- cycle;
\draw[color=blue, thick, ->] (.4,.1) node{$\bullet$} node[below]{$u$} -- (.4,.2);
\draw[color=blue, thick] (.4,.2) -- (.4,.4);
\draw[color=blue, thick, ->] (.85,0) -- (.85,.15);
\draw[color=blue, thick] (.85,.15) -- (.85,.27);
\draw[thick] (0,0) -- plot[domain=0:1, scale=1, samples=80] (\x,{0.03*max(1/abs(\x-.3),1/abs(\x-1.3)) + 0.03*max(1/abs(\x-.7),1/abs(\x+.3))}) -- (1,0) -- cycle;
\draw[color=red, thick] plot[domain=0:1, scale=1, samples=80] (\x,{0.03*max(1/abs(\x-.3),1/abs(\x-1.3)) + 0.03*max(1/abs(\x-.7),1/abs(\x+.3))});
\draw[dotted] (.3,1) -- (.3,.03);
\draw (.3,.03) -- (.3,-.05) node[below]{$p_0$};
\draw[dotted] (.7,1) -- (.7,.03);
\draw (.7,.03) -- (.7,-.05) node[below]{$q_0$};

\draw[color=green!60!black,->] (.4,-.025) --node[midway, below]{$\alpha$} (.85,-.025);

\end{tikzpicture}
\caption{The flow $\Xi$: given a point $u\in D$, the flow $\Xi(u,t)$ (in blue) is the unitary vertical flow in $D$, given by the identifications $(v,T(v))\sim (R_\alpha(v),0)$.}
\end{figure}

Note that the image of $\Xi$ is $\TT$ minus two line segments going from the points $\p$ and $\q$ to $\Sigma$. The map $\Xi$ induces a family of measurable maps $\Psi^t = \Xi^{-1} \phi^t \Xi$ on $D$. The family $\Psi^t$ is called a \emph{special flow} on $D$ with \emph{base} $R_\alpha$ and \emph{roof function} $T$ (see \cite{Special}). Note that flow lines are vertical in restriction to the fundamental domain $D_0$.

We shall see in Section \ref{assymptotic} that, due to the quadratic order of the speed function $\varphi$, the roof function $T$ has two cusps of order $\| x-p_0\|^{-1}$ and $\| x-q_0 \|^{-1}$  in a neighbourhood of the two points $p_0$ and $q_0$ where it is undefined (Proposition~\ref{return times}). In particular, the roof function is not integrable.

\subsection{The asymptotic behaviour of return times} \label{assymptotic}

In this section we estimate the return times $T$ of points to the transverse section $\Sigma$ for the vector field $X$. Our goal is to prove that it only depends on the local behaviour of $\varphi$ around the singularities.

We start by estimating a quantity $\kappa$ similar to $T$, for a local quadratic model and a horizontal flow.

\begin{lemma}\label{bdd diff lemma}
Consider a horizontal vector field $X=(\varphi(x,y),0)$ in $\RR$, where $\varphi$ is a positive definite quadratic form $\varphi(x,y) = a x^2 + 2 b x y + c y^2$ with determinant $d= ac-b^2$. Denote by $\phi^t$ the flow associated to $X$. Fix some $\delta>0$ and, for $y \neq 0$, let $\kappa(y)$ be defined by 
\[\phi^{\kappa(y)}(-\delta, y) = (\delta, y).\]
Then 
\begin{equation}\label{bounded difference}
\kappa(y) = \frac{\pi}{\sqrt{d} |y|} + \gamma(y)
\end{equation}
for some bounded function $\gamma$.
\end{lemma}

\begin{proof}
The non-boundedness of $\kappa$ only occurs in the neighbourhood of $0$, so that one can reduce its study to that on a bounded set of $\R$.

Using the method of separation of variables, we see that 
\[\kappa(y) = \int_{-\delta}^{\delta} \frac{dx}{\varphi(x,y)}.\]
Note that $\kappa(y)$ differs form
\[\kappa_0(y) = \int_{(-\delta-by)/a}^{(\delta-by)/a} \frac{dx}{\varphi(x,y)} \]
by a bounded function, so it suffices to prove (\ref{bounded difference}) with $\kappa_0$ in place of $\kappa$. Moreover, by symmetry, it suffices to consider the case where $y>0$. A direct calculation gives
\begin{align*}
\kappa_0(y) & = \int_{(-\delta-by)/a}^{(\delta-by)/a}  \frac{dx}{a x^2 + 2bxy+cy^2} = \left[ \frac{1}{y \sqrt{d}} \arctan \left(\frac{ax+by}{y\sqrt{d}} \right) \right]_{(-\delta-by)/a}^{(\delta-by)/a} \\
& = \frac{2}{y \sqrt{d}} \arctan\left(\frac{\delta}{y \sqrt{d}}\right). 
\end{align*}
Recall that $\arctan(x) + \arctan(\frac{1}{x}) = \frac{\pi}{2}$ for every $x>0$. Therefore
\[\kappa_0(y) = \frac{\pi}{y \sqrt{d}} - \frac{2}{y \sqrt{d}} \arctan \left( \frac{y \sqrt{d}}{\delta} \right) ,\]
and the proof follows readily since the last term is bounded in $y$.
\end{proof}

\begin{lemma}\label{intble diff lemma}
Consider a horizontal vector field $X=(\varphi(x,y),0)$ in $Q=(-1,1)^2$, where $\varphi:Q \to \R$ is a non-negative $C^3$ function vanishing at $(0,0)$ and strictly positive elsewhere. Suppose that the Hessian $D^2 \varphi(0,0)$ is positive definite and write $d = \det(D^2 \varphi(0,0))$. Denote by $\phi^t$ the flow associated to $X$. Fix some $0< \delta< 1$ and, for $y \neq 0$, let $\kappa(y)$ be defined by 
\[\phi^{\kappa(y)}(-\delta, y) = (\delta, y).\]
Then 
\begin{equation} \label{intble difference}
\kappa(y) = \frac{\pi}{|y| \sqrt{d}} + \sigma(y)
\end{equation}
for some integrable function $\sigma$.
\end{lemma}

\begin{proof}
Let $A = D^2 \varphi (0,0)$ and let $0< \lambda_1 \leq \lambda_2$ be its eigenvalues. To simplify notation we write $(x,y)$ as $\x$. Recall that 
\[ \lambda_1 \| \x \|^2 \leq \x^T A \x \leq \lambda_2 \| \x \|^2 \]
for every $\x \in \RR$.

Since $\varphi$ is $C^3$ we can write $\varphi(\x) = \x^T A \x + R(\x)$, where 
\begin{equation}\label{taylor with rest}
|R(\x)| \leq K \| \x \|^3
\end{equation} for some $K>0$ in a neighbourhood of $\zero$.
Just like in the proof of Lemma~\ref{bdd diff lemma}, we have
\[ T(y) = \int_{-\delta}^{\delta} \frac{dx}{\varphi(\x)} = \int_{\delta}^{\delta} \frac{dx}{\x^T A \x + R(\x)}.\]
We know from Lemma~\ref{bdd diff lemma} that 
\[\int_{-\delta}^{\delta} \frac{dx}{\x^T A \x} \]
differs from $\pi/(|y| \sqrt{d})$ by a bounded function. Therefore, in order to prove Lemma~\ref{intble diff lemma}, it suffices to show that 
\[\sigma(y) = \int_{-\delta}^{\delta} \frac{1}{\x^T A \x}-\frac{1}{\x^T A \x + R(\x)} dx = \int_{-\delta}^\delta \frac{R(\x) dx}{\x^T A \x (\x^T A \x + R(\x))}\]
 is integrable. Note that $\sigma$ is continuous away from $y=0$, so it suffices to show that $\int_{-\delta}^{\delta} |\sigma(y)| dy < \infty$. Note also that changing the value of $\delta$ produces a change in $\sigma$ by a bounded amount. Thus, upon possibly reducing $\delta$, we can (and do) suppose that $K \| \x \|^3 \leq \frac{\lambda_1}{2} \| \x \|^2$  in $(-\delta, \delta)^2$. Consequently, for every $y \in (-\delta, \delta) \setminus \{0\}$, we have
\begin{align*}
|\sigma(y)| & \leq \int_{-\delta}^{\delta} \frac{|R(\x)| dx}{\x^T A \x (\x^T A \x - |R(\x)|)} \\
& \leq  \int_{-\delta}^{\delta} \frac{K \| \x \|^3 dx}{\lambda_1 \| \x \|^2 (\lambda_1 \|\x \|^2 - K \| \x \|^3)}  \\ 
& \leq \frac{2K}{\lambda_1^2}  \int_{\delta}^{\delta} \frac{dx}{\| \x \|}\\
& = \frac{2K}{\lambda_1^2} \int_{-\delta}^{\delta} \frac{dx}{\sqrt{x^2+y^2}} \\
& = \frac{4K}{\lambda_1^2} \left( \log \big( \delta+ \sqrt{\delta^2+y^2}\big)-\log |y| \right). 
\end{align*}
In particular $\int_{-\delta}^\delta |\sigma(y)| \ dy < \infty$. 
\end{proof}

Lemma \ref{intble diff lemma} tells us roughly how much a horizontal flow is slowed down near a stopping point at the origin. In order to apply it to the stopping points $\p, \q$ of an irrational flow of the torus, we need  to perform a change of coordinates. As we shall see, this change of coordinates is equivalent to changing the Hessian of $\varphi$ in a way that does not affect its determinant. Let's go through the details.

Suppose that we have fixed the angle $\alpha$. For small $\delta$, let $Q_\delta = (-\delta, \delta)^2 \subset \RR$ and consider the affine charts $\xi_{\p}, \xi_{\q} : Q_\delta \to \TT$ given by $\xi_{\p}(\x) = I(P\x) + \p$ and $\xi_{\q}= I(P \x) + \q$ where 
\[ P= \left( \begin{matrix}
1 & 0 \\
\alpha & 1
\end{matrix} \right),\] 
and $I: \RR \to \TT$ is the canonical projection. Let 
\[\B_\delta(\p) = \xi_\p (Q_\delta)\qquad \text{and}\qquad \B_\delta(\q) = \xi_\q (Q_\delta).\]
We will refer to $\B_\delta(\p)$ and $\B_\delta(\q)$ as \emph{flow boxes} around $\p$ and $\q$. 
 
For any $x\in\T = \R/\Z$, we set $\tilde x$ a lift of $x$ to $\R$ and define
\[ \|x\| = \min_{n\in\Z} |\tilde x-n|.\]

Let also
\[S_\p : (-\delta, \delta) \to (0,\infty]\]
be the time it takes for the flow to cross $\B_\delta(\p)$, defined by
\[S_\p(y) = \min \{ t>0: \phi^t(\p - \delta (1, \alpha)+(0,y)) \notin \B_\delta(\p)\}.\]

\begin{lemma} \label{time in a box}
Let $X = \varphi X_0$ be a reparameterized linear flow satisfying (SH) with a stopping point at $\p$ and $0 < \delta < 1$ small enough so that $X$ has no other stopping point in $\B_\delta(\p)$. 
Then $y\mapsto S_\p(y) - \frac{\pi}{\sqrt{d_\q} | y |}$ is integrable.
\end{lemma}

\begin{proof}
Let $Q= (-\delta, \delta)^2$ and $\xi_\p : Q \to \B_\delta(\p)$ be the affine chart described in the definition of $\B_\delta(\p)$. 
 Let $\tilde{X} = \xi_{\p}^* X$ be the pull-back of $X$ through $\xi_\p$ and denote by $\tilde{\phi}^t$ its associated flow on $Q$. Note that $\tilde{X}$ is a horizontal vector field and that the time it takes for  a trajectory of $\tilde{\phi}^t$ to cross $Q$ is the same as the time it takes for a trajectory of $\phi^t$ to cross $\B_\delta(\p)$. We claim that $\tilde{X}$ is of the form $\tilde{\varphi} e_1$, where $e_1$ is the unit vector $(1,0)$ and $\tilde{\varphi}$ satisfies
 \begin{equation}\label{equal Hessians}
 \det D^2 \tilde{\varphi} ( \zero) = \det D^2 \varphi(\p) = d_\p.
 \end{equation}
 Once this is shown, the proof follows from Lemma \ref{intble diff lemma}.

 To see why (\ref{equal Hessians}) holds, note that
\[\tilde{X} = \xi_{\p}^* X (\x) = D\xi_\p^{-1}(\xi_{\p}(\x)) X(\xi_\p(\x)) = P^{-1} X (\xi_\p(\x)) = \tilde{\varphi}_{\p} (\x) e_1,\]
where $\tilde{\varphi}_{\p} = \varphi \circ \xi_{\p}$, and therefore
\[D^2 \tilde{\varphi}_{\p}(\zero) = D\xi_{\p}(\zero) ^T D^2 \varphi (\p) D \xi_{\p}(\zero) = P^T D^2 \varphi(\p) P. \]
Since $\det P = 1$ we have $\det D^2 \tilde{\varphi}_{\p}(\zero) = \det D^2 \varphi ( \p) =d_\p$. 
\end{proof}

We are now able to provide a rather nice description of the return time of the flow $X$ to a transverse cross section.

\begin{proposition}\label{return times}
Let $X = \varphi X_0$ be a reparameterized linear flow satisfying (SH) with stopping points at $\p$ and $\q$, and $\Sigma = \{x_0\} \times \T$ a cross section not containing $\p$ nor $\q$. Let 
\begin{align*}
T: \T & \longrightarrow (0,\infty] \\
    y & \longmapsto \min\{ t>0: \phi^t(x_0,y) \in \Sigma \}
\end{align*}
be the return time of the flow to $\Sigma$. Denote by $d_\p$ and $d_\q$ the determinants of the Hessian of $\varphi$ at $\p$ and $\q$ respectively. Then
\[T(y) =  \frac{\pi }{\sqrt{d_\p} \|y-p_0\|} + \frac{\pi }{\sqrt{d_\q} \| y-q_0\|}  + \sigma(y) \]
for some integrable function $\sigma: \T \to \R$.
\end{proposition}

Recall that $p_0$ is the unique point in $\T$ such that $\p = (x_0, y) + r (1, \alpha)$ for some $0< r< 1$. Similarly for $q_0$. Note that Proposition \ref{return times} does not require $p_0$, $q_0$ do be distinct.

\begin{proof}
Fix $\delta>0$ so that the sets $\Sigma$, $\B_\delta(\p)$ and $\B_\delta(\q)$ are pairwise disjoint. Apply Lemma \ref{time in a box} and observe that the time spent by the orbit of $(x_0,y)$ outside $\B_\delta(\p) \cup \B_\delta(\q)$ before it  hits $\Sigma$ is a bounded function.
\end{proof}

Proposition \ref{return times} shows that the behaviour of time averages for the flow $\phi^t$ can be thought of as a problem of infinite ergodic theory. As we shall see in the next section, the behaviour of the 
time averages of the flow is determined by the behaviour of the quotient
\[ \sum_{k=0}^{n-1} \frac{1}{\|x+k\alpha - p_0 \|} \Big/ \sum_{k=0}^{n-1} \frac{1}{\|x+k\alpha - q_0 \|}, \]
for typical $x$, as $n \to \infty$.

\section{Invariant measures}\label{SecInv}

\subsection{A $\sigma$-finite invariant measure}

\begin{figure}\label{simul}
\noindent
\includegraphics[width=.33\linewidth]{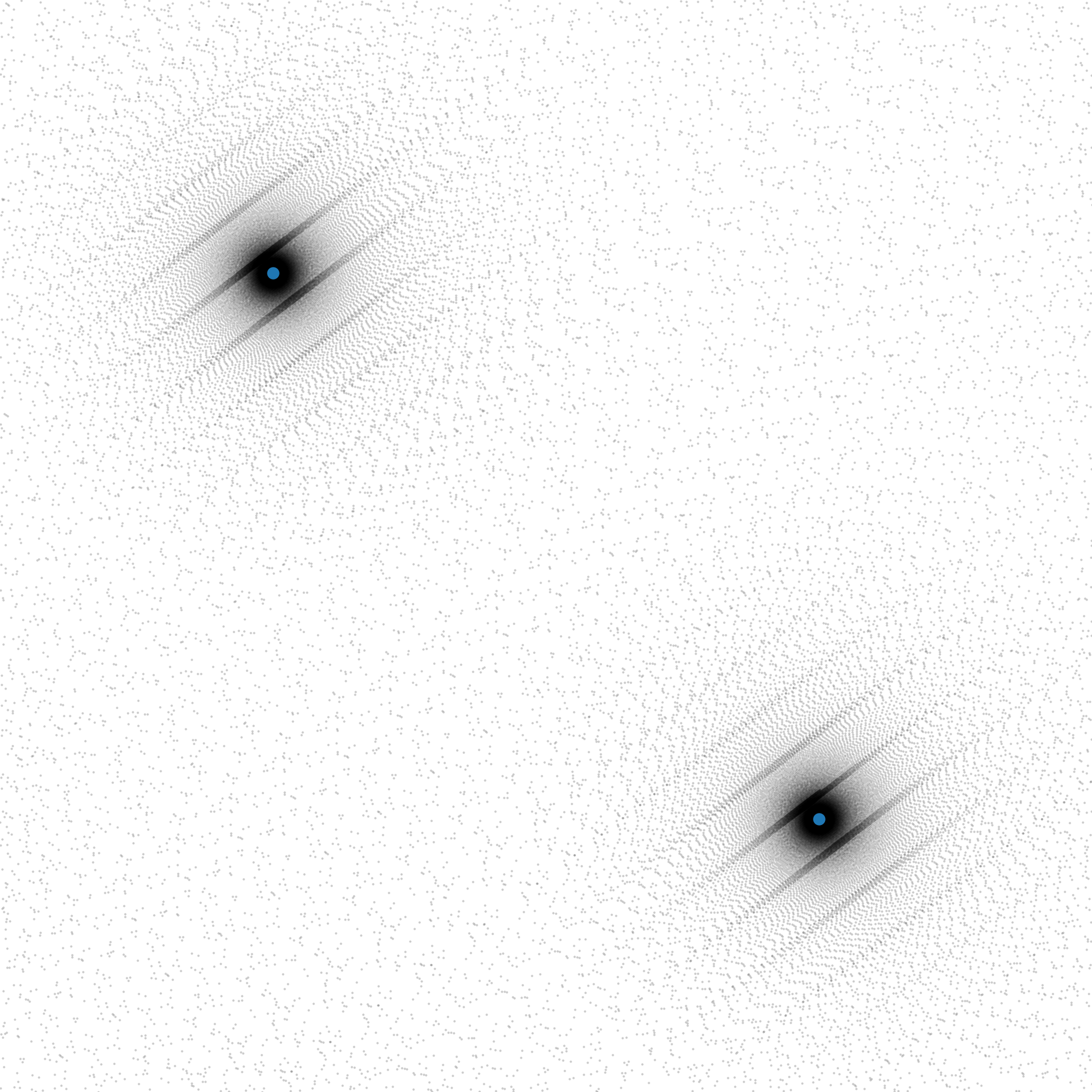}\hfill 
\includegraphics[width=.33\linewidth]{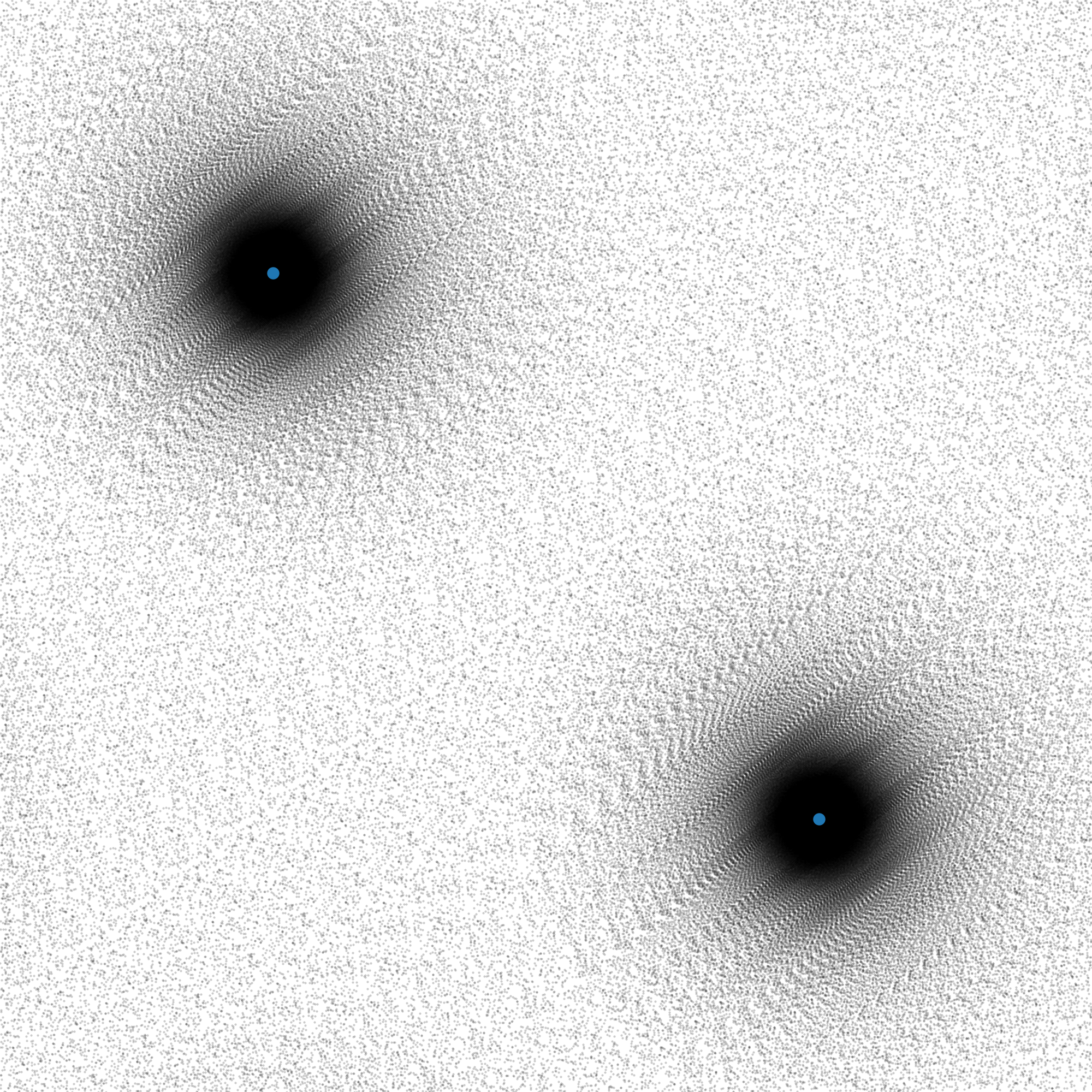}\hfill
\includegraphics[width=.33\linewidth]{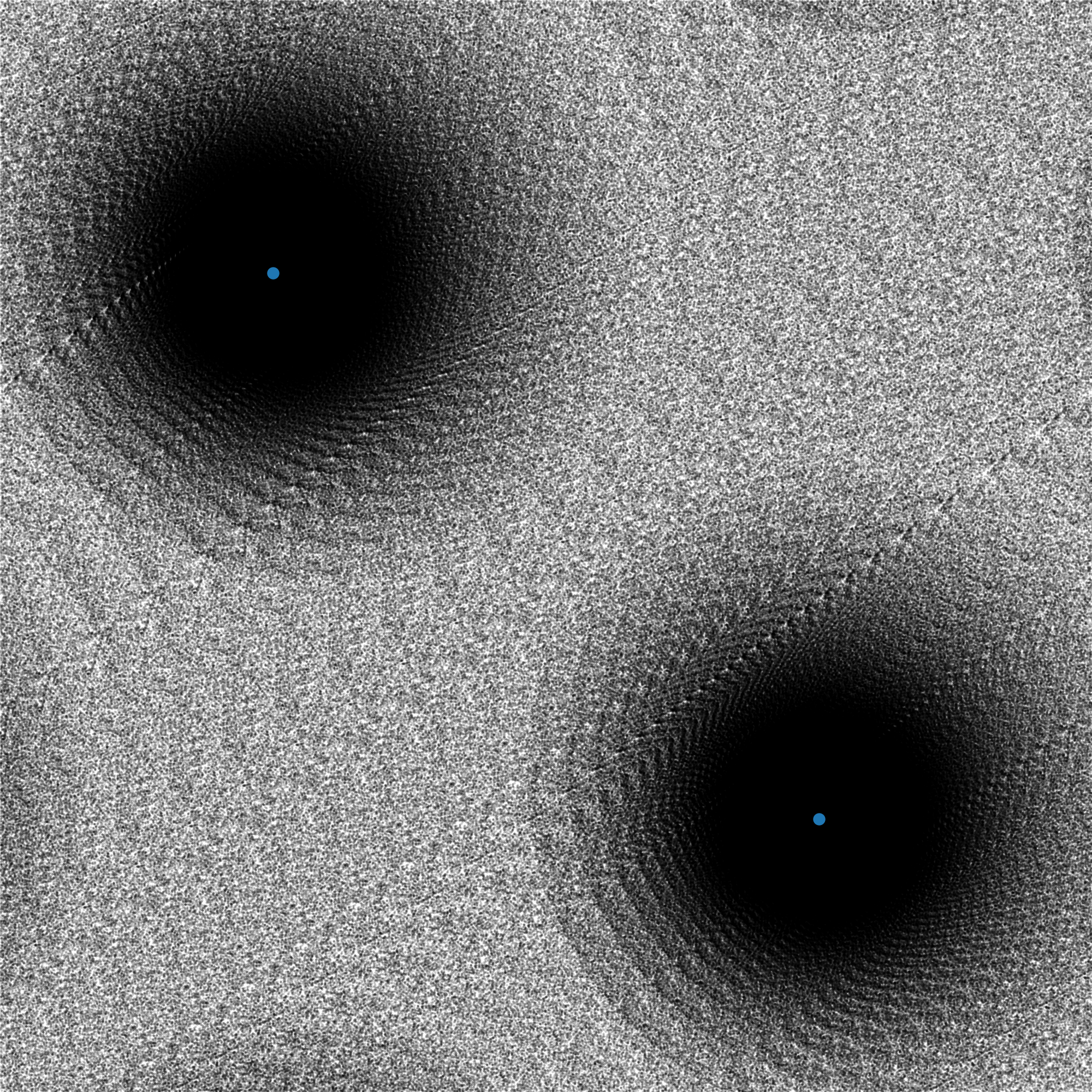} 
\caption{Simulations of the time-1 of the flow $\phi^t$ for simulation times $T=10^5$ (left), $10^6$ (middle) and $10^7$ (right). More precisely, these are simulations of a single orbit starting at point $M=(0.1,\,0.3)$ of the map $M\mapsto M+\delta\varphi(M)(1,\alpha)$, with $\delta=0.1972348$ and $\alpha=0.764831$, and $\varphi(M) = \min(\|M-\p\|_2, \|M-\q\|_2)$ for $\p=(0.25,0.75)$ and $\q=(0.75,0.25)$ (blue dots). Note that some strips can be observed on these simulations (they are more visible for $T=10^5$), which correspond to close returns to the initial conditions of the rotation of angle $\alpha$.}
\end{figure}

Let $\phi^t$ be a reparameterized linear flow satisfying (SH).
Consider the special flow $\Psi^t$ on the domain $D$ as described in Section \ref{special flows}. Let $m$ denote the restriction of the Lebesgue measure on $\RR$ to $D$. 
It is straightforward to check that $m$ is $\Psi^t$-invariant for every $t$. It follows that $\mu = \Xi_* m$ is invariant under $\phi^t$ and absolutely continuous with respect to the Haar measure $\lala$ on $\TT$. However --- and here's the catch --- due to the non-integrability of the roof function (Proposition~\ref{return times}), the measure $\mu$ is not a finite measure (although it is clearly $\sigma$-finite).

\begin{remark}
Instead of looking at $\TT$ we could consider reparameterization of a minimal linear flow on $\T^n$ for $n \geq 3$ with two stopping points. In this case, we would obtain a special flow over $\T^{n-1}$ whose roof function has, again, two asymptotics of order $\|x\|^{-1}$. However, for $n-1 \geq 2$, such a function is integrable. It therefore follows that there is an invariant probability $\mu$ absolutely continuous with respect to the Haar measure on $\T^n$. In particular, there cannot be an extreme historic behaviour in this setting. Unless, of course, the order of the zeros of the speed function at the stopping points is higher than quadratic.
\end{remark}

\subsection{Limit measures}

Let $\M_X$ denote the set of invariant probability measures for $\phi^t$. The following proposition is a special case of a more general result by Saghin-Sun-Vargas \cite[Proposition 1]{MR2670926}.

\begin{proposition}\label{invariant probs}
$\M_X = \{ \alpha \deltap + (1-\alpha) \deltaq: 0 \leq \alpha \leq 1 \}$.
\end{proposition}

We remark that $\phi^t$ is an example of a flow for which any point is non-wandering (and even, is in the closure of the set of recurrent points) but the union of the supports of the invariant measures is finite.
\medskip

Given measures $\mu, \nu \in \M_X$, we use the notation $[\mu, \nu]$ to denote the set $\{ \alpha \mu + (1-\alpha) \nu: \ 0 \leq \alpha \leq 1 \}$. Thus Proposition~\ref{invariant probs} can be written as $\M_X = [\delta_\p, \delta_\q]$.

The following proposition, for its part, says that the limit measures are almost everywhere constant.

\begin{proposition}\label{invariant probs better}
For any reparametrized linear flow (with zero Lebesgue measure set of singularities), the set $p\omega(\x)$ is almost everywhere constant.
\end{proposition}

\begin{proof}
Denote by $\cK(\M_X)$ the set of compact subsets of $\M_X$, endowed with a distance generating the Hausdorff topology. For every $n>0$, the set $\cK(\M_X)$ is covered by a finite number of balls $B(K_i^n,1/n)$. For any $n,i$, the set
\[\big\{\x\in \T^2 : p\omega(\x)\in B(K_i^n,1/n)\big\}\]
is $\phi^t$-invariant (and the union of these sets over $i$ is of full measure). Hence, by ergodicity, at least one of these sets is of measure 1: for any $n$, there exists $i$ such that $p\omega(\x)\in B(K_i^n,1/n)$ for a.e. $\x\in \T^2$. This implies that $p\omega(\x)$ is a.e. constant.
\end{proof}

\subsection{Computing limit measures in Diophantine terms}\label{SecLimitDioph}

In this section we link the ergodic behaviour of the flow $\phi^t$ with some limit behaviour of Birkhoff sums over the rotation $R_\alpha$. More precisely, we show how the presence of an extreme historic behaviour can be reduced to a Diophantine problem of comparing sums of reciprocals. First we develop a general criterion for the existence of an extreme historic behaviour (Proposition~\ref{criterium1}). 

For $x\in\T$, let
\begin{equation}\label{DefSn}
S_k(x) = \sum_{i=0}^{k-1} \frac{1}{\|x +  i \alpha \|}
\end{equation}
and
\begin{equation}\label{DefTheta}
\Theta_k^\beta(x) = \frac{S_k(x)}{S_k(x-\beta)}.
\end{equation}

Note that for any $\x\in\T^2$, by continuity of $t\mapsto \mu_\x^t$, the limit set $p\omega(\x)$ is connected. By combining Propositions \ref{invariant probs} and \ref{invariant probs better}, there exist $0\le \tau_0 \le \tau_1 \le 1$ such that
\begin{equation}\label{crit1}
p\omega(\x) = \Big[\tau_0\delta_\p + (1-\tau_0)\delta_\q,\ \tau_1\delta_\p + (1-\tau_1)\delta_\q\Big] \quad \la-a.e.
\end{equation}

\begin{proposition} \label{criterium1}
Let $\phi^t$ be as in (SH) and $\Sigma$ chosen so that $p_0 \neq q_0$. Let $\beta = q_0-p_0$ and $0\le \tau_0 \le \tau_1 \le 1$ such that \eqref{crit1} holds. Suppose that the positive orbit of $\x$ does not meet neither $\p$ nor $\q$.
Then
\begin{align*}
\limsup_{n \to \infty} \Theta_n^\beta(x) & = \sqrt{\frac{d_\p}{d_\q}}\left(\frac{\tau_1}{1-\tau_1} \right) \quad \la-a.e.,\text{ and}\\
\liminf_{n \to \infty} \Theta_n^\beta(x) & = \sqrt{\frac{d_\p}{d_\q}}\left(\frac{\tau_0}{1-\tau_0} \right) \quad \la-a.e.
\end{align*}
\end{proposition}

We begin by establishing an auxiliary property regarding accumulation points of $\mu_\x^t$. This will shorten the (a little lengthy but straightforward) proof of Proposition~\ref{criterium1}.

\begin{lemma}\label{equivalent accumulation}
Let $\phi^t$ be a reparameterized linear flow satisfying (SH) and $\mu_\x^t$ its associated family of empirical measures. Let $r>0$ be small enough so that $\B_r (\p)$ and $\B_r(\q)$ are disjoint. Then, given any $\x \in \TT$, the following are equivalent:
\begin{enumerate}
\item 
\[\tau \delta_\p + (1-\tau)\delta_\q \in p\omega(\x); \]
\item  
\[ \liminf_{t \to \infty} \mu_\x^t(\B_r(\p)) \leq \tau \leq \limsup_{t \to \infty} \mu_\x^t(\B_r(\p)); \]
\item 
\[ \liminf_{t \to \infty} \mu_\x^t(\B_r(\q)) \leq 1-\tau \leq \limsup_{t \to \infty} \mu_\x^t(\B_r(\q)). \]
\end{enumerate}
\end{lemma}

\begin{proof}
We start by recalling a useful characterization of weak* convergence: a sequence of measures $\mu_n$ converges weakly* to $\mu$ if and only if $\mu_n(U) \to \mu(U)$ for every open set $U$ satisfying $\mu(\partial U) = 0$. 

Suppose that (1) holds. Then we may choose a sequence $t_n \to \infty$ such that 
\[\mu_\x^{t_n} \to \tau \delta_\p + (1-\tau) \delta_\q.\]
Since $r$ is small, the boundary of both flow boxes $\B_r(\p)$ and $\B_r(\q)$ have zero $\tau \delta_\p + (1-\tau) \delta_\q$ measure. Hence
\[\lim_{n \to \infty}\mu_\x^{t_n}(\B_r(\p)) = \tau
\qquad \text{and} \qquad
\lim_{n \to \infty} \mu_\x^{t_n}(\B_r(\q)) = 1-\tau.\]
It follows that 
\begin{equation*}\label{liminfp}
\liminf_{t \to \infty} \mu_\x^t(\B_r(\p)) \leq \tau
\qquad \text{and} \qquad
\liminf_{t \to \infty} \mu_\x^t (\B_r(\q)) \leq 1-\tau.
\end{equation*}
and similarly
\[\limsup_{t \to \infty} \mu_\x^t (\B_r(\p)) \geq \tau
\qquad \text{and} \qquad
\limsup_{t \to \infty} \mu_\x^t (\B_r(\q)) \geq 1-\tau.\]
We have shown that (1) implies both (2) and (3). We shall now show that (2) implies (1). The proof that (3) implies (1) is analogous.

Suppose that (2) holds. By continuity of the map 
\[t \mapsto \mu_\x^t (\B_r(\p))\]
it is possible to find a sequence $t_n \to \infty$ such that 
\[\lim_{n \to \infty} \mu_\x^{t_n}(\B_r(\p)) = \tau.\]
The boundary of $\B_r(\p)$ has zero $\mu$-measure for every $\mu \in p\omega(\x)$. Hence any accumulation point $\mu$ of $\mu_\x^{t_n}$ must satisfy
\begin{equation} \label{accpoint}
\mu(\B_r(\p)) = \tau.
\end{equation}
By Proposition \ref{invariant probs}, only one measure in $\M_X$ satisfies (\ref{accpoint}), namely $\mu = \tau \delta_\p+(1-\tau) \delta_\q$. Therefore $\mu_x^{t_n} \to \mu$, so $\mu \in p\omega(\x)$.
\end{proof}

\begin{proof}[Proof of Proposition~\ref{criterium1}]
Let $r>0$ be small enough so that $\B_r(\p) \cap \B_r(\q) = \emptyset$. Upon possibly reducing $r$, we may suppose that the images of $\B_r(\p)$ and $\B_r(\q)$ by the return map on $\Sigma$ are disjoint (because $p_0\neq q_0$). Using Lemma~\ref{equivalent accumulation}, we know that there is a full $\la$-measure set $Z \subset \T$ such that, given any $y \in Z$, we have 
\begin{equation}\label{limsup}
\limsup_{t \to \infty} \mu_{(x_0,y)}^t(\B_r(\q)) = 1-\tau_0 \quad \text{and} \quad 
\liminf_{t \to \infty} \mu_{(x_0,y)}^t(\B_r(\q)) = 1-\tau_1.
\end{equation}
Note that this property is invariant under the flow, and therefore must hold on a set of full $\lala$-measure in $\TT$.

Consider the following functions from $\T$ to $\R_+ \cup \{\infty \}$ (see \eqref{EqDefTau}).
\begin{align*}
T(y) & = \min\{t>0: \phi^t((x_0,y)) \in \Sigma \} \\
S_\p (y) & = \la (\{ t \in [0, T(y)): \phi^t((x_0,y)) \in \B_r(\p) \})\\
S_\q (y) & = \la (\{ t \in [0, T(y)): \phi^t((x_0,y)) \in \B_r(\q) \}) \\
O(y) & = \la (\{t \in [0,T(y)): \phi^t((x_0,y)) \notin \B_r(\p) \cup \B_r(\q) \}) \\
A(y) & = \frac{\pi}{\sqrt{d_\p} \|y -p_0\| } \\
B(y) & = \frac{\pi}{\sqrt{d_\q} \|y -q_0 \| }. 
\end{align*} 
(We set the value of these functions to $\infty$ whenever their defining expressions are not well defined.) Note that, since $\B_r(\p)$ and $\B_r(\q)$ are disjoint, we have 
\[S_\p + S_\q + O = T.\]
\medskip

We know from Lemma~\ref{time in a box} that there are functions $\sigma_\p, \sigma_\q \in L^1(\T)$ such that 
\[S_\p = A+ \sigma_\p \qquad \text{and} \qquad S_\q = B + \sigma_\q.\]
Writing
\[C = O + \sigma_\p + \sigma_\q \]
and using the notation
\begin{align}
A_n  = \sum_{k=0}^{n-1} A\circ R_\alpha^k, \quad
B_n  = \sum_{k=0}^{n-1} B\circ R_\alpha^k, \quad 
C_n  = \sum_{k=0}^{n-1} C\circ R_\alpha^k, \quad
T_n  = \sum_{k=0}^{n-1} T\circ R_\alpha^k,
\end{align}
we get
\[ A_n + B_n + C_n = T_n. \]

We remark that by the fact that the images of $\B_r(\p)$ and $\B_r(\q)$ by the return map on $\Sigma$ are disjoint, the property \eqref{limsup} can be replaced by
\begin{equation}\label{limsup2}
\limsup_{n \to \infty} \mu_{(x_0,y)}^{T_n(y)}(\B_r(\q)) = 1-\tau_0 \quad \text{and} \quad 
\liminf_{n \to \infty} \mu_{(x_0,y)}^{T_n(y)}(\B_r(\q)) = 1-\tau_1.
\end{equation}
for $\la$ almost every $y\in Y$.
\medskip

Note that 
\begin{align*}
\mu_{(x_0,y)}^{T_n(y)}\big(\B_r(\q)\big) 
& = \frac{\sum_{k=0}^{n-1}S_\q(R_\alpha^k(y))}{T_n(y)}\\
& = \frac{\sum_{k=0}^{n-1}\sigma_\q(R_\alpha^k(y))}{T_n(y)} +  \frac{B_n(y)}{T_n(y)}\\
& = \frac{\sum_{k=0}^{n-1}\sigma_\q(R_\alpha^k(y))}{T_n(y)} + \frac{1}{1+\frac{A_n(y)}{B_n(y)}+\frac{C_n(y)}{B_n(y)}}.
\end{align*}
Recall that $\sigma_\q$ and $C$ are integrable functions whereas $B$ and $T$ are not. Hence 
\[\frac{\sum_{k=0}^{n-1}\sigma_\q(R_\alpha^k(z))}{T_n(z)} \to 0,\qquad \frac{C_n(z)}{B_n(z)} \to 0 \qquad \la-a.e.\]
(because by ergodicity, $C_n(z)/n \to \int C$ almost everywhere, while $B_n(z)/n$ tends to $+\infty$ almost everywhere).

Consequently, as by \eqref{limsup2}
\[\limsup_{n \to \infty} \mu_{(x_0,y)}^{T_n(y)} \big(\B_r(\q)\big) = 1-\tau_0 \quad \la-a.e.,\]
one has
\[\liminf_{n\to \infty} \Theta_n^\beta (y-p_0) = \liminf_{n \to \infty} \sqrt{\frac{d_\p}{d_\q}}\frac{A_n(y)}{B_n(y)} = \sqrt{\frac{d_\p}{d_\q}}\left(\frac{\tau_0}{1-\tau_0} \right) \quad \la-a.e.\] 

A similar argument holds for  $\liminf_{n \to \infty} \Theta_n^\beta$.
\end{proof}

\subsection{Consequences of a symmetry property of $\Theta_n^\beta$}\label{SubsecSym}

We shall see that the functions $\Theta_n^\beta$ have a nice symmetry property. It will imply that there is only one possibility for physical measures (Proposition~\ref{PropPossibOmega}), and give more easily checkable criteria for the existence of an extreme historic behaviour (Propositions \ref{criterium2} and \ref{criterium3}) than Proposition~\ref{criterium1}.

Denote by $I : \T \to \T$ the involution map $x \mapsto -x$ and let
\[J_n^\beta = R_{\beta - (n-1) \alpha} \circ I.\]
Note that $J_n^\beta$ can also be written as $I \circ R_{(n-1)\alpha-\beta}$. 

\begin{lemma}\label{symmetry}
We have
\[\Theta_n^\beta \circ J_n^\beta  = \frac{1}{\Theta_n^\beta}.\]
\end{lemma}

\begin{proof} 
Direct calculation.
\end{proof}

An important consequence of Lemma~\ref{symmetry} is that that it gives us only one possible candidate for physical measure.

\begin{proposition}\label{PropPossibOmega}
If $\phi^t$ has a physical measure, then it is equal to $\mu_\infty$ (defined in \eqref{EqFormPhys}).
\end{proposition}

Hence, either $\phi^t$ has an historic behaviour, or it admits this measure as a physical measure with full basin.

\begin{proof}
We already know that if $\phi^t$ has a physical measure, then this one is unique (it is a consequence of ergodicity, see Proposition~\ref{invariant probs better}). Suppose then that 
\[\mu_\infty = \tau \delta_\p + (1-\tau) \delta_\q\]
is a physical measure for $\phi^t$. Then, according to Proposition~\ref{criterium1} we must have 
\begin{equation}\label{physmeas}
\lim_{n \to \infty} \Theta_n^\beta(x) = \sqrt{\frac{d_\p}{d_\q}}\left( \frac{\tau}{1-\tau} \right)
\end{equation}
for $\la$-a.e. $x\in\T$.
Let $T$ be the right hand side of (\ref{physmeas}). We claim that $T = 1$. Indeed, suppose that $T>1$. Then there exists some $N\in\N$ such that
\begin{equation}\label{EqPossib}
\la  \big\{x\in\T : \Theta_n^\beta(x)>1\big\} > \half
\end{equation}
for every $n \geq N$. Hence, according to Lemma~\ref{symmetry}, we have
 
\[\la  \big\{x\in\T : \Theta_n^\beta \circ J_n^\beta <1 \big\} > \half\]
for every $n \geq N$. But this is not possible since $J_n^\beta$ preserves $\la$. Hence $T \leq 1$. A similar argument shows that $T \geq 1$. 

Solving for $\tau$ in the equation
\[\sqrt{\frac{d_\p}{d_\q}}\left( \frac{\tau}{1-\tau} \right) = 1\]
gives 
\[\tau = \frac{\sqrt{d_\q}}{\sqrt{d_\p}+\sqrt{d_\q}}.\]
\end{proof}

One could expect Lemma~\ref{symmetry} to imply the $\la$-almost everywhere symmetry of the set $p\omega(\x)$. This is not true in its full generality (see Theorem~\ref{PropDivSum}). The reason why it is not stems from an insufficiency of information about the behaviour of $\mu_\x^t$ by considering only the sequence of return times to $\Sigma$.

Lemma~\ref{symmetry} also provides us with two simple criteria for the presence of extreme historic behaviour, under some uniformity hypotheses.

\begin{proposition}\label{criterium2}
Let $\phi^t$ be as in (SH). Suppose that there exists $C>0$ such that, given any $K>1$, one can find $n\in\N$ such that
\begin{equation*}\label{crit2}
\la \big\{x \in \T: \Theta_n^\beta(x)>K \big\} \geq C.
\end{equation*}
Then $\phi^t$ has an extreme historic behaviour.
\end{proposition}

\begin{proposition}\label{criterium3}
Let $\phi^t$ be as in (SH). Suppose that, given any $K>1$, there exists $C>0$ such that
\begin{equation*}\label{crit3}
\la\big\{x \in \T: \Theta_n^\beta(x)>K \big\} \geq C
\end{equation*}
for infinitely many $n$. Then $\phi^t$ has an extreme historic behaviour.
\end{proposition}

\begin{proof}[Proof of Proposition~\ref{criterium2}]
According to Proposition~\ref{criterium1} it suffices to show that 
\begin{equation} \label{limsupinfinfty}
\limsup_{n \to \infty} \Theta_n^\beta(x) = \infty 
\qquad \text{and}\qquad 
\liminf_{n \to \infty} \Theta_n^\beta(x) = 0
\end{equation}
for $\la$-almost every $x$. 

It is straightforward to check that the set on which \eqref{limsupinfinfty} holds is $R_\alpha$-invariant. Thus to show that $\phi^t$ has an extreme historic behaviour, it suffices to show that this set has positive $\la$-measure. 
  
Let 
\[A_{K,n} = \big\{ x \in \T : \Theta_n^\beta(x) >K\big\},\quad A_K = \bigcup_{n \in \N} A_{K,n}, \quad A = \bigcap_{K>1} A_K\]
and
\[B_{K,n} = \big\{x \in \T: \Theta_n^\beta(x) < 1/K \big\},\quad B_K = \bigcup_{n \in \N} B_{K,n}, \quad B = \bigcap_{K>1} B_K.\]
Note that $\limsup_{n\to \infty} \Theta_n^\beta(x) = \infty$ if and only if $x \in A$, and that $\liminf_{n \to \infty} \Theta_n^\beta(x) =0$ if and only if $x \in B$.

The hypothesis in Proposition~\ref{criterium2} implies that $\la(A_K) \geq C$ for every $K>1$. Also, it follows from Lemma~\ref{symmetry} that 
\begin{align*}
\la(B_{K,n}) & = \la \big\{x \in \T: \Theta_n^\beta(x) < 1/K \big\} \\ 
& = \la \big\{ x \in \T : \Theta_n^\beta \circ J_n^\beta (x) > K \big\} \\
& = \la(A_{K,n}) \geq C.
\end{align*}
It therefore follows from our hypothesis that $\la(B_{K}) \geq C$. Note that $A_K$ is a decreasing family in the sense that $A_{K'} \subset A_K$ whenever $K' \geq K$. It follows that $\la(A) \geq C$. The proof that $\la(B)\geq C$ is analogous. 
\end{proof}

\begin{proof}[Proof of Proposition~\ref{criterium3}]
In view of Proposition~\ref{criterium1} it suffices to show that 
\begin{equation}\label{limsupinfty}
\limsup_{n \to \infty} \Theta_n^\beta(x) = \infty 
\end{equation}
and
\begin{equation}\label{liminfzero}
\liminf_{n \to \infty} \Theta_n^\beta(x) = 0
\end{equation}
hold for $\la$-almost every $x$. 

The hypothesis implies that, given any $K>1$, there exists $C>0$ such that 
\[ \la \left( \bigcup_{k \geq n} \big\{x : \Theta_k^\beta(x) > K \big\} \right) \geq C\]
for every $n \geq 1$. 
Note that the sequence $\bigcup_{k \geq n} \{x : \Theta_k^\beta(x) > K \}$ is decreasing in $n$ so that 
\[\la \left( \bigcap_{n \geq 1} \bigcup_{k\geq n}
\big\{x: \Theta_k^\beta(x)>K\big\}\right) \geq C. \]
 In particular, the set
\[A_K=\bigcup_{\epsilon>0} \bigcap_{n \geq 1} \bigcup_{k\geq n}
\big\{x: \Theta_k^\beta (x) > K + \epsilon \big\} \]
has positive $\la$-measure for every $K>1$ (the term $\epsilon$ is added to ensure the invariance of the set $A_K$). Analogously, using Lemma~\ref{symmetry}, one sees that
\[ B_K = \bigcup_{\epsilon>0} \bigcap_{n \geq 1} \bigcup_{k\geq n}
\left\{x: \Theta_k^\beta (x) < \frac{1}{K+\epsilon} \right\} \]
has positive $\la$-measure for every $K>1$.

The sets $A_K$ and $B_K$ are $R_\alpha$-invariant. Indeed, $A_K$ is the set of points on which $\limsup_n \Theta_n^\beta $ is larger than $K$ and $B_K$ is the set on which $\liminf_n \Theta_n^\beta $ is smaller than $1/K$, and it is straightforward to check that $\Theta_n^\beta(x) \sim_{n} \Theta_n^\beta(x+\alpha)$ for every $x$.

By ergodicity of $R_\alpha$ we conclude that $A_K$ and $B_K$ have full $\la$-measure. Note that $A_K$ and $B_K$ are decreasing families. Thus
\[A=\bigcap_{K>1} A_K\qquad \text{and}\qquad B=\bigcap_{K>1} B_K\]
are also of full $\la$-measure. But $A$ and $B$ are the sets on which (\ref{limsupinfty}) and (\ref{liminfzero}) hold, respectively. The proof is therefore complete.
\end{proof}

%
%

%
%
%
%

\section{Circle rotations} \label{PartRotations}

\subsection{Diophantine approximation theory}

Recall that for any $x\in\T = \R/\Z$ we define its norm as 
\[ \|x\| = \min_{m\in\Z} |\tilde x-m|,\]
where $\tilde x$ is a lift of $x$ to $\R$.

Let $\alpha>0$ be an irrational number, and let $R_\alpha: \T \to \T$ be its associated circle rotation. We write $\alpha = [a_0;a_1,a_2\cdots]$ for its expansion as a continued fraction. We denote by $p_n/q_n = [a_0;a_1,a_2,\cdots,a_n]$ and $\alpha_n$ such that
\[\alpha = [a_0;a_1,\cdots, a_{n-1},\alpha_n].\]

The sequence $q_n$ is characterized by the properties
\begin{itemize}
\item $q_0 = 1$, and
\item $q_n = \min \{ k>q_{n-1}: \| k \alpha \| < \|q_{n-1} \alpha \| \} $ for every $n \geq 1$.
\end{itemize}
We also set $\rho_n = q_n \alpha - p_n$ and $\lambda^{(n)} = |\rho_n|$. For $k \in \N$, let $\cO(k)$ denote the orbit $\{R_\alpha^i(0): 0 \leq i \leq k-1 \}$, and
\begin{align*}
m(\cO(k)) & = \min_{x \in \cO(k)} \min_{y \in \cO(k) \setminus{x}} \|x-y\|, \text{ and}\\
M(\cO(k)) & = \max_{x \in \cO(k)} \min_{y \in \cO(k) \setminus{x}} \| x-y \|
\end{align*}
be the smallest resp. largest distance between two consecutive points of the orbit $\cO(k)$ on $\T$ (``gaps''). The following lemma recalls classical facts of Diophantine approximation theory, that will be used in the sequel (some of them can be deduced from renormalization properties, see Figure \ref{FigRenor0}).

\begin{lemma}\label{properties}
\begin{equation}\label{EqContFrac0}
\lambda^{(n)} = (-1)^n \rho_n = \min_{0\le j< q_{n+1}}\|j\alpha\| ;
\end{equation}
\begin{equation}\label{EqContFrac1}
\frac{\lambda^{(n-2)}}{\lambda^{(n-1)}} = \alpha_n \quad \text{and} \quad a_n = \lfloor \alpha_n \rfloor,
\end{equation}
\begin{equation}\label{EqTotTime}
q_{n+1} = q_{n}a_{n+1} + q_{n-1};
\end{equation}
\begin{equation}\label{Eqaeta}
\lambda^{(n-1)}= a_{n+1} \lambda^{(n)} + \lambda^{(n+1)};
\end{equation}
\begin{equation}\label{EqLambdaQ}
q_{n+1} \ge q_n \quad \text{and} \quad\frac{1}{2q_{n+1}} \le \frac{1}{q_{n+1}+q_n} < \lambda^{(n)} < \frac{1}{q_{n+1}}.
\end{equation}
\begin{align*}
 m(\cO(q_n)) & = \lambda^{(n-1)} > \frac{1}{q_{n-1}+q_n};\\
 M(\cO(q_n)) & = \lambda^{(n)} + \lambda^{(n-1)} < \frac{1}{q_n} + \frac{1}{q_{n+1}};
 \end{align*}
given any integer $m$ such that $0<m<q_{n+1}$, we write the Euclidean division $m=\ell q_n+r$, and then
\begin{equation}\label{EqOrbit}
\cO(m) = \left( \bigcup_{i=0}^{\ell-1} R_{\rho_n}^i \cO(q_n)\right) \cup R_{\rho_n}^\ell \cO(r).
\end{equation}
\end{lemma}

\subsection{Renormalization of rotations}\label{subsecrenor}

We now recall some facts about renormalization intervals for circle rotations and their link with continued fractions. This renormalization is at the basis of the ideas of some in the proofs we will present in the sequel and will give nice geometric interpretations of our arguments. We will reuse the notations  of Sina\u{\i} and Ulcigrai. \cite[\S 1.1]{MR2478478} (see also Sina\u{\i} \cite[Lecture 9]{MR1258087} and the nice visualizations of Hariss and Arnoux \cite{VideoArnoux}).


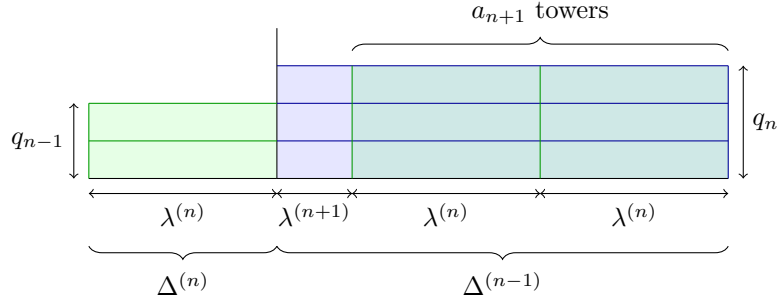
\begin{figure}
\begin{tikzpicture}[scale=1]
\fill[fill=green, opacity=.1] (0,0) rectangle (-2.5,1);
\fill[fill=blue, opacity=.1] (0,0) rectangle (6,1.5);
\fill[fill=green, opacity=.1] (1,0) rectangle (6,1.5);
\draw (0,0) -- (0,2);
\draw (0,0) -- (6,0);
\draw[color=blue!60!black] (0,.5) -- (6,.5);
\draw[color=blue!60!black] (0,1) -- (6,1);
\draw[color=blue!60!black] (0,1.5) -- (6,1.5);
\draw[color=blue!60!black] (6,0) -- (6,1.5);

\draw (-2.5,0) -- (0,0);
\draw[color=green!60!black] (-2.5,.5) -- (0,.5);
\draw[color=green!60!black] (-2.5,1) -- (0,1);
\draw[color=green!60!black] (-2.5,0) -- (-2.5,1);
\draw[color=green!60!black] (1,0) -- (1,1.5);
\draw[color=green!60!black] (3.5,0) -- (3.5,1.5);

\draw [decorate,decoration={brace,amplitude=5pt},xshift=0,yshift=-.9cm]
(6,0) -- (0,0) node [black,midway,yshift=-0.5cm] {$\Delta^{(n-1)}$};
\draw [decorate,decoration={brace,amplitude=5pt},xshift=0,yshift=-.9cm]
(0,0) -- (-2.5,0) node [black,midway,yshift=-0.5cm] {$\Delta^{(n)}$};

\draw [decorate,decoration={brace,amplitude=5pt},xshift=0,yshift=.2cm]
(1,1.5) -- (6,1.5) node [black,midway,yshift=0.5cm] {$a_{n+1}$ towers};

\draw[<->] (6.2,0) --node[midway, right]{$q_n$} (6.2,1.5);
\draw[<->] (-2.7,0) --node[midway, left]{$q_{n-1}$} (-2.7,1);
\draw[<->] (-2.5,-.2) --node[midway, below]{$\lambda^{(n)}$} (0,-.2);
\draw[<->] (1,-.2) --node[midway, below]{$\lambda^{(n+1)}$} (0,-.2);
\draw[<->] (1,-.2) --node[midway, below]{$\lambda^{(n)}$} (3.5,-.2);
\draw[<->] (3.5,-.2) --node[midway, below]{$\lambda^{(n)}$} (6,-.2);
\end{tikzpicture}
\caption{The renormalization interval $\Delta(n-1) = \Delta^{(n)} \cup \Delta^{(n-1)}$ and the associated quantities.}\label{FigRenor0}
\end{figure}

Let
\[\Delta^{(n)} = \left\{\begin{array}{ll}
[0,\{q_n\alpha\})\quad & \text{if $n$ is even}\\
{[}\{q_n\alpha\},1)\quad & \text{if $n$ is odd.}
\end{array}\right.\]
We also denote $\Delta_j^{(n)} = R_\alpha^j (\Delta^{(n)})$. Remark that the length of $\Delta_j^{(n)}$ satisfies $\la(\Delta^{(n)}_j) = \lambda^{(n)}$, and that $\lambda^{(n-1)} = \lambda^{(n+1)} + a_{n+1}\lambda^{(n)}$ (it is Equation \eqref{Eqaeta}, which can be observed on  Figure \ref{FigRenor0}).

For any $n$, the collection made of the intervals $(\Delta_j^{(n)})_{0\le j < q_{n+1}}$ and $(\Delta_j^{(n+1)})_{0\le j < q_n}$ form a partition $\xi^{(n)}$ of $[0,1)$. It is decomposed into two towers
\begin{equation}\label{EqTower}
Z^{(n)}_l = \bigcup_{j=0}^{q_{n+1}-1} \Delta^{(n)}_j \quad \text{and} \quad Z^{(n)}_s = \bigcup_{j=0}^{q_{n}-1} \Delta^{(n+1)}_j,
\end{equation}
called respectively the \emph{large} and \emph{small} towers.

The interval $\Delta(n) = \Delta^{(n)} \cup \Delta^{(n+1)}$ is called the \emph{$n^\text{th}$ renormalization interval} of the rotation of angle $\alpha$ on $\T$. It can be seen as a subset of $\T$, so that one can define the \emph{induced map} $T^{(n)}$ as the first return map of $R_\alpha$ on $\Delta(n)$. This induced map $T^{(n)}$ is a rotation of angle $\pm\lambda^{(n)}$ (the sign depending of the parity of $n$). Moreover, the return time is constant  equal to $q_{n+1}$ on $\Delta^{(n)}$ and constant equal to $q_n$ on $\Delta^{(n+1)}$ (see Figure \ref{FigRenor1}).

For $x\in \T$, we denote $x^{(n)}$ the projection of $x$ on $\Delta(n)$. More precisely
\begin{itemize}
\item if $x\in \Delta^{(n)}_j$ for some $0\le j <q_{n+1}$, then $x^{(n)} = R_\alpha^{-j}(x)$;
\item if $x\in \Delta^{(n+1)}_j$ for some $0\le j <q_{n}$, then $x^{(n)} = R_\alpha^{-j}(x)$.
\end{itemize}
\medskip

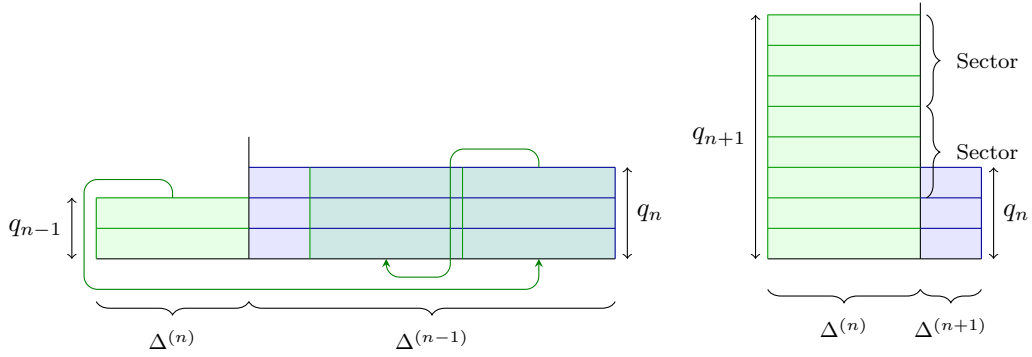
\begin{figure}
\resizebox{\textwidth}{!}{
\begin{tikzpicture}[scale=.8]
\fill[fill=green, opacity=.1] (0,0) rectangle (-2.5,1);
\fill[fill=blue, opacity=.1] (0,0) rectangle (6,1.5);
\fill[fill=green, opacity=.1] (1,0) rectangle (6,1.5);
\draw (0,0) -- (0,2);
\draw (0,0) -- (6,0);
\draw[color=blue!60!black] (0,.5) -- (6,.5);
\draw[color=blue!60!black] (0,1) -- (6,1);
\draw[color=blue!60!black] (0,1.5) -- (6,1.5);
\draw[color=blue!60!black] (6,0) -- (6,1.5);

\draw (-2.5,0) -- (0,0);
\draw[color=green!60!black] (-2.5,.5) -- (0,.5);
\draw[color=green!60!black] (-2.5,1) -- (0,1);
\draw[color=green!60!black] (-2.5,0) -- (-2.5,1);
\draw[color=green!60!black] (1,0) -- (1,1.5);
\draw[color=green!60!black] (3.5,0) -- (3.5,1.5);

\draw [decorate,decoration={brace,amplitude=5pt},xshift=0,yshift=-.7cm]
(6,0) -- (0,0) node [black,midway,yshift=-0.5cm] {\footnotesize $\Delta^{(n-1)}$};
\draw [decorate,decoration={brace,amplitude=5pt},xshift=0,yshift=-.7cm]
(0,0) -- (-2.5,0) node [black,midway,yshift=-0.5cm] {\footnotesize $\Delta^{(n)}$};

\draw[->,>=stealth,rounded corners=5pt,color=green!50!black] (-1.25,1) -- (-1.25,1.3) -- (-2.7,1.3) -- (-2.7,-.5) -- (4.75,-.5) -- (4.75,0);

\draw[->,>=stealth,rounded corners=5pt,color=green!50!black] (4.75,1.5) -- (4.75,1.8) -- (3.3,1.8) -- (3.3,-.3) -- (2.25,-.3) -- (2.25,0);

\draw[<->] (6.2,0) --node[midway, right]{$q_n$} (6.2,1.5);
\draw[<->] (-2.9,0) --node[midway, left]{$q_{n-1}$} (-2.9,1);

\fill[fill=green, opacity=.1] (11,0) rectangle (8.5,4);
\fill[fill=blue, opacity=.1] (11,0) rectangle (12,1.5);
\draw (11,0) -- (11,4.2);
\draw (11,0) -- (12,0);
\draw[color=blue!60!black] (11,.5) -- (12,.5);
\draw[color=blue!60!black] (11,1) -- (12,1);
\draw[color=blue!60!black] (11,1.5) -- (12,1.5);
\draw[color=blue!60!black] (12,0) -- (12,1.5);

\draw (8.5,0) -- (11,0);
\draw[color=green!60!black] (8.5,.5) -- (11,.5);
\draw[color=green!60!black] (8.5,1) -- (11,1);
\draw[color=green!60!black] (8.5,1.5) -- (11,1.5);
\draw[color=green!60!black] (8.5,2) -- (11,2);
\draw[color=green!60!black] (8.5,2.5) -- (11,2.5);
\draw[color=green!60!black] (8.5,3) -- (11,3);
\draw[color=green!60!black] (8.5,3.5) -- (11,3.5);
\draw[color=green!60!black] (8.5,4) -- (11,4);
\draw[color=green!60!black] (8.5,0) -- (8.5,4);

\draw [decorate,decoration={brace,amplitude=5pt},xshift=0,yshift=-.5cm]
(12,0) -- (11,0) node [black,midway,yshift=-0.5cm] {\footnotesize $\Delta^{(n+1)}$};
\draw [decorate,decoration={brace,amplitude=5pt},xshift=0,yshift=-.5cm]
(11,0) -- (8.5,0) node [black,midway,yshift=-0.5cm] {\footnotesize $\Delta^{(n)}$};

\draw[<->] (12.2,0) --node[midway, right]{$q_n$} (12.2,1.5);
\draw[<->] (8.3,0) --node[midway, left]{$q_{n+1}$} (8.3,4);

\draw [decorate,decoration={brace,amplitude=5pt},xshift=.1cm,yshift=0cm]
(11,2.5) -- (11,1) node [black,midway,xshift=0.8cm] {\footnotesize Sector};
\draw [decorate,decoration={brace,amplitude=5pt},xshift=.1cm,yshift=0cm]
(11,4) -- (11,2.5) node [black,midway,xshift=0.8cm] {\footnotesize Sector};
\end{tikzpicture}
}

\caption{Renormalization intervals $\Delta(n-1)$ (left) and $\Delta(n)$ (right) for odd $n$. The green arrows denote the dynamics of the intervals, i.e. the way to build the dynamics of the partition $\xi^{(n)}$ from that of $\xi^{(n-1)}$.}\label{FigRenor1}
\end{figure}

\begin{figure}
\resizebox{\textwidth}{!}{
\begin{tikzpicture}
\fill[fill=green, opacity=.1] (0,0) rectangle (-2.5,1);
\fill[fill=blue, opacity=.1] (0,0) rectangle (6,1.5);
\fill[fill=green, opacity=.1] (1,0) rectangle (6,1.5);
\draw (0,0) -- (0,2);
\draw (0,0) -- (6,0);
\draw[color=blue!60!black] (0,.5) -- (6,.5);
\draw[color=blue!60!black] (0,1) -- (6,1);
\draw[color=blue!60!black] (0,1.5) -- (6,1.5);
\draw[color=blue!60!black] (6,0) -- (6,1.5);

\draw (-2.5,0) -- (0,0);
\draw[color=green!60!black] (-2.5,.5) -- (0,.5);
\draw[color=green!60!black] (-2.5,1) -- (0,1);
\draw[color=green!60!black] (-2.5,0) -- (-2.5,1);
\draw[color=green!60!black] (1,0) -- (1,1.5);
\draw[color=green!60!black] (3.5,0) -- (3.5,1.5);

\draw[color=red!70!black, thick] (0,.2) -- (0,-.2);
\draw[color=red!70!black, thick] (2.5,.2) -- (2.5,-.2);
\draw[color=red!70!black, dashed, thick] (2.5,0) -- (2.5,1.5);
\draw[color=red!70!black, thick] (5,.2) -- (5,-.2);
\draw[color=red!70!black, dashed, thick] (5,0) -- (5,1.5);
\draw[color=red!70!black, thick] (-1,.2) -- (-1,-.2);
\draw[color=red!70!black, dashed, thick] (-1,0) -- (-1,1);

\draw[->,shorten >=2pt,shorten <=2pt,color=red!70!black] (0,-.2) -- (2.5,-.2);
\draw[->,shorten >=2pt,shorten <=2pt,color=red!70!black] (2.5,-.2) -- (5,-.2);
\draw[shorten <=2pt,color=red!70!black] (5,-.2) -- (5.8,-.2);
\draw[densely dotted,color=red!70!black] (5.8,-.2) -- (6,-.2);
\draw[->,shorten >=2pt,color=red!70!black] (-2.3,-.2) -- (-1,-.2);
\draw[densely dotted,color=red!70!black] (-2.5,-.2) -- (-2.3,-.2);

\draw[->,>=stealth,rounded corners=5pt,color=red!50!black, dotted, thick] (0,-.2) -- (0,-.5) -- (.8,-.5) -- (.8,2) -- (2.5,2) -- (2.5,1.6);
\draw[->,>=stealth,rounded corners=5pt,color=red!50!black, dotted, thick] (2.5,-.2) -- (2.5,-.5) -- (3.2,-.5) -- (3.2,2) -- (5,2) -- (5,1.6);
\draw[rounded corners=5pt,color=red!50!black, dotted, thick] (5,-.2) -- (5,-.5) -- (5.7,-.5) -- (5.7,2) -- (6,2) ;
\draw[->,>=stealth,rounded corners=5pt,color=red!50!black, dotted, thick] (-2.5,2) -- (-1,2) -- (-1,1.1);

\draw[color=red!70!black]  (1.5,-.5) node{$\lambda^{(n)}$};


\fill[fill=green, opacity=.1] (11,0) rectangle (8.5,4);
\fill[fill=blue, opacity=.1] (11,0) rectangle (12,1.5);
\draw (11,0) -- (11,4.5);
\draw (11,0) -- (12,0);
\draw[color=blue!60!black] (11,.5) -- (12,.5);
\draw[color=blue!60!black] (11,1) -- (12,1);
\draw[color=blue!60!black] (11,1.5) -- (12,1.5);
\draw[color=blue!60!black] (12,0) -- (12,1.5);

\draw (8.5,0) -- (11,0);
\draw[color=green!60!black] (8.5,.5) -- (11,.5);
\draw[color=green!60!black] (8.5,1) -- (11,1);
\draw[color=green!60!black] (8.5,1.5) -- (11,1.5);
\draw[color=green!60!black] (8.5,2) -- (11,2);
\draw[color=green!60!black] (8.5,2.5) -- (11,2.5);
\draw[color=green!60!black] (8.5,3) -- (11,3);
\draw[color=green!60!black] (8.5,3.5) -- (11,3.5);
\draw[color=green!60!black] (8.5,4) -- (11,4);
\draw[color=green!60!black] (8.5,0) -- (8.5,4);

\draw[color=red!70!black, thick] (10,.2) -- (10,-.2);
\draw[color=red!70!black, thick] (11,.2) -- (11,-.2);
\draw[color=red!70!black, dashed, thick] (10,0) -- (10,4);
\draw[color=orange, thick] (9,.2) -- (9,-.2);
\draw[color=orange, densely dotted, thick] (9,0) -- (9,4);
\draw[color=orange, thick] (11.5,.2) -- (11.5,-.2);
\draw[color=orange, densely dotted, thick] (11.5,0) -- (11.5,1.5);

\draw[->,shorten >=2pt,shorten <=2pt,color=orange] (11,-.2) -- (10,-.2);
\draw[->,shorten >=2pt,shorten <=2pt,color=orange] (10,-.2) -- (9,-.2);
\draw[shorten <=2pt,color=orange] (9,-.2) -- (8.7,-.2);
\draw[densely dotted,color=orange] (8.7,-.2) -- (8.5,-.2);
\draw[densely dotted,color=orange] (12,-.2) -- (11.8,-.2);
\draw[->,shorten >=2pt,color=orange] (11.8,-.2) -- (11.5,-.2);

\draw[color=orange!80!black]  (10.5,-.6) node{$-\lambda^{(n+1)}$};

\end{tikzpicture}}

\caption{Set of preimages of 0 by the rotation in time $q_{n+1}$ (red, dashed) and $q_{n+2}$ (orange, dotted).}\label{FigRenor2}
\end{figure}
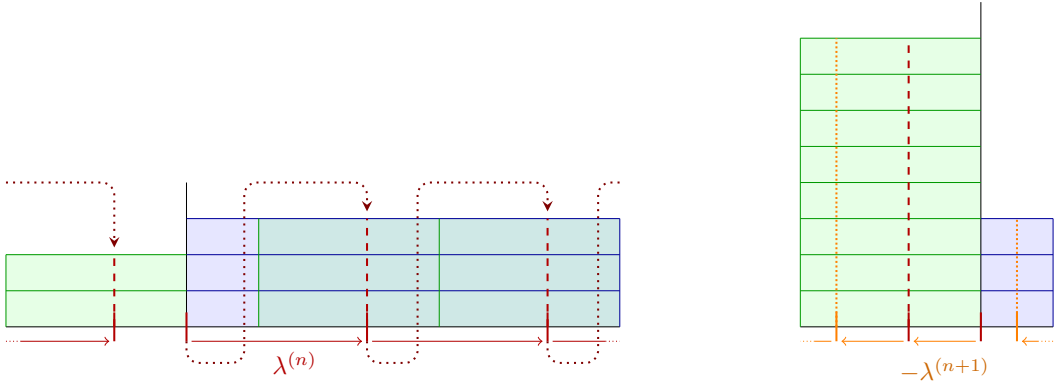

The time it takes for the pre-orbit $(R^{-j}_\alpha(0))_{j>0}$ of $0$ to visit every element of the partition $\xi^{(n)}$ is equal to $q_{n+2}$ (see Figure \ref{FigRenor2}). Indeed, it meets first any element $\Delta^{(n)}_j$ of the large tower $Z_l^{(n)}$ (defined by \eqref{EqTower}) exactly $\lfloor\frac{\lambda^{(n)}}{\lambda^{(n+1)}}\rfloor = a_{n+2}$ times, and then any element $\Delta^{(n+1)}_j$ of the small tower $Z_s^{(n)}$ once. The total time is thus equal to (see \eqref{EqTotTime})
\[
q_{n+1}\left\lfloor\frac{\lambda^{(n)}}{\lambda^{(n+1)}}\right\rfloor + q_n = q_{n+1}a_{n+2} + q_n = q_{n+2}.
\]
This finite pre-orbit is equal to the set of points in the tower $Z_l^{(n+1)}$ above the point of $\Delta^{(n+1)}$ within a distance $\lambda^{(n+2)}$ to 0.
\medskip

\label{Sectors}
To study Birkhoff sums, we will cut them into sums over ``sectors''. The tower $Z^{(n)}_l$ (of height $q_{n+1}$) can be decomposed into a ``basis'' of height $q_{n-1}$, which corresponds to $Z^{(n-1)}_s$, and $a_{n+1}$ groups of floors -- which we will call \emph{sectors} -- of heights $q_n$, made of the floors that project on the same interval of $\Delta^{(n-1)}$ (see Figure \ref{FigRenor1}).

\section{Some estimates}\label{SecTech}

In this whole section we fix $\alpha\notin \Q$ and use notations of the previous section about continued fractions.

For $y\in (0,1)$, we denote
\begin{equation}\label{EqDefPsi}
\psi_1(y) = \frac{1}{y},\quad \psi_2(y) = \frac{1}{1-y} \quad\text{and}\quad \psi(y) = \max\big(\psi_1(y), \psi_2(y)\big) .
\end{equation}
Remark that this implies that if $y$ is seen as an element of $\T$, then $\psi(y) = \|y\|^{-1}$, and moreover
\[\frac{\psi_1(y)+ \psi_2(y)}{2} \le \psi(y) \le \psi_1(y) + \psi_2(y).\]
In the sequel, we will use the notation $\psi(y)$ for $y\in\T$, by identifying the circle $\T$ with $[0,1)$.

For $y\in\T$, set (recall that $\psi$ is defined in \eqref{EqDefPsi})
\begin{equation}\label{EqDefS}
S(y) = \sum_{i=0}^{q_n-1} \psi\big(R_\alpha^{i}(y)\big)
\end{equation}
the Birkhoff sum over a sector.

\begin{lemma}\label{EqSellFinal}
Let $y\in \T$. We denote $y_0$ the point of the orbit $y,R_\alpha(y), \dots, R_\alpha^{q_n-1}(y)$ which is the closest to 0. Then,
\[S(y) \ge \psi(y_0) + \frac{\log q_n}{2 \lambda^{(n-1)}},\]
and
\[S(y) \le \psi(y_0) + \frac{4\log q_n}{\lambda^{(n-1)}}.\]
\end{lemma}

In the sequel, we will use repeatedly the following trivial fact, obtained by comparison with integral (for the second part, the comparison is done with the logarithmic integral function).

\begin{lemma}\label{LemSerHarmo}
For any $k_0 \ge 2$,
\[\sum_{k=k_0}^N \frac{1}{k} \ge \log\left(\frac{N+1}{k_0}\right) \qquad \text{and} \qquad \sum_{k=1}^N \frac{1}{k} \le \log(3N).\]

Moreover, there exists $C> 0 $ such that 
\[\sum_{k=2}^a \frac{1}{\log k} \leq \frac{C a}{\log a}\]
for every integer $a \geq 2$.
\end{lemma}

\begin{proof}[Proof of Lemma~\ref{EqSellFinal}]
Fix any point $y\in\T$, and consider its orbit $\mathcal O = \{y,R_\alpha(y),\dots,$ $ R_\alpha^{q_n-1}(y)\}$ of length $q_n$. We can denote $y_0$ the point of smallest norm of the orbit of $y$ of length $q_n$ and write $\cO^*=\cO \setminus \{y_0\}$. Note that $\| y \| =\| -y \|$ for every $y \in \T$ so, by swapping from $\cO$ to $-\cO = \{-y, -R_\alpha(y), \ldots, -R_\alpha^{q_n-1}(y) \}$ if necessary, we suppose that $y_0\in (1/2,1]$. 

Recall that Lemma~\ref{properties} says that the largest gaps in $\mathcal{O}$ are of size $\lambda^{(n)}+\lambda^{(n-1)}$. 
So if we write $\cO^*=\{y_1, \ldots, y_{q_n-1}\}$ with $0 < y_1 < y_2 < \ldots < y_{q_n-1} < 1$, we have $y_i< i(\lambda^{(n-1)}+\lambda^{(n)}) < 2i\lambda^{(n-1)} $. Hence
\begin{align*}
\sum_{i=0}^{q_n-1} \psi \big(R_\alpha^i(y)\big) & \geq \psi(y_0) + \sum_{i=1}^{q_n-1} \psi_1(y_i) \\
& \ge \psi(y_0) + \sum_{i=1}^{q_n-1} \psi_1\big(2i(\lambda^{(n-1)})\big).
\end{align*}

Using Lemma~\ref{LemSerHarmo}, one deduces that
\[\sum_{i=0}^{q_n-1} \psi \big(R_\alpha^i(y)\big) \ge  \psi(y_0) + \frac{\log(q_n)}{2\lambda^{(n-1)}}.\]
\medskip

We now turn to the second inequality. Following the same ideas, one gets that
\[y_i \ge (i-1)\lambda^{(n-1)} + \frac{\lambda^{(n-1)}}{2} \qquad \text{and} \qquad y_{q_n-j}\le 1-j\lambda^{(n-1)}. \]
and hence
\[S(y) \le \psi(y_0) + 2\sum_{i=1}^{q_n-1} \frac{2}{i\lambda^{(n-1)}},\]
and so, by Lemma~\ref{LemSerHarmo},
\[S(y) \le \psi(y_0) + \frac{4\log q_n}{\lambda^{(n-1)}}.\]
\end{proof}

As a direct consequence of Lemma~\ref{EqSellFinal}, we have the following result.

\begin{corollary} \label{kq orbit}
For every $n \geq 1$, every $k \geq 1$ and every $x \in \T$ we have
\[\sum_{i=0}^{k q_n-1} \psi\big(R_\alpha^i(x)\big) > \frac{k \log q_n}{2\lambda^{(n-1)}}.\]
\end{corollary}

\begin{lemma}\label{LemFinalMartin}
Let $\alpha = [a_0; a_1, a_2, \ldots]$ be an irrational number. Suppose that $x, y \in \T$ satisfy $\|x-y\| \leq   \lambda^{(n)}$ and let $j_0,j_1$ be such that 
\[\psi_1\big(R_\alpha^{j_0}(x)\big) = \max_{0 \leq j < q_n} \psi_1\big(R_\alpha^j(x)\big)
\qquad \text{and} \qquad
\psi_1\big(R_\alpha^{j_1}(y)\big) = \max_{0 \leq j < q_n} \psi_1\big(R_\alpha^j(y)\big).\]
Then
\begin{align*}
 \bigg| \sum_{j=0}^{q_n-1}\psi_1\big(R_\alpha^j(x)\big) & - \sum_{j=0}^{q_n-1}\psi_1\big(R_\alpha^j(y)\big) \bigg| \\
 &  \leq \big| \psi_1\big(R_\alpha^{j_0}(x)\big) - \psi_1\big(R_\alpha^{j_0}(y)\big) \big| +  \big| \psi_1\big(R_\alpha^{j_1}(x)\big) - \psi_1\big(R_\alpha^{j_1}(y)\big) \big| + \frac{\lambda^{(n)} q_n }{\lambda^{(n-1)}}.
 \end{align*}
\end{lemma}

A similar statement holds for $\psi_2$ instead of $\psi_1$. Note that the philosophy of the lemma recalls the cancellations of Sina\u{\i} and Ulcigrai in \cite{MR2478478}.

\begin{proof}
Let us first show that 
\begin{equation}\label{half lambda}
\psi_1\big(R_\alpha^{j}(x)\big) \leq \frac{1}{\lambda^{(n-1)}}
\end{equation}
for every $0 \leq j < q_n$ with $j\neq j_0$.

Suppose, for the sake of arriving at a contradiction, that we can find some $0 \leq j< q_n$ with $j \neq j_0$ such that $\psi_1(R_\alpha^{j}(x)) > 1/\lambda^{(n-1)}$, and write $x_0 = R_\alpha^{j_0}(x)$ and $x'=R_\alpha^j(x)$. Then, since $\psi_1(x_0) \geq \psi_1(x')$ we must also have $\psi_1(x_0) > 1/\lambda^{(n-1)}$. Hence, $x_0,x'\in(0,\lambda^{(n-1)})$ and so
\[\|x_0-x'\| < \lambda^{(n-1)}.\]
But this is absurd since $\lambda^{(n-1)}$ is the smallest distance between distinct points in any orbit of length $q_n$ (see \eqref{EqContFrac0}). This proves (\ref{half lambda}). Similarly, 
\[\psi_1\big(R_\alpha^{j}(y)\big) \leq \frac{1}{\lambda^{(n-1)}}\]
for every $0 \leq j < q_n$ with $j\neq j_1$.
\medskip

Hence, the set
\[\big\{R_\alpha^{j}(x): 0 \leq j< q_n, \ j\neq j_0\big\} \cup \big\{R_\alpha^{j}(y): 0 \leq j< q_n, \ j\neq j_1\big\} \]
is contained in the set 
\[X = \T \setminus \left(0,\, \lambda^{(n-1)}\right).\]
The function $\psi_1$ has Lipschitz constant $1/\lambda^{(n-1)}$ on $X$, so
\begin{align*}
\Big|S_{q_n}(x)-S_{q_n}(y) - \left( \psi_1(R_\alpha^{j_0}(x)) - \psi_1(R_\alpha^{j_0}(y)) \right) & - \left( \psi_1(R_\alpha^{j_1}(x)) - \psi_1(R_\alpha^{j_1}(y)) \right) \Big|\\
& = \Big| \sum_{\substack{0\leq j < q_n \\ j \neq j_0,j_1}} \psi_1\big(R_\alpha^j(x)\big)- \psi_1\big(R_\alpha^j(y)\big)\Big| \\
& \leq \sum_{\substack{0\leq j < q_n \\ j \neq j_0,j_1}}
\frac{1}{\lambda^{(n-1)}} \big\|R_\alpha^j(x)-R_\alpha^j(y)\big\| \\
& \leq \frac{\lambda^{(n)}(q_n-2)}{\lambda^{(n-1)}}.
\end{align*}
\end{proof}

\subsection*{Using comparison with a rational rotation}

\begin{lemma}\label{near rational orbit}
Let $\alpha = [a_0; a_1, a_2, \ldots]$ be an irrational number with convergents $p_n/q_n$. Suppose that for some given $n$, $a_{n+1} \geq 2$. Then there exists a bijection 
\[\sigma: \{0, \ldots, q_n-1\} \to \{0, \ldots, q_n-1 \} \]   
such that 
\[\left\| k \alpha - \frac{\sigma(k)}{q_n} \right\| < \lambda^{(n)} < \frac{1}{q_{n+1}}\] 
for every $ 0 \leq k < q_n$.
\end{lemma}

\begin{proof}
First we shall prove that, given any integer $0 \leq k < q_n$, there exists an integer $0 \leq \ell < q_n$ such that
\[\left\| k \alpha - \frac{\ell}{q_n} \right\| < \lambda^{(n)}.\]
Indeed, suppose there is some $k$ for which no such integer $\ell$ can be found. Then
\[\left|k\alpha - \frac{\ell}{q_n} + m \right| \geq \lambda^{(n)} \]
for every $0 \leq \ell < q_n$ and every $m \in \Z$. But then
\[ \left|k q_n \alpha - \ell + q_n m \right| \geq q_n \lambda^{(n)}\]
for every $0 \leq \ell < q_n$ and every $m \in \Z$. Hence
\[\left\|k q_n \alpha \right\| \geq q_n \lambda^{(n)}.\]
On the other hand, we have that
\[\left\|k q_n \alpha \right\| \leq k \left\|q_n \alpha \right\| = k \lambda^{(n)} \leq (q_n-1) \lambda^{(n)},\]
a contradiction.

Let $\sigma: \{0, \ldots, q_n-1\} \to \{ 0, \ldots, q_n-1\}$ be such that
\[\left\| k\alpha - \frac{\sigma(k)}{q_n} \right\| < \lambda^{(n)} \]
for every $0 \leq k < q_n$. We claim that $\sigma$ is injective; hence a bijection. Suppose it is not. Then there exist $0 \leq k_1 < k_2 < q_n$ such that $\sigma(k_1) =  \sigma (k_2)$. But then
\[\big \|(k_2 - k_1) \alpha \big\| \leq \left\|k_2(\alpha) - \frac{\sigma(k_2)}{q_n} \right\| + \left\|\frac{\sigma(k_1)}{q_n}-k_1 \alpha \right\| < 2 \lambda^{(n)}.\]
On the other hand, since $k_2-k_1 < q_n$, we know that $\|(k_2-k_1) \alpha \|$ must be at least $\lambda^{(n-1)}$. But, by \eqref{Eqaeta},
\[\lambda^{(n-1)} = a_{n+1} \lambda^{(n)} + \lambda^{(n+1)} \geq 2 \lambda^{(n)},\]
a contradiction. 
\end{proof}

\begin{lemma}\label{sumovernonzerorationals}

Let $q$ be a positive integer and let $0<\delta< \frac{1}{q}$. Let $x_1, \ldots, x_{q-1} \in \T$ be such that $\|x_k - \frac{k}{q}\| < \delta$ for every $1 \leq k \leq q-1$. Then
\[
\sum_{k=1}^{q-1} \frac{1}{\| x_k\|} \leq \frac{2q}{1-\delta q}\log(3q).\]
\end{lemma}

\begin{proof}
Let $0<\hx_1< \ldots < \hx_{q-1} < 1$ be representatives of $x_1, \ldots, x_{q-1}$ in the fundamental domain $[0,1)$ of $\T$. In order to prove the lemma, it suffices to prove that
\[\sum_{k=1}^{q-1} \psi_i(\hx_k) \leq \frac{q}{1-\delta q} \log(3 q) \]
for $i=1,2$. (Recall the definition of $\psi_1, \psi_2$ in (\ref{EqDefPsi})). 

By hypothesis, 
\[\hx_k > \frac{k}{q}-\delta = \frac{k}{q}\left( 1-\frac{ \delta q}{k} \right) \geq \frac{k}{q} \left( 1-\delta q \right) \]
for every $1 \leq k \leq q-1$. Hence, according to Lemma~\ref{LemSerHarmo}, we have
\[\sum_{k=1}^{q-1} \psi_1(\hx_k) < \sum_{k=1}^{q-1} \frac{1}{\frac{k}{q}(1-q \delta)} = \left( \frac{q}{1-\delta q}\right) \sum_{k=1}^{q-1} \frac{1}{k}<\frac{q}{1-\delta q} \log(3q).\]
 The estimate for $\psi_2$ is analogous.
\end{proof}

\begin{lemma} \label{LemSizeDiverg}
Let $n \geq 11$ and $0<\epsilon<1$. Suppose that  $a_{n+1} \geq 2$ and consider $x \in \T$ such that for some $0<i<q_n$ we have
\[\|x+i \alpha \| < \frac{ \epsilon }{q_n \log(3 q_n) }.\]
Then
\[\frac{6 \epsilon}{\|x+i \alpha \|} >  \sum_{k=0}^{i-1} \frac{1}{\|x+k\alpha\|}.\]
\end{lemma}

\begin{proof}
Let $\gamma = \epsilon/(q_n \log(3 q_n))$ and take $\delta = \lambda^{(n)} + \gamma$. 
From Lemma~\ref{properties} we have 
\[q_n \lambda^{(n)} < \frac{q_n}{q_{n+1}}< \frac{q_n}{q_n a_{n+1}} \leq \frac{1}{2}.\]

If $\alpha$ is the golden mean, then $q_{11} = 144$  (the $11^{th}$ element of the Fibonacci sequence) and $\log(3 q_{11}) > 6$. For every other value of $\alpha$ we have $q_{11} \geq 144$, whence it follows that $\log(3 q_n) >6$ for every $n \geq 11$. Hence $q_n \gamma < \frac{1}{6}$, so we have $q_n \delta < \frac{2}{3}$ and so
\begin{equation}\label{Eqdeltaqn}
\frac{1}{1-\delta q_n} < 3.
\end{equation}

Let $\sigma$ be as in Lemma~\ref{near rational orbit}. Let $\iota:\{0, \ldots q_n-1\} \to \{0, \ldots q_n -1\}$ be the natural involution given by $\iota(k) = [-k]_{q_n}$, where $[j]_{q_n}$ is the unique element in $\{ 0,\ldots, q_n-1 \}$ such that $[j]_{q_n} \equiv j \mod q_n$,
and denote by $\tilde{\sigma}$ the composition $\iota \circ \sigma$. Then, according to Lemma~\ref{near rational orbit} we have 
\[\left\| -k \alpha - \frac{\tilde{\sigma}(k)}{q_n} \right\| = \left\| \frac{\sigma(k)}{q_n} - k \alpha \right\| < \lambda^{(n)}\]
for every $0 \leq k < q_i$. 
By hypothesis, $0<i< q_n$ is such that $\|x+i\alpha \| < \gamma$. Thus
\[
\left\|x+(i-k)\alpha - \frac{\tilde{\sigma}(k)}{q_n} \right\|\leq \|x+i \alpha \|+ \left\|-k \alpha - \frac{\tilde{\sigma}(k)}{q_n} \right\| < \gamma + \lambda^{(n)} = \delta.
\]
Hence, denoting  $x_k= x+(i- \tilde{\sigma}^{-1}(k))\alpha$, we have 
\begin{equation}\label{EqInfDelta}
\left\|x_k-\frac{k}{q_n}\right\| \le \delta.
\end{equation}

But
\[\sum_{k=0}^{i-1} \frac{1}{\|x+k\alpha\| } = \sum_{k=1}^{i} \frac{1}{\| x + (i-k)\alpha \|} = \sum_{k=1}^i \frac{1}{\|x_{\tilde{\sigma}^{-1}(k)}\|} \leq \sum_{k=1}^{q_n-1} \frac{1}{\|x_k\|},\]
(the fact that $x_0$ does not appear in the last sum comes from the fact that $\tilde \sigma(0) = 0$),
and \eqref{EqInfDelta} allows us to apply Lemma~\ref{sumovernonzerorationals}: 
\[\sum_{k=0}^{i-1} \frac{1}{\|x+k\alpha\| }\leq \frac{2q_n}{1-\delta q_n}\log(3q_n).\]
Combining it with \eqref{Eqdeltaqn} we get
\[\sum_{k=0}^{i-1} \frac{1}{\|x+k\alpha\| }\leq 6 q_n \log(3 q_n),\]
hence
\[ \frac{6 \epsilon}{\|x+i\alpha\|} > 6 q_n \log(3 q_n)  \geq  \sum_{k=0}^{i-1} \frac{1}{\| x+k \alpha \|} \]
as required.
\end{proof}

\begin{lemma}\label{sumoverrationals}
Let $q$ be a positive integer, $A>0$, and $\beta \in \T$ such that 
\[\left\| \frac{n}{q} - \beta \right\| >\frac{A}{q} \quad \forall \ 0 \leq n < q.\]  Then
\begin{equation*}
 \sum_{n=0}^{q-1} \frac{1}{\| \frac{n}{q} - \beta \| } < 2 q \big(A^{-1}+\log( A^{-1} q)\big).
\end{equation*}
\end{lemma}

\begin{proof}
Let $X = \{ \frac{n}{q}- \beta \mod 1: 0 \leq n < q \} \subset (0,1)$ and write $X = \{ x_1, \ldots, x_q \}$ in such a way that 
\[ x_1 < x_2 < \ldots < x_q.\]

By hypothesis we have that $x_1\geq \frac{A}{q}$ and $x_q \leq 1-\frac{A}{q}$. Note also that 
\[x_{k+1}-x_{k} = \frac{1}{q} \] 
for every $1 \leq k < q$. Then, by \eqref{EqDefPsi},
\begin{equation}\label{fromtwosides}
\sum_{n=0}^{q-1} \frac{1}{\|\frac{n}{q}- \beta\|} \leq \sum_{k=1}^q \psi_1(x_k) + \psi_2 (x_k).
\end{equation}

Note that $\frac{1}{q} \sum_{k=2}^q \psi_1(x_k)$ is a lower Riemann sum for $\int_{x_1}^{x_q} \psi_1(t) \ dt$. Hence (as $x_2\ge 1/q$)
\[\sum_{k=1}^q \psi_1(x_k) < \psi_1(x_1) + q \int_{x_1}^{x_q} \frac{dt}{t} < A^{-1} q+\log\left(\frac{x_q}{x_1}\right) < q(A^{-1} + \log( A^{-1} q )).\]

Similarily, we have
\[\sum_{k=1}^q \psi_2(x_k) < q(A^{-1} + \log(A^{-1} q)),\]
and the proof follows from (\ref{fromtwosides}).
\end{proof}

Recall that our goal is to get bounds over the quantity 
$\Theta_n^\beta(x) = S_n(x)/S_n(x-\beta)$ (see \eqref{DefTheta}). It amounts to bound from above/below both $S_n(x)$ and $S_n(x-\beta)$, which will be done in Lemmas~\ref{ABC} and \ref{lower bound}. 

For each non-negative $n$ and positive $\ell < q_{n+1}/q_n$, denote by $E_{n,\ell}$ the $\frac{\lambda^{(n)}}{2}$-neighbourhood of the orbit $\cO(\ell q_n)$, that is,
\begin{equation}\label{DefE}
E_{n, \ell} = \bigcup_{k=0}^{\ell q_n - 1} R_\alpha^{-k}(I_n).
\end{equation}
where $I_n = \big(-\frac{\lambda^{(n)}}{2},\frac{\lambda^{(n)}}{2} \big)$.

\begin{lemma}  \label{size}
One has $\la(E_{n, \ell})=\ell q_n \lambda^{(n)} $.
\end{lemma}

\begin{proof}
Note that $\la(R_\alpha^{-k}(I_n)) = \la(I_n) = \lambda^{(n)}$ and that $\sum_{k=0}^{\ell q_n-1} \la(R_\alpha^{-k}(I_n)) = \ell q_n \lambda^{(n)}$, so the proof follows if we can show that the sets $R_\alpha^{-k}(I_n)$, $k  = 0 , \ldots , \ell q_n-1$ are pairwise disjoint. Suppose they are not. Then there exist $0\leq k < \ell < q_{n+1}$ such that $\| k\alpha - \ell \alpha \|< \lambda^{(n)}$. Writing $m = |k-\ell|$ that gives us $\|m \alpha \| < \lambda^{(n)}$ for some $m< q_{n+1}$. But that is absurd, since $q_{n+1}$ is the smallest number with this property.
\end{proof}

\begin{lemma} \label{ABC}
Fix some $n>0$ and suppose that there are numbers $0<B < A < \half$, a positive integer $\ell$ and some $\beta \in \T$ such that 
\[\left\| \frac{i}{q_n}-\beta \right\|  \geq \frac{A}{q_n}\]
for every $0 \leq i < q_n$, and that
\[\ell+1  \leq B a_{n+1}.\]

Then, given any $x \in E_{n, \ell}$, we have
\[\sum_{k=0}^{\ell q_n - 1} \frac{1}{\|x+k\alpha - \beta \|} < \frac{2 \ell q_n ( A^{-1} + \log( A^{-1} q_n))}{1-B/A} .\] 
\end{lemma}

\begin{proof}
Note that
\[\sum_{k=0}^{\ell q_n -1} \frac{1}{\|x + k \alpha - \beta \|} 
= \sum_{r=0}^{\ell-1} \sum_{s=0}^{q_n-1} \frac{1}{\|x + (r q_n+s) \alpha - \beta \| }.\]
Thus it suffices to show that 
\[ \sum_{s=0}^{q_n-1} \frac{1}{\|x + (r q_n+s) \alpha - \beta \| } < \frac{2 q_n ( A^{-1} + \log( A^{-1} q_n))}{1-B/A}\]
for every $0 \leq r < \ell$.

Fix some $x \in E_{n, \ell}$. Then, by the definition of $E_{n, \ell}$ there exists $0 \leq k < \ell q_n$ such that $\|x+k \alpha \| < \lambda^{(n)} /2$. Let $0 \leq c < \ell$ and $0 \leq d < q_n$ be such that $k = c q_n + d$. Let 
\[\sigma: \{0, \ldots, q_n-1 \} \to \{0, \ldots, q_n-1\}\]
be as in Lemma~\ref{near rational orbit}. (The inequality $\ell+1 < B a_{n+1}$ implies that $a_{n+1} >4$ so that Lemma~\ref{near rational orbit} applies.) Given an integer $i$, let $[i]_{q_n}$ denote the unique integer $0 \leq m < q_n$ such that $i \equiv m \mod q_n$. Then, for every $0 \leq r < \ell$ and $0 \leq s < q_n$ we have 
\begin{multline*}
\big\| x + (r q_n + s) \alpha - \beta \big\|  
 = \big\|  x + (s-d) \alpha +(r-c) q_n \alpha + k \alpha - \beta \big\| \\
 \geq \left\|  \frac{\sigma([s-d]_{q_n})}{q_n} - \beta \right\| 
 - \left\|  (s-d) \alpha- \frac{\sigma([s-d]_{q_n})}{q_n} 
+(r-c) q_n \alpha + x + k \alpha \right\|.
\end{multline*}

By hypothesis,
\[ \left\|  \frac{\sigma([s-d]_{q_n})}{q_n} - \beta \right\|  
> \frac{A}{q_n} > A \lambda^{(n-1)}\]
for every $0 \leq s < q_n$.

Moreover, using the inequality $|r-c|\leq \ell-1$ and Lemma~\ref{near rational orbit},
\begin{align*}
 \left\|  (s-d) \alpha- \frac{\sigma([s-d]_{q_n})}{q_n} 
+(r-c) q_n \alpha + x + k \alpha\right\| \leq & \left\|  (s-d) \alpha- \frac{\sigma([s-d]_{q_n})}{q_n} \right\| \\
& + |r-c| \|q_n \alpha \| + \| x+k \alpha \| \\
\leq &\ \lambda^{(n)} + |r-c| \lambda^{(n)} + \frac{\lambda^{(n)}}{2} \\
\leq &\ (\ell+\half) \lambda^{(n)} < B a_{n+1} \lambda^{(n)}\\
< &\ B \lambda^{(n-1)}.
\end{align*}

Consequently,
\[ \big\| x + (r q_n + s) \alpha - \beta \big\| 
> (1-B/A) \left\|  \frac{\sigma([s-d]_{q_n})}{q_n} - \beta \right\|.\]
Taking reciprocals while summing over $s$ and applying Lemma~\ref{sumoverrationals} gives
\begin{align*}
\sum_{s=0}^{q_n-1} \frac{1}{\| x + (r q_n + s) \alpha - \beta \|} & < \frac{1}{1-B/A}  \sum_{s=0}^{q_n-1} \frac{1}{\| \frac{\sigma([s-d]_{q_n})}{q_n} - \beta \| }  \\
 & = \frac{1}{1-B/A} \sum_{j=0}^{q_n-1} \frac{1}{\| \frac{j}{q_n} - \beta \|} \\
&  < \frac{2 q_n( A^{-1} + \log(A^{-1}q_n))}{1-B/A},
\end{align*}
as required.
\end{proof}

\begin{lemma} \label{lower bound}
Let $\alpha\notin \Q$. Fix some $n>0$ and let $\ell \geq 1$ be such that $\ell q_n < q_{n+1}$. Let $E_{n, \ell}$ be as in (\ref{DefE}).
Then, for any $x\in E_{n, \ell}$,
\[\sum_{k=0}^{\ell q_n-1} \frac{1}{\|x + k \alpha \|} \geq \frac{ \log \ell}{\lambda^{(n)}}.\]
\end{lemma}

\begin{remark}\label{lower bound2}
Replacing the interval $I_n$ by $\tilde I_n = (-2\lambda^{(n)},2\lambda^{(n)})$, and the set $E_{\ell,n}$ by $\tilde E_{\ell,n}$ accordingly, one gets a similar result:
\[\sum_{k=0}^{\ell q_n-1} \frac{1}{\|x + k \alpha \|} \geq \frac{ \log \ell}{4\lambda^{(n)}}.\]
\end{remark}

The proof is based on taking into account only the contribution of points of the ``ground floor'' of the renormalization interval.

\begin{proof}
Fix some  $x \in E_{n, \ell}$. Then, by definition of $E_{n, \ell}$, there exists some $0\leq m < \ell q_n$ such that $\|x+m \alpha \|< \frac{\lambda^{(n)}}{2}$. Let $0\leq c<\ell$ and $0 \leq d < q_n$ be integers such that $m=c q_n + d$. Then
\begin{align*}
 \sum_{k=0}^{\ell q_n - 1} \frac{1}{\|x + k \alpha \|} & = \sum_{r=0}^{\ell-1} \sum_{s=0}^{q_n-1}  \frac{1}{\| x + (r q_n+s) \alpha  \|} \\
 & > \sum_{r=0}^{\ell-1} \frac{1}{\|(x+ (r q_n+d) \alpha\|} \\
& = \sum_{r=0}^{\ell-1} \frac{1}{\|x' + (r-c) q_n \alpha \|},
\end{align*}
where $x' = x+m \alpha$. Note that 
\[\| x' + (r-c) q_n \alpha  \| \leq \|(r-c)q_n \alpha \| + \|x'\| < |r-c| \lambda^{(n)} + \frac{\lambda^{(n)}}{2}.\]
Hence 
\begin{equation}\label{riemannsum}
\sum_{r=0}^{\ell-1} \frac{1}{\| x' + (r-c) q_n \alpha  \|}  \geq \sum_{r=0}^{\ell-1} \frac{1}{\lambda^{(n)} ( |r-c|+ \half)} 
 \geq \sum_{r=0}^{\ell-1} \frac{1}{\lambda^{(n)}(r+\frac{1}{2})}.
\end{equation}

From that the lemma follows easily, since (\ref{riemannsum}) is an upper Riemann sum of the integral
\[\int_{\frac{1}{2}}^{\ell+\frac{1}{2}} \frac{dx}{\lambda^{(n)} x}, \]
whose value is greater than $ \frac{\log \ell}{\lambda^{(n)}}$.
\end{proof}

We end this section by a lemma that will be used in the next one.

\begin{lemma}\label{smallest distance}
Let $a, b, q$ be positive integers, with $b \geq 2$. Suppose that $a$ and $b$ are coprime and also that $b$ and $q$ are coprime. Then
\begin{equation*}
\left\| \frac{n}{q}- \ab \right\| \geq \frac{1}{bq}
\end{equation*}  
 for every $n \in \Z$.
\end{lemma}

\begin{proof}
It follows from $\gcd(a,b) = \gcd(b,q) = 1$ that $a q \not\equiv 0 \mod b$. Thus 
\[b(n+mq)-aq \neq 0\] 
for every $n, m \in \Z$, and hence
\[ \left| \frac{n}{q} - \frac{a}{b} +m \right|  = \left| \frac{b(n+mq)-aq}{bq} \right| \geq \frac{1}{bq}, \]
proving the lemma. 
\end{proof}

\section{Extreme historic behaviour}
\label{SecLiou}

This section is devoted to two theorems on the existence of reparameterized linear flows with extreme historic behaviour. Theorem~\ref{refined rational distinct orbits} deals with stopping points on rationally separated orbits whereas Theorem~\ref{generic distinct orbits} deals with the generic case.

\subsection{Precise statements and sketch of proofs}

We start by identifying the set of angles for which we are going to prove that the conclusion of Theorem~\ref{rational distinct orbits} holds before stating a more precise version of it.

\begin{definition}
Let $\nu$ be a positive number, $k \geq 2$ be an integer, and $\alpha = [a_0; a_1, \ldots]$ an irrational number with convergents $p_n/q_n$. We say that $\alpha$ is \emph{$(\nu, k)$-approximable} if there are infinitely many $n \in \N$ such that 
\[ 
\begin{cases}
a_{n+1} \geq q_n^\nu, \text{ and} \\
\gcd(q_n, k) = 1.
\end{cases}
\]
Let $\cW(\nu,k)$ denote the set of numbers that are $(k, \nu)$-approximable. Let  
\[\cW(\nu) = \bigcap_{k\geq 2} \cW(\nu,k) \]
and 
\[\cW = \bigcup_{\nu>0} \cW(\nu).\]
\end{definition}

\begin{proposition}\label{W is generic}
For every $\nu>0$ and every integer $k \geq 2$, the set $\cW(\nu,k)$ is a dense $G_\delta$ subset of $\R$.
\end{proposition}

Since $\cW(\nu,k)$ is a dense $G_\delta$ set, so is $\cW(\nu)$ for every $\nu$. 

The proof of Proposition~\ref{W is generic} is a straightforward $G_\delta$ argument, but it relies on the fact that we may make small alterations to $\alpha$ to obtain $\gcd(q_n,k) =1$ for large $n$. The following lemma ensures that this is possible. 

\begin{lemma} \label{pqk}
Let $a,b,c$ be positive integers such that $a$ and $b$ are coprime. Then there exists a positive integer $i$ such that $a+ib$ and $c$ are coprime.
\end{lemma}




\begin{proof}
Let $i$ be the product of all prime factors of $c$ that do not divide $ab$, if such factors exist. Otherwise let $i=1$. Note that $a$, $b$, and $i$ have no common factors, and that every prime factor of $c$ divides $abi$. We claim that $c$ and $a+ib$ are coprime. Indeed, suppose that $p$ is a prime factor of $c$. Then $p \mid abi$. If $p$ is a factor of $a$ then $p$ is not a factor of $bi$. In particular $p \nmid a+ib$. If $p$ is not a factor of $a$, then $p$ is a factor of $bi$. Here again $p \nmid a+ib$. We have shown that there is no prime number that divides both $c$ and $a+ib$.
\end{proof}

\begin{proof}[Proof of Proposition~\ref{W is generic}]

Fix $\nu>0$ and an integer $k \geq 2$.
Let us denote by $C(a_0; a_1, \ldots, a_n)$ the open cylinder set
\[
\left\{ 
  a_0 + 
    \frac{1}{
      a_1+ \frac{1}{
         \ddots \  +\frac{1}{a_n+y}
            }}: 0<y<1 \right\} \]
     
Let $\cC_n$ be the collection of all cylinders of the form $C(a_0;a_1, \ldots, a_n)$. 
If $\alpha$ and $\alpha'$ belong to the same cylinder $C(a_0;a_1, \ldots, a_n)$, then 
\[ \|\alpha-\alpha'\| \leq \|\alpha-\frac{p_n}{q_n}\| + \| \alpha'-\frac{p_n}{q_n}\| < \frac{2}{q_{n+1}} \leq 2 \cdot 2^{-\lfloor n/2 \rfloor}\]
so that the diameter of cylinders in $\cC_n$ converge to zero uniformly as $ n \to \infty$. 

Let $\cA_n$ be the collection of cylinders on which $\gcd(q_n,k)=1$ is satisfied. (Note that $q_n=q_n(\alpha)$ is constant on cylinders in $\cC_n$.) We claim that any open set in $\R$ contains an element of $\cA_n$. Indeed, sine the diameter of cylinders in $\cC_n$ tend uniformly to zero, any open set contains a cylinder in $\cC_{n-1}$, $C(a_0; a_1, \ldots, a_{n-1})$ say, for $n$ sufficiently large. Now, by Lemma \ref{pqk}, we may choose a number $a_n \geq 1$ such that $q_n = q_{n-2}+ a_n q_{n-1}$ and $k$ are coprime. Hence $C(a_0; a_1, \ldots, a_n) \in \cA_n$. 
We have proved that 
\[\bigcup_{n\geq m} \bigcup_{C \in \cA_n}C\]
is dense in $\T$ for every $m$.

Let $\cB_{n+1}$ be the collection of cylinders of the form $C(a_0; a_1, \ldots, a_{n+1})$ such that
\begin{enumerate}
\item $C(a_0;a_1, \ldots, a_n) \in \cA_n$, and
\item points in $C(a_0, a_1, \ldots, a_{n+1})$ satisfy $a_{n+1} \geq q_n^\nu$. 
\end{enumerate}
It is clear that if $C(a_0;a_1, \ldots,a_n)$ belongs to $\cA_n$ then $C(a_0; a_1, \ldots, a_n, \ell)$ belongs to $\cB_{n+1}$ for $\ell$ sufficiently large. In other words, each $C \in \cA_n$ contains a subcylinder  $C' \in \cB_{n+1}$. Consequently
\[O_m = \bigcup_{n\geq m} \bigcup_{C \in \cB_n} C\]
is dense in $\T$. The proof follows by observing that
\[\cW(\nu,k) = \bigcap_{m} O_m.\]
\end{proof}

\begin{theorem}[refined Theorem~\ref{rational distinct orbits}]\label{refined rational distinct orbits}
Let $\p=(0,0)$ and $\q = (0, \frac{a}{b})$, where $\gcd(a,b)=1$, and suppose that $\alpha \in \cW(\nu,b)$ for some $\nu>0$. If $\phi^t$ is a reparametrized linear flow satisfying (SH), then it has an extreme historic behaviour.
\end{theorem}

We now turn to the generic case.

\begin{definition}
We say that $\alpha = [a_0; a_1, a_2, \ldots ]$ is a Liouville number if, given any $k>0$  there are infinitely many values of $n$ for which
\[a_{n+1} > q_n^k\]
holds. 
\end{definition}

\begin{remark}
The above definition Liouville number is stated in a form suited for the needs in this this paper. Is distinct from, but equivalent to, the standard definition that $\alpha$ is Liouville if, given any positive integer $k$, there exist $p,q \in \Z$ such that
\[\left| \alpha - \frac{p}{q} \right| < \frac{1}{q^k}.
\]
\end{remark}

\begin{theorem} \label{generic distinct orbits}
Let $\alpha$ be a Liouville number. Then there exists a dense $G_\delta$ set $B \subset  \T$ such that if $\beta \in B$, then any reparameterized linear flow satisfying (SH) with angle $\alpha$ and stopping points at $(0,0)$ and $(0,\beta)$ has an extreme historic behaviour.
\end{theorem}

The proofs of Theorems~\ref{refined rational distinct orbits} and \ref{generic distinct orbits} are based on some rather simple ideas, as we now shall explain. 
 
Propositions \ref{criterium2} and \ref{criterium3} tell us that in order to detect an extreme historic behaviour for a reparameterized flow with stopping points at $\p=(0,0)$ and $\q=(0, \beta)$, we must show that the ratio between $S_m(x)$ and $S_m(x-\beta)$ can be made larger (or smaller) than an arbitrary constant on a set of substantial measure. We now explain how this is done.

Consider a situation in which $\alpha$ has two successive convergents $p_n/q_n$ and $p_{n+1}/q_{n+1}$ such that $q_{n+1}$ is very large compared to $q_n$. Then the orbit $\cO(q_n)$ is very close to the set $\{k/q_n: 0 \leq k < q_n \}$ (Lemma~\ref{near rational orbit}). Suppose that $\beta$ happens to lie at a safe distance ($> A \lambda^{(n-1)} \approx A/ q_n$) from this orbit. Then $R_\beta(\cO(q_n))$ --- the orbit of length $q_n$ starting from $\beta$ --- is intertwined with that of $\cO(q_n)$ so that each gap of one orbit contains exactly on point of the other orbit and vice versa. Moreover, points of the two orbits are separated from one another by a distance of order $A \lambda^{(n-1)}$.

One way to guarantee that $\beta$ is on a safe distance from the orbit $\cO(q_n)$ is to take $\beta$ to be rational of the form $\ab$, and ask for $q_n$ and $b$ to be coprime (Lemma~\ref{smallest distance}). This is why we take $\alpha \in \cL_b$ in Theorem~\ref{refined rational distinct orbits}. Another way is to simply move $\beta$ so that it lies more or less in the middle of one of the gaps defined by the orbit $\cO(q_n)$. This is the idea exploited in the proof of Theorem~\ref{generic distinct orbits}.

Now, since $q_{n+1}$ is much larger than $q_n$ we have that (see Lemma~\ref{properties}) $q_{n+1} \approx a_{n+1} q_n$ and also that $\lambda^{(n+1)}  \approx a_{n+1} \lambda^{(n)}$ (see \eqref{Eqaeta}). Let $\ell$ be some integer approximately equal to $B a_{n+1}$ for some fixed $0<B<A$. Then the orbit $\cO(\ell q_n)$ is the union of $q_n$ small ``blocks'', each of length $\ell$ (Equation\eqref{EqOrbit}). If $B$ is not too big, the orbits $\cO(\ell q_n)$ and $R_\beta(\cO(\ell q_n))$ are still on a safe distance from one another. Figure~\ref{orbits} illustrates this in a situation where $q_n=7$ and $\ell = 4$.

\begin{figure}\label{orbits}
\includegraphics[trim=90pt 90pt 70pt 70pt, clip, scale=0.5]{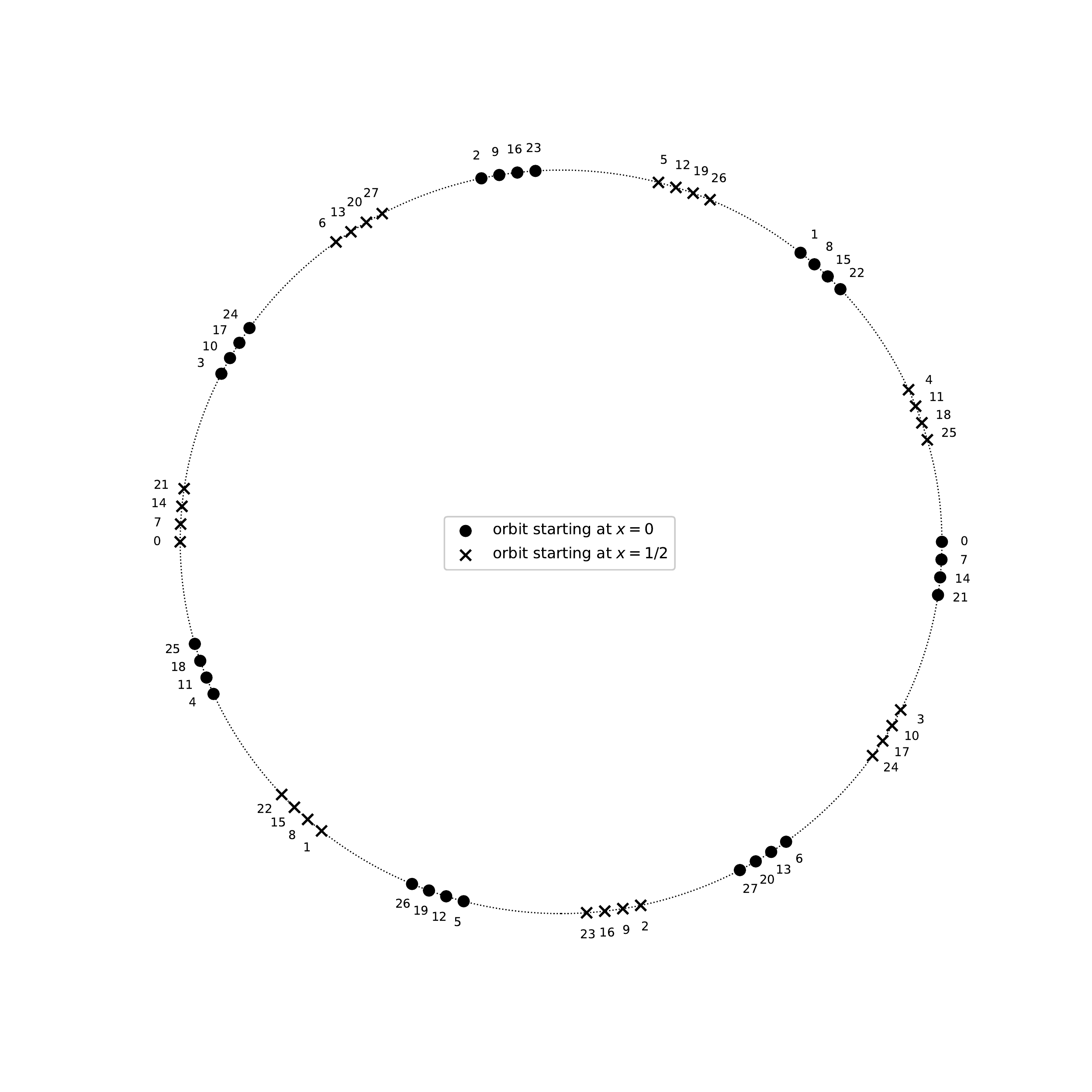} 
\caption{Illustration of an orbit (dots) of length $\ell q_n$ where $q_n= 7$ and $\ell = 4$, together with its rotation (crosses). Each orbit has seven blocks of four points each. The distance between corresponding points in adjacent blocks is approximately $1/q_n$, whereas the distance between points within the same block is approximately $1/q_{n+1}$. In this figure, $q_{n+1}$ is about ten times larger than $q_n$ so that each block of lenght four 'fills up' nearly half the gap between points of the orbit of length $q_n$. The proof of theorem \ref{refined rational distinct orbits} requires the ration between $q_{n+1}$ and $q_n$ to be larger than some fixed positive power of $q_n$ for infinitely many $n$, while the proof of theorem \ref{generic distinct orbits} requires the same ratio to be larger than any power of $q_n$.}
\end{figure}

Let $E_{n, \ell}$ be an $\lambda^{(n)}/2$-neighbourhood of $\cO(\ell q_n)$. Then, modulo a finite number of points
, $E_{n, \ell}$ is a disjoint union of $q_n$ intervals of size $\approx B \lambda^{(n-1)} \approx B /q_n$, each interval corresponding to a 'block'. Now let us rotate this set $E_{n, \ell}$ by the angle $-\ell q_n \alpha$. The resulting set, $E_{n, \ell}'$ is then a $\lambda^{(n)}/2$-neighborhood of the pre-orbit of length $\ell q_n$ of the point $0$. Thus a point $x$ belongs to $E_{n, \ell}'$ if and only if one of its first $\ell q_n$ iterates lies within an $\lambda^{(n)}/2$-distance from $0$. In this case Lemma~\ref{lower bound} tells us that  $S_m ( x) $ is at least of order $(\log \ell) / \lambda^{(n)} \approx  a_{n+1} q_n \log \ell $. 
On the other hand, $S_m(x-\beta)$ is at most of order $B a_{n+1} q_n \log(q_n)$ (Lemma~\ref{ABC}). Thus in order to have $S_m(x)$ larger than, say, $K S_m(x-\beta)$, we impose the condition that $a_{n+1}$ (and hence $\ell$) be of order $q_n^{BK}$. 

Since $E_{n, \ell} '$ consists of $\ell q_n $ disjoint intervals of length $\lambda^{(n)}$, its $\la$-measure is equal to $\ell q_n \lambda^{(n)} \approx B a_{n+1} q_n \lambda^{(n)} \approx B$. In the proof of Theorem~\ref{generic distinct orbits}, since $\alpha$ is Liouville, $a_{n+1}$ will be larger than $q_n^{BK}$ infinitely often, whatever the value of $BK$. Hence the value of $B$ is uniformly bounded away from zero (i.e. does not depend on $K$). On the other hand, in Theorem~\ref{refined rational distinct orbits}, $BK$ has to be of order $\nu$, so $B$ needs to be taken smaller as we increase $K$.

\subsection{Proof of Theorem~\ref{refined rational distinct orbits}}

We fix some $\beta \in \Q \setminus \Z$ and write $\beta = \ab$, with $\gcd(a,b)=1$. Fix also some $\nu>0$ and $\alpha \in \cW(\nu,b)$.

We shall prove that, given any $K>1$ and  $n \in \N$, there exists an integer $m \geq n$ and a set $E_{n, \ell} \subset \T$ with $\la(E_{n, \ell}) \geq \nu/(64bK)$ such that 

\begin{equation}\label{dominance1}
\sum_{k=0}^{m-1}\frac{1}{\|k \alpha + x \|} > K \sum_{k=0}^{m-1} \frac{1}{\|k \alpha+x-\ab \|}
\end{equation} 
holds for every $x \in E_{n, \ell}$. Then $\phi^t$ has an extreme historic behaviour according to Proposition~\ref{criterium3}.

Fix $K>1$ and $N \in \N$. Upon possibly increasing $K$ we can (and do) suppose that $K> \nu/8$. Since $\alpha \in \cW(\nu,b)$ there exists $n\geq N$ for which 
\[a_{n+1}>q_n^\nu\]
and $\gcd(q_n,b)=1$. Pick such $n$, with the additional property that 
\begin{equation}\label{large q_i}
q_n^{\nu/2} > \frac{16bK e^{\nu/2}}{\nu} \qquad \text{and}\qquad q_n^\nu > \frac{32 bK}{\nu}.
\end{equation} 
Therefore we can choose an integer $\ell \geq 1$  such that 
\[2<\frac{\nu a_{n+1}}{16bK} < \ell < \ell+1 \leq \frac{\nu a_{n+1}}{8bK}.\]

Let $E_{n, \ell}$ be as in (\ref{DefE}). We claim that (\ref{dominance1}) holds for $m=\ell q_n$ and any $x \in E_{n, \ell}$. 

Lemma~\ref{smallest distance} tells us that 
\[\left\|\frac{i}{q_n}- \ab \right\| \geq \frac{1}{b q_n}\]
for every integer $i$.

We can therefore apply  Lemma~\ref{ABC} with $A=1/b$, $B=\nu/(8bK)$. Doing so gives (Recall that we are assuming that $K>\nu/4$.)
\begin{align*}
\sum_{k=0}^{\ell q_n-1} \frac{1}{\| x+k\alpha - \ab \|} 
& < \frac{2\ell q_n}{1-\frac{\nu}{8K}}(b + \log(bq_n))\\
& < 4 \ell q_n (b+\log(bq_n))
\end{align*}
for every $x \in E_{n, \ell}$. 

Moreover, as $K>\nu/8>\nu/(8b)$), we have
\[\ell q_n < \frac{\nu a_{n+1}q_n}{8bK} \le \frac{\nu q_{n+1}}{8bK} < q_{n+1}.\]
We may therefore  apply Lemma~\ref{lower bound}, obtaining the estimate
\[\sum_{k=0}^{\ell q_n-1} \frac{1}{\| x+k\alpha\|} \geq \frac{\log \ell}{\lambda^{(n)}}.\]
Thus in order to show (\ref{dominance1}), it suffices to show that 
\[\log \ell >   4 K \lambda^{(n)} \ell q_n (b+\log(bq_n)).\]
But by Lemma~\ref{properties},
\[ \ell \lambda^{(n)} q_n < \frac{\nu a_{n+1} \lambda^{(n)} q_n}{8bK} < \frac{\nu \lambda^{(n-1)} q_n}{8bK} < \frac{\nu}{8bK}, \]
so for $q_n$ large enough
\[4K\lambda^{(n)} \ell q_n\big(b+\log(b q_n)\big) < \frac{\nu}{2b}\big(b+\log(bq_n)\big) < \frac{\nu}{2}\big(1+\log(q_n)\big).\]
Hence it suffices to show that 
\[ \frac{\nu}{2} \big(1+\log(q_n)\big) <  \log \ell. \]
But $\ell$ was chosen so that 
\[\ell >   \frac{\nu a_{n+1}}{16bK}> \frac{\nu}{16bK} q_n^\nu
> e^{\nu/2} q_n^{\nu/2}\]
in view of (\ref{large q_i}). We have therefore shown that (\ref{dominance1}) holds for $m=\ell q_n$ whenever $x \in E_{n, \ell}$.

It remains to show that $\lambda(E_{n, \ell}) \geq \nu/(64 b K)$. Applying Lemma~\ref{size} to the set $E_{n, \ell}$ we see that 
\[\la(E_{n, \ell}) = \ell q_n \lambda^{(n)} \geq \frac{\nu a_{n+1}}{16bK} \lambda^{(n)} q_n  > \frac{\nu \lambda^{(n-1)}}{32 bK} q_n > \frac{\nu}{64bK}.\]
 This completes the proof.

\subsection{Proof of Theorem~\ref{generic distinct orbits}}

Fix some Liouville number $\alpha$. 
Let 
\[ C(n,K,\beta)= \big\{ x \in \T: \Theta_n^\beta (x)> K \big\},\]
and
\[D(n,K) = \big\{ \beta \in  \T: \la(C(n,K, \beta))>1/16 \big\}. \]

Clearly, the sets $D(n,K)$ are open. Let 
\[\cR = \bigcap_{K > 1} \bigcup_{n \geq 1} D(n,K).\]
Then according to Proposition~\ref{criterium2}, any reparameterized flow satisfying (SH) with $ \beta  \in  \cR $ has an extreme historic behaviour. Thus in order to prove Theorem~\ref{generic distinct orbits} it suffices to prove that $\bigcup_{n \geq 1} D(n,K)$ is dense in $\T$ for every $K>1$.

To this end, we fix $K>1$, $\beta_0 \in \T$ and $\epsilon>0$ arbitrarily. We shall prove that there is some $m \geq 1$ and $ \beta \in \T$ with  $|\beta-\beta_0| < \epsilon$ such that $\beta \in D(m, K)$. 

Write $\alpha$ as $[a_0; a_1, a_2, \ldots]$ and let $p_n/q_n$ be its convergents. Then fix some $n \in \N$ such that $1/q_n < \epsilon$ and $a_{n+1} > q_n^{K+1}$. Upon possibly increasing $n$ we can (and do) assume that $q_n > n$ and also that  $q_n >  8 e^{2K} 2^K$ (which in particular is larger than $16$). 

Chose an integer $0 \leq b < q_n$ such that 
\[\left\|\beta_0 - \left(\frac{b}{q_n} + \frac{1}{2 q_n} \right) \right\| \le \epsilon \]
and let 
\[ \beta = \frac{b}{q_n} + \frac{1}{2 q_n} .\]

Since $a_{n+1} > 16$ we can (and do) fix $ \ell \in \N$ such that
\[\frac{a_{n+1}}{8} < \ell< \ell+1 < \frac{a_{n+1}}{4}.\]
Let $m = \ell q_n$. We claim that  $ \beta \in D(m,K)$. More specifically, let $E_{n,\ell}$ be as in (\ref{DefE}). We shall prove that  $E_{n, \ell} \subset C(m, K, \beta)$ and that $\la(E_{n, \ell}) \geq 1/16$. Indeed, the latter follows from Lemma~\ref{size}, our choice of $\ell > a_{n+1}/8$ and the inequality $a_{n+1} \lambda^{(n)} \geq \lambda^{(n-1)}/2$. 

Now fix $x \in E_{n, \ell}$. We shall prove that 
\begin{equation}\label{dominance4}
\sum_{k=0}^{m-1}\frac{1}{\|k \alpha + x \|} > K \sum_{k=0}^{m-1} \frac{1}{\|k \alpha+x-\beta \|}.
\end{equation}

Our choice of $\beta$ is such that 
\[ \left\| \frac{i}{q_n} - \beta \right\| \geq \frac{1}{2 q_n}  \]
for every $0 \leq i < q_n$. 

We can therefore apply Lemma~\ref{ABC} with $A = 1/2$ and $B=1/4$. Doing so gives us the estimate
\[K \sum_{k=0}^{m-1} \frac{1}{\|k \alpha+x-\beta \|} < 4 K \ell q_n ( 2 + \log(2 q_n))\ .\]

From Lemma~\ref{lower bound} we have 
\[ \sum_{k=0}^{\ell q_n-1} \frac{1}{\|x + k \alpha \|} \geq \frac{ \log \ell}{\lambda^{(n)}} \]
for every $x \in E_{n, \ell}$. Thus in order to prove that \eqref{dominance4} holds, it suffices to show that \
\[\frac{\log \ell}{\lambda^{(n)}}  > 4K  \ell q_n (2 + \log(2 q_n))  .\]
But we have chosen $\ell$ such that (using Lemma~\ref{properties})
\[\ell \lambda^{(n)} q_n < \frac{a_{n+1} \lambda^{(n)} q_n}{4} < \frac{\lambda^{(n-1)} q_n}{4} < \frac{1}{4}.\]
It therefore suffices to verify that $\log \ell > K(2+\log(2 q_n))$. But 
\[\ell>\frac{a_{n+1}}{8} > \frac{q_n^{K+1}}{8} > \frac{q_n^K 8e^{2K}2^K}{8},\]
which implies the required property.

\section{Divergence of sums: Diophantine case}

In this section, we use the estimates we got from Section \ref{SecTech} and Diophantine properties of almost any number to get historic behaviour for almost any angle $\alpha$.

\begin{theorem}\label{PropDivSum}
Let $\alpha \in \R$ be such that $a_n \geq 2$ for infinitely many $n$ and
\[\sum_{\substack{n\ge 2 \\ a_n, a_{n+1}\ge 2}} \frac{1}{\log q_n} = \infty.\]
Then, given any $\p$ and $\q$, the reparameterized linear flow $\phi^t$ satisfying (SH) with stopping points at $\p$ and $\q$  has  historic behaviour.

When moreover $\q$ is on the positive orbit of $\p$, then the ergodic limit set of almost any point $\x$ is explicit:
\[ p\omega(\x) = \left[\mu_\infty\, ,\ \delta_\p\right].\]
\end{theorem}

The fact that the set of angles $\alpha$ satisfying the hypotheses of this theorem is of full measure is a consequence of a theorem due to Khinchin and Levy, which asserts that for Lebesgue-almost every $\alpha \in \R$, the denominators of the
convergents satisfy
\[\lim_{n\to +\infty}\frac{\log q_n}{n} = \frac{\pi^2}{12 \log 2},\]
hence
\[\sum_{n\ge 0} \frac{1}{\log q_n} =  +\infty.\]

Moreover, for Lebesgue-almost every $\alpha \in \R$, and any $b\in\N$, one has 
\[\lim_{n\to +\infty} \frac{1}{n}\card\big\{j\le n : a_j = b\big\} = \log_2\left(\frac{(b+1)^2}{b(b+2)}\right),\]
and the Gauss map is mixing, implying that for Lebesgue-almost every $\alpha \in \R$, and any $b,b'\in\N$, one has 
\[\lim_{n\to +\infty} \frac{1}{n}\card\big\{j\le n : a_j = b, a_{j+1} = b'\big\} = \log_2\left(\frac{(b+1)^2}{b(b+2)}\right)\log_2\left(\frac{(b'+1)^2}{b'(b'+2)}\right),\]
For a proof see e.g. Propositions 3.1 and 3.4 of Durand \cite{Durand}.

From this one can easily deduce the following\footnote{E.g. using the partition of $\N$ by intervals $[2^k, 2^{k+1})$.}:
\begin{equation}\label{EqSumInvQ}
\sum_{\substack{n\ge 0\\ a_n, a_{n+1}\ge 2}} \frac{1}{\log q_n} = +\infty
\end{equation}

It could be conjectured that there is an extreme historic behaviour property for almost any $\alpha$ and ``most of'' $\beta$; unfortunately we were only able to establish that the sequences $\Theta_k^\beta (x)$ fail to converge for almost every $x$: the theorem's proof tells us that the sequences $\Theta_k^\beta (x)$ have at least $0$ or $+\infty$ as a limit point, hence that $p\omega(\x)$ contains at least $\delta_\p$ or $\delta_\q$. The key property that allows us to conclude in the case where $\q$ is on the positive orbit of $\p$ is that $\liminf \Theta_k^\beta (x) = 1$.
\medskip

Let us move to the Theorem's proof. Fix $\alpha$ as in the hypothesis of Theorem~\ref{PropDivSum} and let $0< \epsilon_n<1$ be a decreasing sequence of positive numbers such that $\epsilon_n \to 0$ as $n \to \infty$ and satisfying
\[\sum_{\substack{n \geq 2 \\ a_n,a_{n+1} \geq 2}} \frac{\epsilon_n}{\log (3q_n)} = \infty.\]

For $n \geq 1$ let
\[I_n = \left(-\frac{\epsilon_n}{q_n \log(3 q_n)}, \frac{\epsilon_n}{q_n \log(3q_n)} \right ) \]
and set 
\[E_n = \bigcup_{i=0}^{q_n-1} R_\alpha^{-i} (I_n)\]
and
\[E = \bigcap_{N \geq 1 }\bigcup_{\substack{n \ge N\\ a_{n+1} \ge 2}} E_n.\]

\begin{lemma} \label{LemEnLeb1}
Under the hypotheses of Theorem~\ref{PropDivSum}, the set $E$ is of full $\la$-measure.
\end{lemma}

We first show how to deduce Theorem~\ref{PropDivSum} from Lemma~\ref{LemEnLeb1}. Then we proceed to the proof of the lemma. 

\begin{proof}[Proof of Theorem~\ref{PropDivSum}]
First, by Lemma~\ref{LemEnLeb1}, denoting $R_\alpha^{-\N}(0)$ the pre-orbit of $0$, the set
\[E^* = E \setminus R_\alpha^{-\N}(0)\]
has total measure. Recall that $\psi(x) = \|x\|^{-1}$.

Let us prove that if $x \in E^*$, then for any $M\in\N$, there exists $i, n\in\N$ satisfying $M \le i \le q_n$ and such that
\begin{equation}\label{EqThDioph}
\psi\big(R_\alpha^i(x)\big) > \frac{1}{6 \epsilon_n} \,\sum_{j=0}^{i-1} \psi\big(R_\alpha^j(x)\big).
\end{equation}
Fix $M\in\N$ arbitrarily. As $x$ is not in the pre-orbit of $0$, one has $d = \min_{0\le i < M} \|R_\alpha^i(x)\|>0$. For any $N\ge M$ large enough, one has $d>1/(q_N \log(3q_N))$. As $x\in E_n$ for some $n\ge N$ with $a_{n+1}\ge 2$, we know that there exists $i < q_n$ such that $R_\alpha^{i}(x) \in I_n$. Hence, $\|R_\alpha^{
i}(x)\| < d $ and so $i\ge M$. Applying Lemma~\ref{LemSizeDiverg} with $\epsilon_n$ in place of $\epsilon$ yields \eqref{EqThDioph}.
We will denote by $i_k, n_k$ some increasing sequences of numbers (depending \textit{a priori} on $x$) satisfying \eqref{EqThDioph}.

Let $\Sigma = \{x_0\} \times \T$ be chosen as in Section \ref{SecDefFlow}, so that $p_0 \neq q_0$ and let $\beta = q_0 - p_0$. As $\beta \neq 0$, there exists $A>0$ such that if $\psi(y)>A$, then $\psi(y-\beta)<A$. Hence, for $k$ large enough, one has ($S_n$ is defined in \eqref{DefSn})
\[S_{i_k+1}(x) \ge \left(1+\frac{1}{6 \epsilon_{n_k}}\right) S_{i_k}(x) \quad \text{and} \quad S_{i_k+1}(x-\beta) \le  S_{i_k}(x-\beta) + A.\]
(the second inequality comes from the fact that $\lim_{k\to \infty}\psi\big(R_\alpha^{i_k}(x)\big) = \infty$ and thus $\psi\big(R_\alpha^{i_k}(x-\beta)\big)\le A$ for $k$ large enough.)
Hence ($\Theta$ is defined in \eqref{DefTheta}), using the fact that $\lim_{k\to \infty} S_{i_k}(x-\beta) = \infty$, for any $k$ large enough,
\[
\Theta_{i_k+1}^\beta(x) = \frac{S_{i_k+1}(x)}{S_{i_k+1}(x-\beta)}\ge \frac{(1+1/(6 \epsilon_{n_k}) ) S_{i_k}(x)}{S_{i_k}(x-\beta) + A} \ge \frac{1}{12 \epsilon_{n_k}} \Theta_{i_k}^\beta(x).
\]
Hence, there exists $(i_k) \to +\infty$ such that 
\begin{equation}\label{EqFinalCOntra2}
\Theta_{i_k}^\beta(x) = o\left(\Theta_{i_k+1}^\beta(x)\right).
\end{equation}

But Proposition~\ref{PropPossibOmega} tells us that if $\phi^t$ has a physical measure, then it is equal to $\mu_\infty$, and Proposition~\ref{criterium1} tells us that in this case the sequences $\Theta_{n}^\beta(x)$ converge to a positive real number for almost every $x$. This is in contradiction with \eqref{EqFinalCOntra2}, so $\phi^t$ has an historic behaviour for almost any point.
\medskip

For the second part of the theorem, 
the hypothesis that $\q$ is on the positive orbit of $\p$  implies the existence of $j>0$ such that $\beta = j\alpha \mod 1$. Hence, for any $x\notin R_\alpha^{-\N}(0)$, and every $n \geq j$ we have
\begin{equation}\label{independent of n}
S_n(x-\beta) - S_{n-j}(x)  = S_{j}(x-\beta),
\end{equation}
which is independent of $n$. Thus for a given $x$ the right hand side of (\ref{independent of n}) is a constant $B>0$, say. Thus for every $n \geq j$ we have
\begin{align} \label{LastEq}
\Theta_n^\beta(x) & = \frac{S_n(x)}{S_n(x-\beta)} = \frac{S_n(x)-S_{n-j}(x)}{S_n(x-\beta)} + \frac{S_{n-j}(x)- S_n(x-\beta)}{S_n(x-\beta)} + 1\\
& \geq - \frac{B}{S_n(x-\beta)} + 1,\nonumber
\end{align}
(since $S_{n}(x)-S_{n-j}(x) = S_j(x+(n-j) \alpha )\geq 0$ for every $n \ge j$).  It follows that $\liminf \Theta_n^\beta(x) \ge 1$.

Similarly to what we have seen in the first part of the proof, as $\alpha\notin\Q$, there exists $C>0$ such that if $\psi(y)>C$, then $\psi(R_\alpha^i(y))<C$ for any $1\le i \le j$. Hence (using \eqref{LastEq} applied to $n=i_k+j+1$),
\[ \Theta_{i_k+j+1}^\beta(x) = \frac{S_{i_k+j+1}(x)-S_{i_k+1}(x)}{S_{i_k+j+1}(x-\beta)} + \frac{S_{i_k+1}(x)- S_{i_k+j+1}(x-\beta)}{S_{i_k+j+1}(x-\beta)} + 1,\]
which implies that
\[\Theta_{i_k+j}^\beta(x) -1 \le  \frac{jC - B}{S_{i_k+j}(x-\beta)} \underset{n\to +\infty}{\longrightarrow} 0,\]
and that $\liminf \Theta_n^\beta(x) \le 1$.

Summing up, one has $\liminf \Theta_n^\beta(x) = 1$. Moreover, from \eqref{EqFinalCOntra2} one also has $\limsup \Theta_n^\beta(x) = +\infty$. The theorem then directly follows from Proposition~\ref{criterium1}.
\end{proof}

To prove Lemma~\ref{LemEnLeb1}, we use a variation of Fuchs and Kim \cite[Theorem 1.2]{MR3494133}. The initial statement deals with inhomogeneous Diophantine approximation: it gives a criterion under which the orbit of almost any point of the circle under a rigid rotation approaches the origin at a given speed. Its proof consists in a suitable application of a Borel-Cantelli lemma, allowed by Denjoy-Koksma inequality. 

\begin{theorem}\label{CiteTh12}
Let $\varphi(n)$ be a nonnegative sequence and $\alpha$ be an irrational number
with principal convergents $p_n/q_n$. For $j\in\N$, denote $n(j)$ the number satisfying $q_{n(j)-1} \le j < q_{n(j)}$. Then, for almost all $x\in\R$,
\[\|x+j\alpha \| < \varphi\big(n(j)\big)\]
for infinitely many $j\in\N$ if and only if
\[\sum_{n=1}^\infty\Big((q_{n}-q_{n-1}) \min\big(\varphi(n),\|q_{n-1}\alpha\|\big)\Big) = \infty.\]
\end{theorem}

This theorem can be easily adapted from the proof of Fuchs and Kim \cite[Theorem 1.2]{MR3494133}, by cheking that the hypothesis of $\psi$ being decreasing is useless in the case where it is constant equal to $\varphi$ on every interval $[q_n,q_{n+1})$.

\begin{proof}[Proof of Lemma~\ref{LemEnLeb1}]
Lemma~\ref{LemEnLeb1} lies in an application of Theorem~\ref{CiteTh12}. More precisely, by \eqref{EqLambdaQ}, one has $\|q_{n-1} \alpha\| = \lambda^{(n-1)}\ge \frac{1}{2q_{n}}$.
Choose
\[\varphi(n) = \begin{cases}
\frac{\epsilon_{n}}{q_{n}\log(3q_{n})}\quad & \text{if } a_{n}, a_{n+1}\ge 2\\
0 & \text{if } a_{n} =1 \text{ or }a_{n+1} =1.
\end{cases}\]
In particular, if $a_n, a_{n+1} \ge 2$, then $\varphi(n) \le \|q_{n-1} \alpha\|$.

We compute
\begin{align*}
\sum_{\substack{n \ge 1}}\Big((q_{n}-q_{n-1}) \min\big(\varphi(n),\|q_{n-1}\alpha\|\big)\Big)
& \ge \sum_{\substack{n \in \N\\ a_n, a_{n+1} \ge 2}}\!\!\! \Big((q_{n}-q_{n-1}) \min\big(\varphi(n),\|q_{n-1}\alpha\|\big)\Big)\\
& \ge \sum_{\substack{n \in \N\\ a_n, a_{n+1} \ge 2}} \frac{q_{n}-q_{n-1}}{q_{n}}\frac{\epsilon_{n}}{\log(3q_{n})}.
\end{align*}
But as $a_n>1$, one has
\[\frac{q_{n}-q_{n-1}}{q_{n}} \ge 1-\frac{1}{a_n} \ge 2,\]
so
\[\sum_{\substack{n \ge 1}}\Big((q_{n}-q_{n-1}) \min\big(\varphi(n),\|q_{n-1}\alpha\|\big)\Big)
\ge \sum_{\substack{n \in \N\\ a_n, a_{n+1} \ge 2}} \frac{\epsilon_{n}}{\log(3q_{n})}.\]
Hence, Theorem~\ref{CiteTh12} applies and implies that for almost all $x\in\R$, for any $M\in\N$, there exists $j\ge M$ such that $a_{n(j)}, a_{n(j)+1}\ge 2$ and 
\[\|x+j\alpha \| < \frac{\epsilon_{n(j)}}{q_{n(j)}\log(3q_{n(j)})},\]
in other words that $x\in E$.
\end{proof}

\section{Physical measures for stopping points on the same orbit}\label{SecPhysSame}

\subsection{Statement and ideas of proof}

The aim of this section is to provide conditions on $\alpha$ under which flows with stopping points on the same orbit have a physical measure. The simplest situation in which this happens is when the sequence $a_n$ tends to infinity sufficiently fast. 

\begin{theorem}\label{simple physical measure}
Let $\alpha=[a_0; a_1, a_2, \ldots]$ be such that 
\[\sum_n \frac{1}{\log a_n} < \infty.\]
Then any reparameterized linear flow satisfying $(SH)$, with $\p$ and $\q$ in the same orbit, has a unique physical measure, which attracts Lebesgue almost any point, and equal to $\mu_\infty$ (defined in \eqref{EqFormPhys}).
\end{theorem}

A similar result can be obtained by assuming sufficiently rapid growth of $q_n$.

\begin{theorem}\label{PropConv}
If $\p$ and $\q$ lie on the same orbit of the flow, and if there exist $C>0$ and $\gamma>0$ such that $q_n \ge C \exp(n^{2+\gamma})$, then the system has a unique physical measure, which attracts Lebesgue almost any point, and equal to $\mu_\infty$ (defined in \eqref{EqFormPhys}).
\end{theorem}

Theorem~\ref{PropConv} is harder to prove than Theorem~\ref{simple physical measure} because it allows large oscillation in the sequence $a_n$ forcing us to use different estimates depending on whether $a_n$ is small or large.

Comparing these statements with Theorem~\ref{PropDivSum}, we can observe that in the case of stopping points on the same orbit,  the flow has historic behaviour for Diophantine $\alpha$ and a unique physical measure for sufficiently Liouvillian $\alpha$. It could be counterintuitive at first sight, but one has to keep in mind that in the Liouvillian case, as the orbit of $0$ eventually comes back really close to $0$, it lets space for most of the other points to come back far away from 0.

\begin{lemma}\label{lemHaus}
The set of numbers $\alpha\in\R$ such that there exist $C>0$ and $\gamma>0$ such that $q_n \ge C \exp(n^{2+\gamma})$ is of zero Lebesgue measure but full Hausdorff dimension.
\end{lemma}

\begin{proof}
This is a direct consequence of a theorem of Jarn{\'i}k Besicovitch: following Durand \cite{Durand}, combining Proposition 1.8 with Theorem 3.1, for any $\tau>2$, the set of $\alpha\in\R$ such that there exists $C>0$ such that $\log q_n < (\tau-1)^n C$ has Hausdorff dimension $\ge 2/\tau$.
\end{proof}

Let us first explain the idea of the proof of Theorem~\ref{PropConv}. To make it simpler we will suppose that the projections $p_0$ and $q_0$ of respectively $\p$ and $\q$ on the transverse section $\Sigma$ satisfy $q_0 = R_\alpha(p_0)$ (see Section \ref{SecDefFlow}).

Using Proposition~\ref{criterium1}, we want to show that for almost any point $\x\in\T^2$ and for any $j$ large enough, the Birkhoff sum for the observable $\|\cdot - p_0\|^{-1}$ over the $j$ first return times of $\x$ on $\Sigma$ for the flow with only one stopping point at $\p$ is more or less the same as the sum of the $j$ first return times for the flow with only one stopping point at $\q$. But as $q_0 = R_\alpha(p_0)$, the difference between these two sums is more or less the value of the last return time of the orbit of $p_0$. Hence, supposing without loss of generality that $p_0=0$, what we want to show is that for almost any $x\in\T$, (see Lemma~\ref{CritPhys})
\[\psi\big(R_\alpha^j(x)\big) = o\left(\sum_{i=0}^{j-1} \psi(R_\alpha^i(x))\right).\]

More precisely, we will prove that the measure $\lambda_n$ of the set of points $x\in \T$ for which there exists a time $q_n\le j < q_{n+1}$ such that 
\begin{equation}\label{EqDiv}
\psi(R_\alpha^j(x)) \ge n^{-\gamma/6}\sum_{i=0}^{j-1} \psi(R_\alpha^i(x))
\end{equation}
satisfies $\sum_n \lambda_n<+\infty$ (Lemmas~\ref{alternative summability}, \ref{summability under rapid growth} and \ref{LemSerHarmo}). Hence, almost any point of the circle will be eventually ``good'' between times $q_n$ and $q_{n+1}$.

To do this, we have to prove that for most of points, the sum on the right of \eqref{EqDiv} is sufficiently large. We separate two different cases for each $n$:
\begin{itemize}
\item Either $a_n$ is big, that is, $q_{n+1}\gg q_n$ --- note that it has to happen an infinite number of times (otherwise we could not have $q_n \ge C \exp(n^{2+\gamma})$). It will turn out that in this case, the most important contribution for the Birkhoff sum comes from the returns in the ground floor $\Delta^{(n-1)}$, that is, the term $\psi_1(y_0)$ of the first part of Lemma~\ref{EqSellFinal}. In practice, we will cut the ground floor $\Delta^{(n-1)}$ into the $a_n$ ground floors of the sectors (defined page \pageref{Sectors}), and throw away the points that are sufficiently close to the the preimage of 0. The remaining points will not satisfy \eqref{EqDiv}, simply because they will return in the ground floor a lot of times before coming close to 0, which will increase sufficiently the right part of \eqref{EqDiv}.
\item Or $a_n$ is small, that is, $q_{n+1}\not\gg q_n$. In this case, the most important contribution for the Birkhoff sum comes from the sums in the whole sectors but the ground floor, that is, the term $\log q_n / \big(2 \lambda^{(n-1)}\big)$ of the first part of Lemma~\ref{EqSellFinal}. The fact that $q_n$ is large enough will imply that the right part of \eqref{EqDiv} is large enough, which will ensure that the proportion of points satisfying \eqref{EqDiv} is small enough.
\end{itemize}
Of course, these considerations will be made precise in the proof of Theorem~\ref{PropConv}.

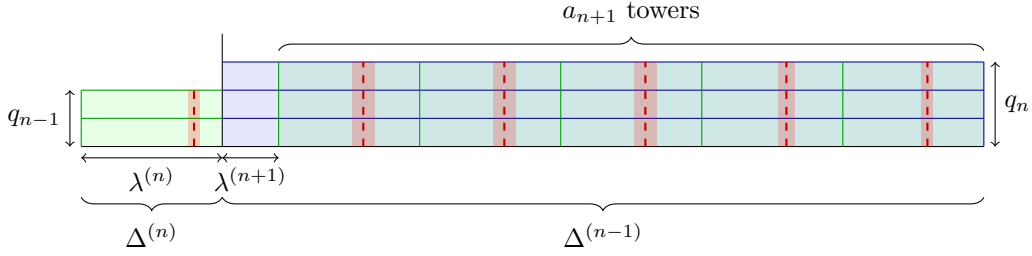
\begin{figure}
\begin{tikzpicture}[scale=.75]
\fill[fill=green, opacity=.1] (0,0) rectangle (-2.5,1);
\fill[fill=blue, opacity=.1] (0,0) rectangle (13.5,1.5);
\fill[fill=green, opacity=.1] (1,0) rectangle (13.5,1.5);
\draw (0,0) -- (0,2);
\draw (0,0) -- (13.5,0);
\draw[color=blue!60!black] (0,.5) -- (13.5,.5);
\draw[color=blue!60!black] (0,1) -- (13.5,1);
\draw[color=blue!60!black] (0,1.5) -- (13.5,1.5);
\draw[color=blue!60!black] (13.5,0) -- (13.5,1.5);

\draw (-2.5,0) -- (0,0);
\draw[color=green!60!black] (-2.5,.5) -- (0,.5);
\draw[color=green!60!black] (-2.5,1) -- (0,1);
\draw[color=green!60!black] (-2.5,0) -- (-2.5,1);
\draw[color=green!60!black] (1,0) -- (1,1.5);
\draw[color=green!60!black] (3.5,0) -- (3.5,1.5);
\draw[color=green!60!black] (6,0) -- (6,1.5);
\draw[color=green!60!black] (8.5,0) -- (8.5,1.5);
\draw[color=green!60!black] (11,0) -- (11,1.5);

\foreach \x in {1,...,5}
{\draw[dashed, thick, color=red!70!black] (2.5*\x,0) -- (2.5*\x,1.5);}
\draw[dashed, thick, color=red!70!black] (-.5,0) -- (-.5,1);
\fill[color=red, opacity=.2] (2.3,0) rectangle (2.7,1.5);
\fill[color=red, opacity=.2] (4.8,0) rectangle (5.2,1.5);
\fill[color=red, opacity=.2] (7.3,0) rectangle (7.7,1.5);
\fill[color=red, opacity=.2] (9.85,0) rectangle (10.15,1.5);
\fill[color=red, opacity=.2] (12.4,0) rectangle (12.6,1.5);
\fill[color=red, opacity=.2] (-.6,0) rectangle (-.4,1);

\draw [decorate,decoration={brace,amplitude=5pt},xshift=0,yshift=-.9cm]
(13.5,0) -- (0,0) node [black,midway,yshift=-0.5cm] {$\Delta^{(n-1)}$};
\draw [decorate,decoration={brace,amplitude=5pt},xshift=0,yshift=-.9cm]
(0,0) -- (-2.5,0) node [black,midway,yshift=-0.5cm] {$\Delta^{(n)}$};

\draw [decorate,decoration={brace,amplitude=5pt},xshift=0,yshift=.2cm]
(1,1.5) -- (13.5,1.5) node [black,midway,yshift=0.5cm] {$a_{n+1}$ towers};

\draw[<->] (13.7,0) --node[midway, right]{$q_n$} (13.7,1.5);
\draw[<->] (-2.7,0) --node[midway, left]{$q_{n-1}$} (-2.7,1);
\draw[<->] (-2.5,-.2) --node[midway, below]{$\lambda^{(n)}$} (0,-.2);
\draw[<->] (1,-.2) --node[midway, below]{$\lambda^{(n+1)}$} (0,-.2);


\end{tikzpicture}
\caption{The pre-orbit of 0 (hatched lines) and the set $D_n$ (light red rectangles).
}\label{FigRenor8}
\end{figure}

\subsection{A criterium for convergence}

The following easy lemma gives a sufficient condition for having $\mu_\infty$ (defined in \eqref{EqFormPhys}) as a physical measure.

\begin{lemma}\label{CritPhys}
If $\p$ and $\q$ are on the same $\phi^t$-orbit, $x$ is not on the $R_\alpha$-orbit of $p_0$ and
\[\psi(R_\alpha^j(x)) = o\left(\sum_{i=0}^{j-1} \psi(R_\alpha^i(x))\right),\]
then $p\omega(x)$ is equal to $\{\mu_\infty\}$.
\end{lemma}

\begin{proof}
By Proposition~\ref{criterium1}, it suffices to prove that $\Theta_n^\beta(x) \to_n 1$. Recall that
\[\Theta_n^\beta(x) = \frac{S_n(x)}{S_n(x-\beta)}.\]
As $\p$ and $\q$ are on the same orbit, there exists a section $\Sigma$ and a number $n_0\in\N^*$ such that, writing $p_0$ and $q_0$ as in Paragraph \ref{SecDefFlow}, one has $q_0 = R_\alpha^{n_0}(p_0)$, in other words $\beta \equiv n_0\alpha  \mod 1$. It is straightforward to verify that 
\[S_{n_0}(x)+ S_n(x+\beta) = S_{n+n_0}(x) = S_n(x)+S_{n_0}(R_\alpha^n(x)).\]
Hence,
\[\Theta_n^\beta \circ R_\alpha^{n_0}(x) = \frac{S_n(x+\beta)}{S_n(x)} =  1 - \frac{S_{n_0}(x)}{S_n(x)} + \frac{S_{n_0}(R_\alpha^n(x))}{S_n(x)}.\]
Of course, to show that $\Theta_n^\beta(x) \to 1$ $\la$-almost everywhere is equivalent to show that $\Theta_n^\beta \circ R_\alpha^{n_0}(x) \to 1$ $\la$-almost everywhere, since $\la$ is $R_\alpha$-invariant.

As $S_n(x)\to_n +\infty$, the second term tends to 0. So it suffices to prove that the last term also tends to 0.

Remark that under the hypothesis of the lemma, one can easily check by recurrence that for any $k\ge 0$,
\[\psi\big(R_\alpha^{j+k}(x)\big) = o_j\left(\sum_{i=0}^{j-1} \psi(R_\alpha^i(x))\right),\]
so that for any $k_0 \ge 0$, 
\[\sum_{k=0}^{k_0}\psi\big(R_\alpha^{j+k}(x)\big) = o_j\left(\sum_{i=0}^{j-1} \psi(R_\alpha^i(x))\right),\]
in other words
\[\frac{S_{n_0}(R_\alpha^n(x))}{S_n(x)} = o_n(1).\]
\end{proof}

%

\subsection{Bad sets of initial points}

For $n \geq 0$ and $k \geq 1$ we set
\[a_{n,k} = \frac{\log k}{\lambda^{(n)}},
\qquad
b_{n,k} =  \frac{k \log q_n}{\lambda^{(n-1)}},\]
and
\[c_{n,k} = \max\{a_{n,k}, b_{n,k} \}.\]
Note that $c_{n,k}$ is positive as long as $n \geq 2$. 

Let $u_n$ be an increasing sequence of positive numbers tending to infinity. For $n \geq 2$ and $1 \leq k \leq a_{n+1}$ let 
\begin{equation}\label{def of I}
I_{n,k} = \left[-\frac{u_n}{2 c_{n,k}},\, \frac{u_n}{2 c_{n,k}}\right].
\end{equation}
Define the sequences 
\begin{equation}\label{EqDefNi}
n(i) = \max \{n \geq 0: q_n \leq i \}
\quad \text{and} \quad
k(i) = \max \{k \geq 0: k q_{n(i)} \leq i \},
\end{equation}
so that $q_{n(i)} \le i < q_{n(i)+1}$ and $k(i) q_{n(i)} \le i < (k(i)+1)q_{n(i)}$.

For $i \geq 0$ let
\[B_i  = R_\alpha^{-i} (I_{n(i), k(i)}),\]
and
\begin{equation}\label{EqDefDn}
D_n = \bigcup_{i=q_n}^{q_{n+1}-1} B_i .
\end{equation}

\begin{lemma}\label{LemConvDn2}
Let $n \geq 2$. If $x\notin D_n$, then for any $q_n \leq i < q_{n+1}$, and any $x'$ such that $\|x-x'\|\le \lambda^{(n)}$, we have
\begin{equation} \label{at time j_0}
\psi\big(R_\alpha^i(x)\big) \le \frac{8}{u_n} \sum_{j=0}^{i-1} \psi\big(R_\alpha^j(x')\big).
\end{equation}
\end{lemma}

In particular, we will apply this lemma for $x'=x$.

\begin{proof}[Proof of Lemma~\ref{LemConvDn2}]

Fix some $q_n \leq i < q_{n+1}$ (so that $n=n(i)$) and let $j_0$ be such that
\[\|R_\alpha^{j_0}(x)\| = \min_{q_n \leq j \leq i} \|R_\alpha^j(x)\|.\] Let $y_0 = R_\alpha^{j_0}(x)$. Note that the right hand side of (\ref{at time j_0}) is increasing in $i$. Hence it suffices to check that 
\begin{equation} \label{enough}
\psi\big(R_\alpha^{j_0}(x)\big) \le \frac{8}{u_n}  \sum_{j=0}^{j_0-1} \psi\big(R_\alpha^{j}(x)\big).
\end{equation}
We do that in two different cases.
For ease of notation, we  write $k_0 = k(j_0)$.

\textbf{Case (1):} $\| y_0 \| > \lambda^{(n)}/2$.\\
By definition of $k_0$ we have $j_0 = q_n k_0 + \ell$ for some $0 \leq \ell < q_n$. Let $j_1 = q_n+ \ell$ and denote by $x_0$ the point $R_\alpha^{j_1}(x)$. Then, using the assumption that $\|y_0\|> \lambda^{(n)}/2$, we have
\begin{align*}
\|x_0 \| & = \|y_0 - (k_0-1) q_n \alpha \| \\
& \leq \|y_0 \| + (k_0-1) \|q_n \alpha\| \\
& = \|y_0 \| + (k_0-1) \lambda^{(n)} \\
& < (2k_0-1) \|y_0\|  <2 k_0 \|y_0\| .
\end{align*}
Hence 
\[\psi(y_0) < 2k_0 \psi(x_0).\]
Since $x \notin D_n$, it follows that  $x_0 = R_\alpha^{j_0}(x) \notin I_{n,1}$ and therefore 
\[\psi(x_0) < 2\frac{\log q_n}{u_n \lambda^{(n-1)}}.\]
Thus by Corollary~\ref{kq orbit} we have
\[\psi(y_0) < 2k_0 \psi(x_0) < 4k_0\frac{\log q_n}{u_n \lambda^{(n-1)}} < \frac{8}{u_n} \sum_{j=0}^{k_0 q_n - 1} \psi\big(R_\alpha^j(x')\big) \leq \frac{8}{u_n} \sum_{j=0}^{j_0-1} \psi\big(R_\alpha^j(x')\big)\]
as required.

\textbf{Case (2):} $\|y_0 \| \leq \frac{\lambda^{(n)}}{2}$.\\
It follows from the hypothesis $x \notin D_n$ that
\begin{equation} \label{c estimate}
\psi(y_0) < 2\frac{c_{n,k_0}}{u_n}.
\end{equation}
To show (\ref{at time j_0}) we need two estimates.

First, from Corollary~\ref{kq orbit} we have
\begin{equation} \label{a estimate}
 \sum_{j=0}^{j_0-1} \psi\left(R_\alpha^{j}(x')\right) \geq \sum_{j=0}^{k_0 q_n-1} \psi\left(R_\alpha^{j}(x')\right) > \frac{k_0 \log q_n}{2 \lambda^{(n-1)}} = \frac{b_{n, k_0}}{2}.
\end{equation}

Second, from Remark~\ref{lower bound2} (following Lemma~\ref{lower bound}), that can be applied because $\|y_0\|\le \lambda^{(n)}/2$ and $\|x-x'\|\le\lambda^{(n)}$, we obtain the estimate
\begin{equation} \label{b estimate}
  \sum_{i=0}^{j_0-1} \psi(R_\alpha^i(x'))
  \geq \sum_{i=0}^{k_0 q_n-1} \psi(R_\alpha^i(x'))
  \geq \frac{\log k_0}{4\lambda^{(n)}}
  = \frac{a_{n, k_0}}{4}.
 \end{equation}
Putting (\ref{a estimate}) and (\ref{b estimate}) together and comparing with (\ref{c estimate}) 
\[\frac{8}{u_n} \sum_{j=0}^{j_0-1} \psi\big(R_\alpha^{j}(x')\big) > \frac{2 c_{n,k_0}}{u_n} 
 \ge \psi(y_0),\]
again proving that (\ref{enough}) must hold.
\end{proof}

\subsection{Proof of Theorems \ref{simple physical measure} and \ref{PropConv}}

We now turn to the proof of Theorems~\ref{simple physical measure} and \ref{PropConv}. In view of Lemma~\ref{CritPhys} and Lemma~\ref{LemConvDn2}, it suffices to prove that $\la$-almost every point $x \in \T$ belongs to $D_n$ for at most finitely many $n$. By virtue of the Borel-Cantelli lemma, this is the case whenever $\la(D_n)$ is summable. In other words, Theorems~\ref{simple physical measure} and \ref{PropConv} follow, respectively, from the following two lemmas.

\begin{lemma} \label{alternative summability}
Suppose that
\begin{equation} \label{hyp on a_n}
\sum_{n} \frac{1}{\log a_{n}} < \infty.
\end{equation}
 Then it is possible to choose  $u_n$ in \eqref{def of I}  so that 
\[\sum_{n \geq 2} \la(D_n) < \infty\]
\end{lemma}

\begin{lemma}\label{summability under rapid growth}
Suppose that $\alpha$ is such that $q_n> C \exp(n^{2+\gamma})$ for some $C, \gamma>0$ and every $n$. Then, taking $u_n = n^{\gamma/4}$ in \eqref{def of I}, we have
\[\sum_{n\geq 2} \la(D_n) < \infty.\]
\end{lemma}

\begin{proof}[Proof of Lemma~\ref{alternative summability}]
Let $u_n$ be an increasing sequence of positive numbers tending to infinity slowly enough so that
\begin{equation} \label{summability after multiplication}
\sum_{n} \frac{u_n}{\log a_{n}} < \infty
\end{equation}
and let $D_n$ be defined accordingly as in \eqref{def of I}. We decompose the sum as 
\begin{equation} \label{decomposition2}
\sum_{n \geq 2} \la(D_n) \leq \sum_{n \geq 2} q_n \la(I_{n,1}) + \sum_{n \geq 2} \sum_{k=2}^{a_{n+1}} \la(I_{n,k}) q_n.
\end{equation}

Recall from Lemma~\ref{properties} that $\lambda^{(n-1)} q_n < 1$. Hence 
\[\sum_{n \geq 2} q_n\la (I_{n,1}) \leq \sum_{k \geq 2} \frac{\lambda^{(n-1)} q_n u_n}{\log q_n} \leq \sum_{k \geq 2} \frac{u_n}{ \log q_n} < \sum_{k \geq 2} \frac{u_n}{ \log a_n}< \infty.\]
We turn to the second term in (\ref{decomposition2}). Let 
\[P_n = \sum_{k=2}^{a_{n+1}}  \la(I_{n,k}) q_n .\]
Using Lemmas~\ref{properties} and \ref{LemSerHarmo} we have
\[P_n  \leq \sum_{k = 2}^{a_{n+1}} \frac{q_n u_n}{ c_{n,k}}  \leq \sum_{k=2}^{a_{n+1}} \frac{q_n u_n}{ a_{n,k}} = \sum_{k=2}^{a_{n+1}} \frac{\lambda^{(n)} u_n}{ \log k}
\leq \frac{C a_{n+1} \lambda^{(n)} q_n u_n}{ \log(a_{n+1})}
\leq \frac{C u_n}{ \log{a_{n+1}}} \leq \frac{C u_{n+1}}{a_{n+1}}. \]

Hence by \eqref{summability after multiplication} we conclude that
\[ \sum_{n \geq 2} P_n < \infty.\]

\end{proof}

\begin{proof}[Proof of Lemma~\ref{summability under rapid growth}]

Just as in the proof of Lemma~\ref{alternative summability} we decompose the sum as 
\begin{equation} \label{decomposition}
\sum_{n \geq 2} \la(D_n) \leq \sum_{n \geq 2} \la(I_{n,1}) + \sum_{n \geq 2} \sum_{k=2}^{a_{n+1}} \la(I_{n,k}) q_n.
\end{equation}

Recall from Lemma~\ref{properties} that $\lambda^{(n-1)} q_n < 1$. Hence 
\[\sum_{n \geq 2} \la (I_{n,1}) \leq \sum_{k \geq 2} \frac{\lambda^{(n-1)} q_n n^{\gamma/4}}{\log q_n} \leq \sum_{k \geq 2} \frac{n^{\gamma/4}}{ \log q_n} < \sum_{n \geq 2} \frac{n^{\gamma/4}}{n^{2+\gamma}} < \infty.\]
We turn to the second term in (\ref{decomposition}). Let 
\[P_n = \sum_{k=2}^{a_{n+1}}  \la(I_{n,k}) q_n .\]
We can bound $P_n$ from above in two ways. On the one hand, using Lemmas~\ref{properties} and \ref{LemSerHarmo} we have
\[P_n  \leq \sum_{k = 2}^{a_{n+1}} \frac{q_n n^{\gamma/4}}{ c_{n,k}}  \leq \sum_{k=2}^{a_{n+1}} \frac{q_n n^{\gamma/4}}{ a_{n,k}} = \sum_{k=2}^{a_{n+1}} \frac{\lambda^{(n)} n^{\gamma/4}}{ \log k}
\leq \frac{C a_{n+1} \lambda^{(n)} q_n n^{\gamma/4} }{ \log(a_{n+1})}
\leq \frac{C n^{\gamma/4}}{ \log{a_{n+1}}}. \]

On the other hand 
\[P_n  \leq \sum_{k = 2}^{a_{n+1}} \frac{q_n n^{\gamma/4}}{ c_{n,k}}  \leq \sum_{k=2}^{a_{n+1}} \frac{ q_n n^{\gamma/4}}{ b_{n,k}} \leq \sum_{k=2}^{a_{n+1}} \frac{ q_n \lambda^{(n-1)} n^{\gamma/4} }{ k \log(q_n)} \leq \frac{\log a_{n+1}  n^{\gamma/4}}{ \log(q_n)}.\]

Let 
\[\cA = \{n \geq 2: a_{n+1} > \exp(n^{1+\gamma/2}) \}\]
and
\[\cB = \{n \geq 2: a_{n+1} \leq \exp(n^{1+\gamma/2}) \}.\]
If $n \in \cA$ then
\[P_n \leq \frac{C\, n^{\gamma/4}}{ \log a_{n+1}} \leq \frac{C\,n^{\gamma/4}}{n^{1+\gamma/2}} = \frac{C}{ n^{1+\gamma/4}}.\]
Hence
\[\sum_{n \in \cA} P_n < \infty.\]
If $n \in \cB$ then
\[P_n \leq \frac{ \log a_{n+1} n^{\gamma/4} }{ \log q_n} \leq \frac{n^{1+\gamma/2} n^{\gamma/4}}{n^{2+\gamma}} = \frac{1}{n^{1+\gamma/4}}.\]
Hence
\[\sum_{n \in \cB} P_n < \infty.\]

We conclude the proof by noting that 
\[\sum_{n \geq 2} \sum_{k=2}^{a_{n+1}} \la (I_{n,k})  = \sum_{n \geq 2} P_n  = \sum_{n \in \cA} P_n + \sum_{n \in \cB} P_n < \infty.\]
\end{proof}

\section{Physical measures for stopping points on different orbits}\label{SecDiff}

In this last section, we use the arguments of the proof of Theorem~\ref{PropConv} to get the existence of flows satisfying (SH) with stopping points $\p$ and $\q$ on different orbits and with a physical measure.

\begin{theorem}\label{physmeas different orbits}
Let $\alpha=[a_0; a_1, a_2, \ldots]$ be such that one of the following conditions holds:
\begin{itemize}
\item $\sum_n \frac{1}{\log a_n} < \infty$;
\item there exist $C>0$ and $\gamma>0$ such that $q_n \ge C \exp(n^{2+\gamma})$.
\end{itemize}
Let
\begin{equation}\label{EqDefBeta0}
\beta_0 = \sum_{n\geq 0 } \rho_n.
\end{equation}

If $\phi^t$ is a reparametrized linear flow with angle $\alpha$ and stopping points at $(0,0)$ and $(0, -\beta_0)$ satisfying (SH), then $\phi^t$ has a physical measure.
\end{theorem}

%

Remark that for a fixed $\p$, the set of $\q$ satisfying the conclusion of this theorem is dense (simply by the fact that the positive orbit under the linear flow of such a point $\q$ is dense).

We start by proving that the number $\beta_0$ is not in the $R_\alpha$-orbit of 0 whenever $a_n$ is not eventually constant equal to 1 (Lemma~\ref{beta0notorbit}). Hence, Theorem~\ref{physmeas different orbits} gives the existence of reparametrized linear flows with two stopping points which are not on the same orbit and with a unique physical measure; in particular, using Lemma~\ref{lemHaus}, its conclusion holds on a set of full Hausdorff dimension (as the set of points failing to satisfy Lemma~\ref{beta0notorbit} is at most countable).

\subsection{The number $\beta_0$\label{SubSecBeta0}}

Let $\beta_0$ as in \eqref{EqDefBeta0}. Note that $\beta_0$ can be written as 
\[\beta_0 =  \sum_{k=0}^{n-1} (q_k \alpha - p_k) + \sum_{k=n}^\infty \rho_k  = \ell_{n-1} \alpha + \beta_n \mod 1,\]
where
\[\ell_{n-1} = q_0  + \ldots + q_{n-1} \qquad \text{and} \qquad
\beta_n = \sum_{k=n}^\infty \rho_k.\]

\begin{lemma} \label{ell versus q}
The sequence $\ell_n$ satisfies 
\begin{enumerate}
\item $\ell_{n} < q_{n}+q_{n+1}$ for every $n\geq 0$, and

\item  $\ell_{n} < q_{n+1}$ for every $n \geq 1$ such that $a_{n+1} \geq 2$.
\end{enumerate}
\end{lemma}

\begin{proof}
We prove $(1)$ by induction. Clearly
\[\ell_0 = q_0 < q_0 + q_1,\]
so $(1)$ holds for $n=0$. Now suppose that $(1)$ hods for $n$. Then, using $q_{n-1} + q_n \leq q_{n+1}$, we obtain
\[\ell_{n+1} = \ell_{n}+q_{n+1} < (q_{n}+q_{n+1}) + q_{n+1} \leq q_{n+2}+q_{n+1}.\]
In other words, $(1)$ holds for $n+1$ and the proof follows by induction.

We now turn to the proof of $(2)$. Fix some $n \geq 1$ with $a_{n+1} \geq 2$. Then (see Lemma~\ref{properties}) $q_{n+1} \geq 2 q_n+ q_{n-1}$. We know from $(1)$  that $\ell_{n-1} < q_{n-1}+q_{n}$. Hence
\[\ell_{n} = \ell_{n-1}+ q_{n} < q_{n-1}+2 q_{n} \leq q_{n+1}.\]
\end{proof}

\begin{lemma} \label{alternating sum}
We have $\| \beta_n \| <  \lambda^{(n)}$.
\end{lemma}

\begin{proof}
Decreasing alternating series.
\end{proof}

\begin{lemma}\label{beta0notorbit}
Let $\alpha = [a_0; a_1, a_2, \ldots ]$ and suppose that there are infinitely many $n$ such that $a_{n} \geq 2$ (which is true if $\alpha\notin \Q[\sqrt 5]$). Then the point $\beta_0$ is not on the $R_\alpha$-orbit of zero. 
\end{lemma}

\begin{proof}
We begin by showing that $\beta_0$ is not on the positive $R_\alpha$-orbit of $0$. To this end, fix some $k \geq 0$ and choose $n$ such that $q_{n+1}>k$ and $k\neq \ell_n$. From (1) of Lemma~\ref{ell versus q}, we have $\ell_n < q_{n+1}+q_n$. Hence $|\ell_n - k| < q_{n+2}$ so that by \eqref{EqContFrac0}, $\|(\ell_n-k) \alpha \| \geq  \lambda^{(n+1)}$. Moreover, from Lemma~\ref{alternating sum} we have   $\| \ell_n \alpha - \beta_0 \| = \| \beta_{n+1}\| < \lambda^{(n+1)}$. It follows that
\[\|k \alpha - \beta_0 \| \geq \| k \alpha - \ell_n \alpha \| - \| \ell_n \alpha - \beta_0 \| > \lambda^{(n+1)}-\lambda^{(n+1)}> 0.\]
This shows that $\beta_0 \neq k \alpha \mod 0$. Note that $k \geq 0$ was chosen arbitrarily. Hence $\beta_0$ is not on the positive orbit of zero under $R_\alpha$. 
\medskip

Fix some  integer $m<0$ and choose $n$ such that $q_{n} > -m$ and $a_{n+1}\ge 2$.
By (2) of Lemma~\ref{ell versus q}, we have $\ell_n < q_{n+1}$ and so $0<\ell_n-m < q_{n+1}+q_{n} \le q_{n+2}$. This implies that $\|(\ell_n-m)\alpha\|\ge \lambda^{(n+1)}$. Hence,
\begin{align*}
\| m \alpha - \beta_0 \| & \ge \| m \alpha  - \ell_n\alpha \| - \|\ell_n\alpha - \beta_0\| \\
& = \|(m- \ell_n ) \alpha \| - \|\beta_{n+1} \| \\
& > \lambda^{(n+1)} - \lambda^{(n+1)} = 0.
\end{align*}
Since $m<0$ has been taken arbitrarily this shows that $\beta_0$ is not on the negative $R_\alpha$ orbit of 0.
\end{proof}

\begin{remark}
The converse of Lemma \ref{beta0notorbit} is also true: if $\alpha$ is such that $a_n = 1$ for all but finitely many $n$, then $\beta_0$ is on the $R_\alpha$-orbit of $0$. To see why, fix $N$ such that $ a_n = 1$ for every $n \geq N+1$. Then (see Lemma~\ref{properties}) 
\[\rho_n= \rho_{n-2}-\rho_{n-1}\]
for every $n \geq N+1$. So for $n \geq N$ we can write
\begin{align*}
\sum_{k=0}^n \rho_k &= \sum_{k=0}^N \rho_k + \sum_{k=N+1}^n \rho_k \\
& = \sum_{k=0}^N \rho_k + \sum_{k=N+1}^n (\rho_{k-2}-\rho_{k-1}) \\
& = \sum_{k=0}^N \rho_k + \rho_{N-1} - \rho_{n-1}.
\end{align*}
Taking the limit $n \to \infty$ we get
\[\beta_0 = \sum_{k=0}^\infty \rho_n = \sum_{k=0}^N \rho_k + \rho_{N-1} = (\ell_N + q_{N-1}) \alpha \mod 1.\]
Hence $\beta_0$ is on the $R_\alpha$-orbit of $0$. 

\end{remark}

\subsection{Proof of Theorem~\ref{physmeas different orbits}}

Without loss of generality one can suppose that $p_0=0$ (where $(x_0,p_0)$ is the point of the section $\Sigma$ corresponding to $\p$, see Paragraph \ref{SecDefFlow}); hence $\beta=q_0$ corresponds to the projection of the point $\q$ on $\Sigma$.

Recall that in the previous section we have defined (in \eqref{EqDefDn}) the set $D_n$ of ``bad points''. We now define an alternative version of this set, using an alternative version of \eqref{EqDefNi}: let $u_n$ be as in Lemmas \ref{alternative summability} or \ref{summability under rapid growth} (depending on whether we are in the first or the second hypothesis of Theorem~\ref{physmeas different orbits}), and consider a sequence $(v_n)$ of integers such that
\begin{equation}\label{eqpropvn}
v_n\underset{n\to\infty}{\longrightarrow}\infty,\qquad \sum_{n\ge 2} v_n \lambda(D_n) < \infty, \qquad \frac{v_{n+1}}{u_n}\underset{n\to\infty}{\longrightarrow}0.
\end{equation}
Set
\begin{equation*}
\tilde n(i) = \max \{n \geq 0: v_n q_n \leq i \}
\quad \text{and} \quad
\tilde k(i) = \max \{k \geq 0: k q_{\tilde n(i)} \leq i \},
\end{equation*}
so that $v_{\tilde n(i)} q_{\tilde n(i)} \le i < v_{\tilde n(i) + 1} q_{\tilde n(i)+1}$ and $\tilde k(i) q_{\tilde n(i)} \le i < (\tilde k(i)+1)q_{\tilde n(i)}$. Note that
\begin{equation}\label{EqTildeK}
v_{\tilde n(i)} \ \le \ \tilde k(i)\ < \ v_{\tilde n(i)+1} \frac{q_{\tilde n(i)+1}}{q_{\tilde n(i)}}.
\end{equation}

For $i \geq 0$ let (recall that by \eqref{def of I}, one has $
I_{n,k} = \big[-u_n/(2 c_{n,k}),\, u_n/(2 c_{n,k})\big]$)
\[\tilde B_i  = R_\alpha^{-i} \left(I_{\tilde n(i), \tilde k(i)}\right);\]
this allows to define an alternative version of \eqref{EqDefDn}
\begin{equation*}\label{EqDefDn2}
\tilde D_n = \bigcup_{i=v_n q_n}^{v_{n+1} q_{n+1}\,-1} \tilde B_i,
\end{equation*}
and
\begin{equation}\label{EqD}
\tilde D = \bigcap_{N\in\N} \bigcup_{n\ge N} \big(\tilde D_n \cup (\tilde D_n-\beta_0) \cup (\tilde D_n-\ell_{n-1}\alpha) \big).
\end{equation}
By a trivial adaptation of Lemma~\ref{summability under rapid growth}, this set has null measure (by using \eqref{eqpropvn} and the fact that the measure of the new set $\tilde D_n$ is smaller than $v_n+1$ times the measure of the old $D_n$ defined in \eqref{EqDefDn}).

\begin{proposition}\label{LemFinalPhysmeas}
Under the hypotheses of Theorem~\ref{physmeas different orbits}, for $\beta_0$ defined by \eqref{EqDefBeta0} and for any $x\notin \tilde D$ (defined in \eqref{EqD}) which is not in the preorbits of $p_0$ or $q_0$, the point $(x_0,x)$ is in the basin of attraction of the measure $\mu_\infty$, where $\q$ corresponds to $(x_0,-\beta_0)\in\Sigma$.
\end{proposition}

This proposition implies Theorem~\ref{physmeas different orbits}, as the set $\tilde D$ has null measure.

We shall use the following adaptation of Lemma~\ref{LemConvDn2}, whose proof is identical.

\begin{lemma}\label{LemConvDn3}
Let $n \geq 2$. If $x\notin \tilde D_n$, then for any $v_n q_n \leq i < v_{n+1}q_{n+1}$, and any $x'$ such that $\|x-x'\|\le \lambda^{(n)}$, we have
\begin{equation*}
\psi\big(R_\alpha^i(x)\big) \le \frac{8}{u_n} \sum_{j=0}^{i-1} \psi\big(R_\alpha^j(x')\big).
\end{equation*}
\end{lemma}

\begin{proof}[Proof of Proposition~\ref{LemFinalPhysmeas}]
As $x\notin \tilde D$, there exists $N\in\N$ such that
\[x\notin \bigcup_{n\ge N} \big(\tilde D_n \cup (\tilde D_n-\beta_0)\cup (\tilde D_n-\ell_n\alpha)\big).\]

Using Proposition~\ref{criterium1}, we shall prove that 
\[S_i(x) \overset{\text{def.}}{=} \sum_{j=0}^{i-1} \psi\big(R_\alpha^j(x)\big) \underset{i\to +\infty}{\sim} \sum_{j=0}^{i-1} \psi\big(R_\alpha^j(x+\beta_0)\big) = S_i(x+\beta_0),\]
which is true if
\[\frac{\big|S_i(x) - S_i(x+\beta_0)\big|}{S_i(x) + S_i(x+\beta_0)} \underset{n\to+\infty}{\longrightarrow}0.\]

Consider $n\ge N$ such that $v_n> 2$, $2q_n\ge v_N q_N$ and that the two returns closest to 0 of the orbit of $x$ of length $2 q_n$ under $R_\alpha$ are of indices bigger than $v_Nq_N$.

Take $i$ such that $v_n q_n \le i < v_{n+1} q_{n+1}$ (note that in this case, $n=\tilde n(i)$).

Recall that $\beta_0$ can be written as the sum $\beta_0 = \ell\alpha +\beta_n$ with $0\le \ell \le q_n+q_{n-1}\le 2 q_n$  and $|\beta_n| < \lambda^{(n)}$ (note that in this case, $\ell=\ell_{n-1}$). Hence, 
\begin{equation}\label{EqDivBeta}
S_i(x) - S_i(x+\beta_0) = \big(S_i(x) - S_i(x+\ell\alpha)\big) + \big(S_i(x+\ell\alpha) - S_i(x+\ell\alpha+\beta_n)\big).
\end{equation}
We will treat each element of this sum separately.

\begin{lemma}\label{LemSecondTerm9}
Under the hypotheses of Proposition~\ref{LemFinalPhysmeas}, 
\[\max_{v_n q_n \le i < v_{n+1} q_{n+1}}
\frac{\big|S_i(x+\ell_{n-1}\alpha) - S_i(x+\ell_{n-1}\alpha+\beta_n)\big|}{S_i(x+\beta_0)} \underset{n\to +\infty}{\longrightarrow}0.\]
\end{lemma}

\begin{lemma}\label{LemFirstTerm9}
Under the hypotheses of Proposition~\ref{LemFinalPhysmeas}, 
\[\max_{v_n q_n \le i < v_{n+1} q_{n+1}} 
\frac{\big|S_i(x) - S_i(x+\ell_{n-1}\alpha)\big|}{S_i(x) + S_i(x+\beta_0)} \underset{n\to+\infty}{\longrightarrow}0\]
\end{lemma}

These two lemma prove Proposition~\ref{LemFinalPhysmeas}.
\end{proof}

\begin{proof}[Proof of Lemma~\ref{LemSecondTerm9}]
In this proof, we suppose that $n$ is odd, the even case being identical. Denote 
\[\ell=\ell_{n-1},\qquad \overline{x} = x+\ell\alpha \qquad \text{and} \qquad k = \lfloor i/q_n\rfloor,\]
so that $k=\tilde k(i)$. Note that $\overline{x} + \beta_n = x+\beta_0$. To begin with, we bound the sum using $\psi_1$ and $\psi_2$ (defined in \eqref{EqDefPsi}):
\begin{equation}\label{Cut12}
\big|S_i(\overline{x}) - S_i(\overline{x}+\beta_n)\big| \le 
\bigg|\sum_{j=0}^{i-1} \psi_1(\overline{x}) - \psi_1(\overline{x}+\beta_n)\bigg|
+\bigg|\sum_{j=0}^{i-1} \psi_2(\overline{x}) - \psi_2(\overline{x}+\beta_n)\bigg|.
\end{equation}

We treat the first term of \eqref{Cut12}; we will indicate the changes for the last term when needed. For $0 \le r \le k$, we denote by $j_r$ (resp. $j'_r$) the time corresponding to the closest return to 0 of the orbit of $R_\alpha^{r q_n}(\overline x)$ (resp. $R_\alpha^{r q_n}(\overline x+\beta_n)$) of length $q_n$ in the fundamental domain $[0,1)$. As the orbit of length $i$ is made of at most $k+1$ such pieces of orbits of length $q_n$, by Lemma~\ref{LemFinalMartin},
\begin{align}\label{EqFirstDecomposition}
\bigg|\sum_{j=0}^{i-1} \psi_1(\overline{x}) - \psi_1(\overline{x}+&\beta_n)\bigg|\le 
 \sum_{r=0}^k\left|\psi_1\big(R_\alpha^{j_r}(\overline{x})\big) - \psi_1\big(R_\alpha^{j_r}(\overline{x}+\beta_n)\big)\right| \\
& + \sum_{r=0}^k\left|\psi_1\big(R_\alpha^{j'_r}(\overline{x})\big) - \psi_1\big(R_\alpha^{j'_r}(\overline{x}+\beta_n)\big)\right| + (k+1) \frac{\lambda^{(n)}q_n}{{\lambda^{(n-1)}}} \nonumber.
\end{align}

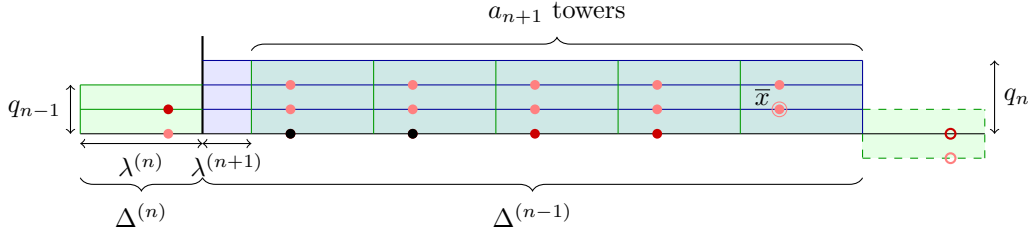
\begin{figure}
\begin{tikzpicture}[scale=.65]
\fill[fill=green, opacity=.1] (0,0) rectangle (-2.5,1);
\fill[fill=green, opacity=.1] (13.5,-.5) rectangle (16,.5);
\fill[fill=blue, opacity=.1] (0,0) rectangle (13.5,1.5);
\fill[fill=green, opacity=.1] (1,0) rectangle (13.5,1.5);
\draw[thick] (0,0) -- (0,2);
\draw (-2.5,0) -- (16,0);
\draw[color=blue!60!black] (0,.5) -- (13.5,.5);
\draw[color=blue!60!black] (0,1) -- (13.5,1);
\draw[color=blue!60!black] (0,1.5) -- (13.5,1.5);
\draw[color=blue!60!black] (13.5,0) -- (13.5,1.5);

\draw[color=green!60!black] (-2.5,.5) -- (0,.5);
\draw[color=green!60!black] (-2.5,1) -- (0,1);
\draw[color=green!60!black] (-2.5,0) -- (-2.5,1);
\draw[color=green!60!black] (1,0) -- (1,1.5);
\draw[color=green!60!black] (3.5,0) -- (3.5,1.5);
\draw[color=green!60!black] (6,0) -- (6,1.5);
\draw[color=green!60!black] (8.5,0) -- (8.5,1.5);
\draw[color=green!60!black] (11,0) -- (11,1.5);
\draw[color=green!60!black, dashed] (13.5,.5) -- (16,.5);
\draw[color=green!60!black, dashed] (13.5,-.5) -- (16,-.5);
\draw[color=green!60!black, dashed] (16,-.5) -- (16,.5);
\draw[color=green!60!black, dashed] (13.5,0) -- (13.5,-.5);

\foreach \x in {1,...,4}
{\fill[color=red!80!black] (2.5*\x-.7,0) circle (.1);
\fill[color=red!50!white] (2.5*\x-.7,.5) circle (.1);
\fill[color=red!50!white] (2.5*\x-.7,1) circle (.1);}

\fill[color=black] (2.5-.7,0) circle (.1);
\fill[color=black] (5-.7,0) circle (.1);

\fill[color=red!50!white] (2.5*5-.7,1) circle (.1);
\fill[color=red!50!white] (2.5*5-.7,.5) circle (.1);
\draw[color=red!50!white] (2.5*5-.7,.5) circle(.15);
\fill[color=black] (2.5*5-.7,.4) node[above left]{$\overline x$};

\fill[color=red!80!black] (-.7,.5) circle (.1);
\fill[color=red!50!white] (-.7,0) circle (.1);

\draw[color=red!80!black, thick] (15.3,0) circle (.1);
\draw[color=red!50!white, thick] (15.3,-.5) circle (.1);


\draw [decorate,decoration={brace,amplitude=5pt},xshift=0,yshift=-.9cm]
(13.5,0) -- (0,0) node [black,midway,yshift=-0.5cm] {$\Delta^{(n-1)}$};
\draw [decorate,decoration={brace,amplitude=5pt},xshift=0,yshift=-.9cm]
(0,0) -- (-2.5,0) node [black,midway,yshift=-0.5cm] {$\Delta^{(n)}$};

\draw [decorate,decoration={brace,amplitude=5pt},xshift=0,yshift=.2cm]
(1,1.5) -- (13.5,1.5) node [black,midway,yshift=0.5cm] {$a_{n+1}$ towers};

\draw[<->] (16.2,0) --node[midway, right]{$q_n$} (16.2,1.5);
\draw[<->] (-2.7,0) --node[midway, left]{$q_{n-1}$} (-2.7,1);
\draw[<->] (-2.5,-.2) --node[midway, below]{$\lambda^{(n)}$} (0,-.2);
\draw[<->] (1,-.2) --node[midway, below]{$\lambda^{(n+1)}$} (0,-.2);

\end{tikzpicture}
\caption{Piece of positive orbit of the point $\overline x$ (of length \emph{a priori} smaller than $i$). Each set of $q_n-1$ consecutive light red points in a column contributes to the last term of \eqref{EqFirstDecomposition}. The other points (red and black) contribute to the first terms of \eqref{EqFirstDecomposition}, i.e. the points $R_\alpha^{j_r}(\overline x)$. Among the $a_{n+1}$ first of them, at most two (in black) belong to $[0,2\lambda^{(n)}]$ (first term of the second line of \eqref{multfirst}) The other ones (in red) give birth to the second term of the second line of \eqref{multfirst}. In this example, the interval on the right of $\Delta^{(n-1)}$ is some $\Delta^{(n)}_j$ but it also could be some $\Delta^{(n-1)}_j$. The green dashed shifted small tower on the right is the same as the small leftmost tower.
}\label{FigRenor666}
\end{figure}

Let us first bound the last term of \eqref{EqFirstDecomposition}, which corresponds to the pink points of Figure~\ref{FigRenor666}. Recall that by Corollary~\ref{kq orbit}, one has
\begin{equation}\label{EqTheLast}
S_i(x+\beta_0) \ge k\frac{\log q_n}{2\lambda^{(n-1)}}.
\end{equation}
By Lemma~\ref{properties}, this gives
\begin{equation}\label{EqFirstIziTerm}
(k+1)\frac{\lambda^{(n)}q_n}{{\lambda^{(n-1)}} S_i(x+\beta_0)}
\le (k+1)\frac{2q_n/q_{n+1}}{k\log q_n}
\le \frac{4}{a_{n+1}\log q_n}.
\end{equation}

For the first term of \eqref{EqFirstDecomposition} (the ground floor in Figure~\ref{FigRenor666}, i.e. the black and red points), we will use the fact that as $x\notin (\tilde D_n-\ell_{n-1}\alpha) \cup (\tilde D_n-\beta_0)$, $\overline x\notin \tilde D_n$ and  $\overline{x}+\beta_n \notin \tilde D_n$. So it is possible to apply Lemma~\ref{LemConvDn3}: as $\|\overline x-(\overline x+\beta_n)\| = \|\beta_n\| < \lambda^{(n)}$, for any $r \le k$, one has
\begin{equation}\label{EqBoundCloose}
\psi\big(R_\alpha^{j_r}(\overline{x})\big) + \psi\big(R_\alpha^{j_r}(\overline{x}+\beta_n)\big) \le \frac{16}{u_n} \sum_{j=0}^{i-1} \psi\big(R_\alpha^j(\overline x+\beta_n)\big).
\end{equation}

We treat two cases separately.

\noindent {\bfseries Case 1:} $k\ge a_{n+1}$.\\
In this case, one cuts the sum into pieces of length $a_{n+1}$:
\begin{multline*}
\sum_{r=0}^k\left|\psi_1\big(R_\alpha^{j_r}(\overline{x})\big) - \psi_1\big(R_\alpha^{j_r}(\overline{x}+\beta_n)\big)\right| \\
 \le \sum_{m=0}^{\lceil k/a_{n+1}\rceil}\sum_{p=0}^{a_{n+1}-1}
 \left|\psi_1\big(R_\alpha^{j_{p+a_{n+1}m}}(\overline{x})\big) - \psi_1\big(R_\alpha^{j_{p+a_{n+1}m}}(\overline{x}+\beta_n)\big)\right|.
\end{multline*}

As can be seen in Figure~\ref{FigRenor666}, and because of the form of the rotation map $R_\alpha$ in the renormalization tower, among $a_{n+1}$ consecutive terms $R_\alpha^{j_r}(\overline x)$, at most one of them does not belong to $\Delta^{(n-1)}$. There are between $a_{n+1}-1$ and $a_{n+1}$ remaining terms, which belong to $\Delta^{(n-1)}$. These remaining terms are made of at most two pieces of orbit of the rotation of angle $\rho_n=q_n\alpha-p_n$, and if there are two of them all the terms of one piece are on the left to all the terms of the other (for the order on the segment $\Delta^{(n-1)}$).

One can isolate the points $R_\alpha^{j_{p+a_{n+1}m}}(\overline{x})$ that belong to $[0,2\lambda^{(n)}]$ (there are at most two of them, in black in Figure~\ref{FigRenor666}) -- which will give the first term of \eqref{multfirst} -- from the others, in red in Figure~\ref{FigRenor666} -- which will give the second term of \eqref{multfirst}. Reasoning as in the proof of Lemma~\ref{LemFinalMartin}, one gets
\begin{multline}\label{multfirst}
\sum_{p=0}^{a_{n+1}-1}\left|\psi_1\big(R_\alpha^{j_{p+a_{n+1}m}}(\overline{x})\big) - \psi_1\big(R_\alpha^{j_{p+a_{n+1}m}}(\overline{x}+\beta_n)\big)\right|\\
\le \frac{32}{u_n} \sum_{j=0}^{i-1} \psi_1\big(R_\alpha^j(\overline x+\beta_n)\big) + \sum_{p=1}^{a_{n+1}} \left(\frac{1}{p\lambda^{(n)}} - \frac{1}{(p+1)\lambda^{(n)}}\right)\\
\le \frac{32}{u_n} S_i(x+\beta_0) + \frac{1}{\lambda^{(n)}}.
\end{multline}
But by \eqref{EqTildeK} and \eqref{EqTotTime} one has (for $n$ large enough)
\[\lceil k/a_{n+1} \rceil \le \frac{a_{n+1}+1}{a_{n+1}} v_{n+1} + 1 \le 3v_{n+1},\]
and also, because $k\ge a_{n+1}$, one has $\lceil k/a_{n+1} \rceil \le \frac{2k}{a_{n+1}}$. These two bounds lead to
\[\sum_{r=0}^k\left|\psi_1\big(R_\alpha^{j_r}(\overline{x})\big) - \psi_1\big(R_\alpha^{j_r}(\overline{x}+\beta_n)\big)\right|
 \le \frac{96}{u_n} v_{n+1}\, S_i(x+\beta_0) + \frac{2 k}{a_{n+1}\lambda^{(n)}}.\]
Using Equation \eqref{EqTheLast} together with \eqref{Eqaeta}, this gives
\begin{equation*}
\frac{\sum_{r=0}^k\left|\psi_1\big(R_\alpha^{j_r}(\overline{x})\big) - \psi_1\big(R_\alpha^{j_r}(\overline{x}+\beta_n)\big)\right|}{S_i(x+\beta_0)}
 \le \frac{96}{u_n} v_{n+1} + \frac{8}{\log q_n}.
\end{equation*}
Combined with \eqref{eqpropvn}, this implies that
\begin{equation}\label{multfirst'}
\frac{\sum_{r=0}^k\left|\psi_1\big(R_\alpha^{j_r}(\overline{x})\big) - \psi_1\big(R_\alpha^{j_r}(\overline{x}+\beta_n)\big)\right|}{S_i(x+\beta_0)}
\underset{n\to +\infty}{\longrightarrow}0 .
\end{equation}

The same proof works for $\psi_2$ instead of $\psi_1$, and also for $j'_r$ instead of $j_r$ (second term of \eqref{EqFirstDecomposition}).
\medskip

\noindent {\bfseries Case 2:} $k < a_{n+1}$.\\
As in the first case, and as in the proof of Lemma~\ref{LemFinalMartin}, one can isolate the points $R_\alpha^{j_{p+a_{n+1}}}(\overline{x})$ that belong to $[0,2\lambda^{(n)}]$ (there are at most two of them, in black in Figure~\ref{FigRenor666}) -- which will give the first term of \eqref{multsecond} -- from the others, in red in Figure~\ref{FigRenor666} -- which will give the second term of \eqref{multsecond}. For this second family of points, let us call $m_0$ the index of the first ground floor interval of length $\lambda^{(n)}$ containing one of these points; in other words the is no point $R_\alpha^{j_{p+a_{n+1}}}(\overline{x})$ of this family in $[0,\lambda^{(n+1)}+2\lambda^{(n)})$ and one in $[\lambda^{(n+1)}+2\lambda^{(n)},\lambda^{(n+1)}+3\lambda^{(n)}]$. In this case, one gets 
\begin{multline}\label{multsecond}
\sum_{r=0}^k\left|\psi_1\big(R_\alpha^{j_r}(\overline{x})\big) - \psi_1\big(R_\alpha^{j_r}(\overline{x}+\beta_n)\big)\right|\\
\le \frac{32}{u_n} \sum_{j=0}^{i-1} \psi_1\big(R_\alpha^j(\overline x+\beta_n)\big) + \sum_{p=m_0}^{k} \left(\frac{1}{p\lambda^{(n)}} - \frac{1}{(p+1)\lambda^{(n)}}\right)\\
\le \frac{32}{u_n} S_i(x+\beta_0) + \frac{1}{m_0\lambda^{(n)}}.
\end{multline}
On the other hand, by the same kind of reasoning, using $k\ge v_n$ and Lemma~\ref{LemSerHarmo},
\begin{align}
\nonumber S_i(x+\beta_0) & \ge \sum_{r=0}^{k-1}\psi_1\big(R_\alpha^{j_r}(\overline x + \beta_n)\big)\\ 
& \ge \sum_{p=m_0}^{m_0+v_n-3} \frac{1}{p\lambda^{(n)}} \ge \frac{1}{\lambda^{(n)}}\log\left(1+\frac{v_n-3}{m_0}\right).\label{EqMinorSiSpec}
\end{align}
We will also use the following fact, easily coming from the convexity of $\log$: For any $m_0\ge 1$,
\begin{equation*}
m_0\log\left(1+\frac{v}{m_0}\right) \ge \log\big(1+v\big).
\end{equation*}
Applying this to \eqref{EqMinorSiSpec} and \eqref{multsecond}, one gets
\begin{equation}\label{multsecond'}
\frac{\sum_{r=0}^k\left|\psi_1\big(R_\alpha^{j_r}(\overline{x})\big) - \psi_1\big(R_\alpha^{j_r}(\overline{x}+\beta_n)\big)\right|}{S_i(x+\beta_0)}
\le \frac{32}{u_n}  + \frac{1}{\log\big(v_n-2\big)}.
\end{equation}
The same holds for $j'_r$ instead of $j_r$.
\medskip

For this second case $k<a_{n+1}$, we also need to treat the case of $\psi_2$. The reader should refer to Figure~\ref{FigRenor6666}. Now, for $0 \le r \le k$, we denote by $j_r$ (resp. $j'_r$) the time corresponding to the closest return to 0 of the orbit of $R_\alpha^{r q_n}(\overline x)$ (resp. $R_\alpha^{r q_n}(\overline x+\beta_n)$) of length $q_n$ in the fundamental domain $(-1,0]$ (which is adapted to the map $\psi_2$).

This time, we consider the interval $J$ made of the union of $\Delta^{(n)}$ with the interval of $\xi^{(n)}$ on its left, denoted by $\Delta_\iota^{(n-1)}$ (by the properties of the renormalization procedure, we know that this interval is some $\Delta_j^{(n-1)}$ and not some $\Delta_j^{(n)}$). As for the renormalization interval $\Delta(n)$, the first return map in restriction to $J$ is the rotation of angle $\rho_n = q_n\alpha-p_n$, and the return time is always smaller than $q_n$. This implies that the points $R_\alpha^{j_r}(\overline x)$ form an orbit segment for $R_{\rho_n}$ of length $k<a_{n+1}$ -- in particular, quotienting $J$ by its endpoints to get a circle, the order of the points on this circle corresponds to the order on their indices $r$.

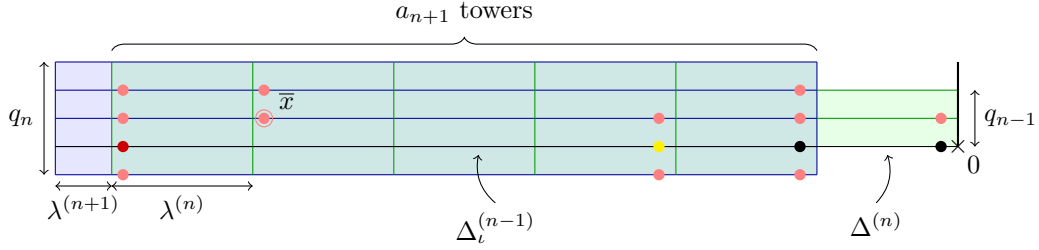
\begin{figure}
\begin{tikzpicture}[scale=.75]
\fill[fill=green, opacity=.1] (13.5,.5) rectangle (16,1.5);
\fill[fill=blue, opacity=.1] (0,0) rectangle (13.5,2);
\fill[fill=green, opacity=.1] (1,0) rectangle (13.5,2);

\draw[color=blue!60!black] (0,.5) -- (13.5,.5);
\draw[color=blue!60!black] (0,0) -- (13.5,0);
\draw[color=blue!60!black] (0,1) -- (13.5,1);
\draw[color=blue!60!black] (0,1.5) -- (13.5,1.5);
\draw[color=blue!60!black] (0,2) -- (13.5,2);
\draw[color=blue!60!black] (13.5,0) -- (13.5,2);
\draw[color=blue!60!black] (0,0) -- (0,2);

\draw[color=green!60!black] (13.5,1.5) -- (16,1.5);
\draw[color=green!60!black] (13.5,1) -- (16,1);
\draw[color=green!60!black] (1,0) -- (1,2);
\draw[color=green!60!black] (3.5,0) -- (3.5,2);
\draw[color=green!60!black] (6,0) -- (6,2);
\draw[color=green!60!black] (8.5,0) -- (8.5,2);
\draw[color=green!60!black] (11,0) -- (11,2);

\draw[thick] (16,.5) -- (16,2);
\draw (0,.5) -- (16,.5);
\draw (16,.5) node{$\times$} node[below right]{$0$};

\fill[color=red!50!white] (3.7,1) circle (.1);
\draw[color=red!50!white] (3.7,1) circle (.15);
\fill[color=black] (3.8,1) node[above right]{$\overline x$};
\fill[color=red!50!white] (3.7,1.5) circle (.1);
\fill[color=red!50!white] (1.2,0) circle (.1);
\fill[color=red!80!black] (1.2,.5) circle (.1);
\fill[color=red!50!white] (1.2,1) circle (.1);
\fill[color=red!50!white] (1.2,1.5) circle (.1);
\fill[color=black] (15.7,.5) circle (.1);
\fill[color=red!50!white] (15.7,1) circle (.1);
\fill[color=red!50!white] (13.2,0) circle (.1);
\fill[color=black] (13.2,.5) circle (.1);
\fill[color=red!50!white] (13.2,1) circle (.1);
\fill[color=red!50!white] (13.2,1.5) circle (.1);
\fill[color=red!50!white] (10.7,0) circle (.1);
\fill[color=yellow] (10.7,.5) circle (.1);
\fill[color=red!50!white] (10.7,1) circle (.1);

\draw[<-] (14.75,.35) to[bend left] (14.55,-.5) node[below] {$\Delta^{(n)}$};
\draw[<-] (7.5,.35) to[bend right] (7.8,-.5) node[below] {$\Delta_\iota^{(n-1)}$};

\draw [decorate,decoration={brace,amplitude=5pt},xshift=0,yshift=.2cm]
(1,2) -- (13.5,2) node [black,midway,yshift=0.5cm] {$a_{n+1}$ towers};

\draw[<->] (-.2,0) --node[midway, left]{$q_n$} (-.2,2);
\draw[<->] (16.3,.5) --node[midway, right]{$q_{n-1}$} (16.3,1.5);
\draw[<->] (1,-.2) --node[midway, below]{$\lambda^{(n)}$} (3.5,-.2);
\draw[<->] (1,-.2) --node[midway, below]{$\lambda^{(n+1)}$} (0,-.2);

\end{tikzpicture}
\caption{Adaptation of Figure~\ref{FigRenor666} for $\psi_2$: we consider the large interval $\Delta_\iota^{(n-1)}$ on the left of $\Delta^{(n)}$ (the interval in black inside the large tower), which lies inside the tower above $\Delta^{(n-1)}$ (the bottom interval of the left tower). Among the positive orbit of $\overline x$, there are four types of points: the pink ones that give contribution to the last term of \eqref{EqFirstDecomposition}, the red, the yellow and the black ones. There are at most $v_{n+1}$ black ones; the red and the yellow ones are treated as in the proof for $\psi_1$.
}\label{FigRenor6666}
\end{figure}

Among these points, we isolate the $v_n$ ones that are the closest to 0 (in black in Figure~\ref{FigRenor6666}); we denote $B$ the set of indices of these points. It separates the remaining points $R_\alpha^{j_r}(\overline x)$ into (at most) two orbit segments of $R_{\rho_n}$ in $J$, in red and yellow in Figure~\ref{FigRenor6666}. Let us denote $R$ and $Y$ the sets of indices of these orbit segments.

For the black points, using \eqref{EqBoundCloose}, one has
\[\sum_{r\in B}\left|\psi_2\big(R_\alpha^{j_r}(\overline{x})\big) - \psi_2\big(R_\alpha^{j_r}(\overline{x}+\beta_n)\big)\right| \le \frac{16 v_n}{u_n}S_i(\overline x+\beta_n).\]
For the red points, if the set $R$ has cardinality smaller than $v_n$, the same estimate holds. If not, then the proof strategy of \eqref{multsecond} and \eqref{EqMinorSiSpec} works identically. The same is true for the yellow points. Finally, this proves that
\[\frac{\sum_{r=0}^k\left|\psi_2\big(R_\alpha^{j_r}(\overline{x})\big) - \psi_2\big(R_\alpha^{j_r}(\overline{x}+\beta_n)\big)\right| }{S_i(\overline x+\beta_n)}  \underset{n\to +\infty}{\longrightarrow}0.\]
\medskip

Combining it with \eqref{EqFirstDecomposition} with \eqref{EqFirstIziTerm}, \eqref{multfirst'} and \eqref{multsecond'}, we deduce that
\[\frac{\big|S_i(\overline{x}) - S_i(\overline{x}+\beta_n)\big|}{S_i(x+\beta_0)} \underset{n\to +\infty}{\longrightarrow}0.\]
\end{proof}

\begin{proof}[Proof of Lemma~\ref{LemFirstTerm9}]
We now treat the first term of \eqref{EqDivBeta}. Remark that the number $\ell=\ell_n$ depends on $n$, so results of Section \ref{SecPhysSame} do not apply directly.

So one wants to compare $|S_i(x) - S_i(x+\ell\alpha)|$ with $S_i(x)+S_i(x+\beta_0)$. Let us compute (using $\ell < 2q_n$ and $i\ge v_n q_n$ with $v_n\ge 4$):
\begin{align*}
S_i(x) - S_i(x+\ell\alpha) & = \sum_{j=0}^{i-1} \left(\psi(R_\alpha^j(x)) - \psi(R_\alpha^j(x+\ell\alpha))\right)\\
 & = \sum_{j=0}^{i-1} \left(\psi(R_\alpha^j(x)) - \psi(R_\alpha^{j+\ell}(x))\right)\\
  & = \sum_{j=0}^{\ell-1} \psi(R_\alpha^j(x)) - \sum_{j=i}^{i+\ell-1} \psi(R_\alpha^j(x)).
\end{align*}
Hence, 
\[\big|S_i(x) - S_i(x+\ell\alpha)\big| \le S_{\ell}(x) + S_{\ell}\big(R_\alpha^{i}(x)\big).\]

By Lemma~\ref{EqSellFinal}, one has (using $\ell\le 2q_n$)
\[S_{\ell}(x) \le \frac{8\log q_n}{\lambda^{(n-1)}} + \psi(y_0) + \psi(y_1),\]
where $y_0$ and $y_1$ are the two closest returns to $0$ of the orbit of $x$ of length $2q_n$, and similarly
\[S_{\ell}\big(R_\alpha^{i}(x)\big) \le \frac{8\log q_n}{\lambda^{(n-1)}} + \psi(y'_0) + \psi(y'_1),\]
with $y_0'$ and $y_1'$ the two closest returns to $0$ of the orbit of $R_\alpha^{i}(x)$ of length\footnote{Note that here the length of the orbit in which we choose the closest points is smaller, but the proof of the Lemma~\ref{EqSellFinal} works identically in this case.	} $\ell$.

Recall that by the choice of $n\ge N$ large enough, denoting $y_0 = x+m_0\alpha$ and $y_1 = x+m_1\alpha$, one has $v_Nq_N \le m_0,m_1 \le \ell <i$. So $\tilde n(m_0), \tilde n(m_1) \ge N$. As $x\notin \tilde D_{\tilde n(m_0)} \cup \tilde D_{\tilde n(m_1)}$, by Lemma~\ref{LemConvDn3} (and the fact that $(u_n)$ is increasing),
\[\psi(y_0) \le \frac{8}{u_{\tilde n(m_0)}} \sum_{j=0}^{m_0-1} \psi\big(R_\alpha^j(x)\big) \le \frac{8}{u_n} \sum_{j=0}^{i-1} \psi\big(R_\alpha^j(x)\big),\]
and the same for $y_1$.

We now treat the points $y_0'$ and $y_1'$. Note that they are the two returns closest to 0 of the orbit of $R_\alpha^{i-\ell}(\overline x)$ of length $\ell$ (recall that $\overline{x}=x+\ell\alpha$). Let us denote $y_0'=\overline x+m'_0\alpha$ and $y_1'=\overline x+m'_1\alpha$, so that $v_Nq_N \le 2 q_n \le m'_0,m'_1 < i$. As before, as $\overline x\notin \tilde D_{\tilde n(m_0)} \cup \tilde D_{\tilde n(m_1)}$, by Lemma~\ref{LemConvDn3},
\[\psi(y'_0) \le \frac{8}{u_{\tilde n(m'_0)}} \sum_{j=0}^{m'_0-1} \psi\big(R_\alpha^j(\overline x)\big) \le \frac{8}{u_n} \sum_{j=0}^{i-1} \psi\big(R_\alpha^j(\overline x)\big),\]
and the same for $y'_1$.

Using Lemma~\ref{LemSecondTerm9} (which tells that $S_i(\overline x)$ is more or less equal to $S_i(x+\beta_0)$), we deduce that for all $n$ large enough,
\[\psi(y_0'), \psi(y_1') \le \frac{16}{u_n} S_i(x+\beta_0).\]

Putting this bound together with \eqref{EqTheLast}, one gets (using $k\ge v_n$, by \eqref{EqTildeK})
\begin{align*}
\frac{\left|S_i(x) - S_i(x+\ell\alpha)\right|}{S_i(x) + S_i(x+\beta_0)}
& \le  4\frac{8\log q_n}{\lambda^{(n-1)}}\frac{2\lambda^{(n-1)}}{k\log q_n} + \frac{48}{u_n}\\
& \le  \frac{64}{k} + \frac{48}{u_n}\\
& \le  \frac{64}{v_n} + \frac{48}{u_n}.
\end{align*}
Note that this is in this part we use the $v_n$ factor introduced specifically for this proof.
\medskip

Putting all these estimates together, and using \eqref{eqpropvn}, one gets 
\[\frac{\big|S_i(x) - S_i(x+\beta_0)\big|}{S_i(x) + S_i(x+\beta_0)} \underset{n\to+\infty}{\longrightarrow}0.\]
\end{proof}

\bibliographystyle{plain}

\bibliography{flow}

\end{document}